\documentclass[a4paper,10pt]{article}
\usepackage[english]{babel}
\usepackage[latin1]{inputenc}
\usepackage[T1]{fontenc}
\usepackage{textcomp} 
\usepackage{amsmath}
\usepackage{amsfonts}
\usepackage{amsthm}
\usepackage{mathrsfs}
\usepackage{amssymb}
\usepackage{enumerate}
\usepackage{csquotes}
\usepackage{bbm}
\usepackage{graphicx}
\usepackage[mathscr]{euscript}
\DeclareSymbolFont{rsfs}{U}{rsfs}{m}{n}
\DeclareSymbolFontAlphabet{\mathscrsfs}{rsfs}

\theoremstyle{definition}
\newtheorem{Def}{Definition}[section]

\newtheorem{Rmk}[Def]{Remark}

\theoremstyle{plain}
\newtheorem{Prop}[Def]{Proposition}
\newtheorem{Thm}[Def]{Theorem}
\newtheorem{Lemma}[Def]{Lemma}
\newtheorem{Cor}[Def]{Corollary}

\numberwithin{equation}{section}

\newcommand{\R}{\mathbb{R}}
\newcommand{\E}{\mathbb{E}}

\newcommand{\Z}{\mathbb{Z}}
\newcommand{\N}{\mathbb{N}}

\newcommand{\calD}{\mathscrsfs{D}}

\def\card{{\cal C}ard}

\newcommand{\footremember}[2]{%
   \footnote{#2}
    \newcounter{#1}
    \setcounter{#1}{\value{footnote}}%
}

\makeatletter

\def\oversortoftilde#1{\mathop{\vbox{\m@th\ialign{##\crcr\noalign{\kern3\p@}%
      \sortoftildefill\crcr\noalign{\kern3\p@\nointerlineskip}%
      $\hfil\displaystyle{#1}\hfil$\crcr}}}\limits}

\def\sortoftildefill{$\m@th \setbox\z@\hbox{$\braceld$}%
  \braceld\leaders\vrule \@height\ht\z@ \@depth\z@\hfill\braceru$}

\makeatother

\title{Wavelet-Type Expansion of Generalized Hermite Processes with rate of convergence}
\author{A Ayache \footremember{Lille}{Univ. Lille, 
CNRS, UMR 8524 - Laboratoire Paul Painlev\'e, F-59000 Lille, France. antoine.ayache@univ-lille.fr}, J. Hamonier \footnote{Univ. Lille, CHU Lille, ULR 2694 - METRICS : \'Evaluation des technologies de sant\'e et des
pratiques m\'edicales, F-59000 Lille, France. julien.hamonier@univ-lille.fr}, L. Loosveldt\footnote{\underline{\textbf{Corresponding author}} University of Luxembourg, Department of Mathematics (DMATH), Maison du Nombre, 6, avenue de la Fonte, L-4364 Esch-sur-Alzette, Grand Duchy Of Luxembourg. laurent.loosveldt@uni.lu}}

\begin{document}

\maketitle

\begin{abstract}
Wavelet-type random series representations of the well-known Fractional Brownian Motion (FBM) and many other related stochastic processes and fields have started to be introduced since more than two decades. Such representations provide natural frameworks for approximating almost surely and uniformly rough sample paths at different scales and for study of various aspects of their complex erratic behavior. 

Hermite process of an arbitrary integer order $d$, which extends FBM, is a paradigmatic example of a stochastic process belonging to the $d$th Wiener chaos. It was introduced very long time ago, yet many of its properties are still unknown when $d\ge 3$. In a paper published in 2004, Pipiras raised the problem to know whether wavelet-type random series representations with a well-localized smooth scaling function, reminiscent to those for FBM due to Meyer, Sellan and Taqqu, can be obtained for a Hermite process of any order $d$. He solved it in this same paper in the particular case $d=2$ in which the Hermite process is called the Rosenblatt process. Yet, the problem remains unsolved in the general case $d\ge 3$. The main goal of our article is to solve it, not only for usual Hermite processes but also for generalizations of them. Another important goal of our article is to derive almost sure uniform estimates of the errors related with approximations of such processes by scaling functions parts of their wavelet-type random series representations.
\end{abstract}


\section{Introduction and background} 
\label{sec:intro}

Fractional Brownian Motion (FBM) with Hurst parameter $h \in (0,1)$, denoted $\{B_h(t) \}_{t \in \R}$, was introduced by Kolmogorov, in 1940, to generate Gaussian ``spirals'' in Hilbert spaces \cite{MR0003441}. Its first systematic study was carried out in the famous paper \cite{MR242239} by Mandelbrot and Van Ness, in 1968. It is the unique Gaussian process with $B_h(0)=0$, mean zero and covariance function
\[ \mathbb{E}[B_h(t) B_h(s)]=\frac{c_h}{2} \left( |t|^{2h}+|s|^{2h}-|t-s|^{2h} \right),\quad\mbox{for all $(t,s)\in\R^2$,}\]
where $c_h$ is a positive constant only depending on the Hurst parameter $h$.
Among its most fundamental properties, FBM has stationary increments and is $h$-self-similar, meaning that, for all fixed $a>0$, the processes $\{a^{-h} B_h(at)\}_{t \in \R}$ and $\{B_h(t) \}_{t \in \R}$ have the same finite-dimensional distributions. When $h=1/2$, the process $\{B_{1/2}(t) \}_{t \in \R}$ is a usual Brownian motion. We refer for instance to the monograph \cite{MR3076266} for a clear and concise presentation of various fundamental facts concerning FBM.

FBM appears naturally in many real-life applications in various domains, such as telecommunications, biology, finance, image processing, and so on. We refer for instance to \cite{MR1956041} for a monograph with an overview of its different \linebreak areas of applications. Thus, study of FBM and related processes has become a crucial issue since a long time. To this end, it is very useful to construct well appropriate representations for these processes. An important class of such representations consists in wavelet-type random series representations. More than two decades ago, they were introduced in the framework of FBM in several articles. We focus on the Meyer, Sellan and Taqqu seminal article \cite{MR1755100} whose main goal was to obtain representations which clearly separate the low frequency part of FBM from its high frequency part, and, more importantly, to express the low frequency part in terms of a well-localized smooth scaling function. For a better understanding of our paper, we believe it useful to precisely present in our introduction the most classical one of these wavelet-type representations of FBM due to \cite{MR1755100}, since one of our principle aims is to extend it to Generalized Hermite process. The article \cite{MR1755100} made use of the well-known class of the Meyer orthonormal wavelet bases of $L^2 (\R)$ as the main ingredient for its constructions of wavelet-type random series representations for FBM. Some fundamental properties of the two functions (scaling function and mother wavelet) generating such a basis are the following:
\begin{Rmk}
\label{rem:mey-bas}
Univariate scaling function and mother wavelet $\phi$ and $\psi$ associated with a Meyer orthonormal wavelet basis of $L^2(\R)$ belong to the Schwartz class $\mathcal{S}(\R)$ of infinitely differentiable functions whose derivatives of any order rapidly decay at infinity. Moreover, the Fourier transforms $\widehat{\phi}$ and $\widehat{\psi}$ are infinitely differentiable compactly supported functions satisfying
\[
{\rm supp}\,\widehat{\phi}\subseteq \left [-\frac{4\pi}{3},\frac{4\pi}{3}\right]\quad\mbox{and}\quad {\rm supp}\,\widehat{\psi}\subseteq \left [-\frac{8\pi}{3},\frac{8\pi}{3}\right]\setminus \left(-\frac{2\pi}{3},\frac{2\pi}{3}\right).
\]
Notice that throughout our article, we use the rather common convention that ${\cal F}(f)=\widehat{f}$, the Fourier transform of an arbitrary function $f\in\mathcal{S}(\R)$ is defined, for all $\xi\in\R$, as ${\cal F}(f)(\xi)=\widehat{f}(\xi):=(2\pi)^{-1/2}\int_{\R} e^{-i\xi x}f(x)\,dx$, while ${\cal F}^{-1}(f)$, the inverse Fourier transform of $f$, is defined, for every $x\in\R$, as ${\cal F}^{-1}(f)(x):=(2\pi)^{-1/2}\int_{\R} e^{i x \xi}f(x)\,d\xi$. 
\end{Rmk}

The article \cite{MR1755100} also made an extensive use of the notion of fractional primitive and derivative, which can be defined as follows:

\begin{Def}\label{def:fractwav}
Let $f$ be an arbitrary function of the Schwartz class $\mathcal{S}(\R)$. For all $h \in (1/2,1)$ (resp. $h\in (0,1/2]$), the fractional primitive of $f$ of order $h-1/2$ (resp. the fractional derivative of $f$ of order $1/2-h$) is the function denoted by $f_h$, which generally speaking belongs to $L^2 (\R)$, and which is defined through its Fourier transform $\widehat{f}_h$ by:
\begin{equation}
\label{def:fractwav:eq1}
\widehat{f}_h(\xi)=(i \xi)^{1/2-h}\widehat{f}(\xi),\quad \text{for almost all $\xi\in\R$.}
\end{equation}
One mentions that, using the common convention that, for all $(y,\alpha) \in \R^2$, when $y>0$ one has $y_+^\alpha=y^\alpha$ and otherwise one has $y_+^\alpha=0$, then, for any $h\in (1/2,1)$, the fractional primitive $f_h$ can be expressed as:
\begin{equation}\label{eqn:fractantideri}
f_h(s) = \frac{1}{\Gamma(h-1/2)} \int_\R (s-x)_+^{h-3/2} f(x) \, dx,\quad\mbox{for all $s\in\R$,}
\end{equation}
where $\Gamma$ is the usual "Gamma" Euler function defined, for all $z \in (0,+\infty)$, as $\Gamma(z):= \int_0^{+\infty} u^{z-1} e^{-u} \, du$. Also, one mentions that, when the Fourier transform $\widehat{f}$ of $f$ vanishes on a neighborhood of $0$ (notice the univariate Meyer mother wavelet $\psi$ satisfies this property), then one can drop the restriction $h\in (0,1)$ and may allow $h$ to be any real number. In the latter case, for all $h\in (1/2,+\infty)$ (resp. for all $h\in (-\infty,1/2]$) the fractional primitive (resp. derivative) $f_h$ can still be defined through its Fourier transform as in \eqref{def:fractwav:eq1}, and the equality \eqref{eqn:fractantideri} for fractional primitive remains valid. Moreover, for every $h\in\R$, one can easily check that $f_h$ belongs to the Schwartz class $\mathcal{S}(\R)$.
\end{Def}

Unfortunately, since for a univariate Meyer scaling function $\phi$ the Fourier transform $\widehat{\phi}$ does not vanish on a neighborhood of $0$, for all $h\in (0,1)$, the fractional primitive or derivative $\phi_h$, of $\phi$, fails to be a smooth well-localized function. In order to overcome this serious difficulty, a clever idea of \cite{MR1755100} was to "replace" $\phi_h$ by the so called fractional scaling function $\Phi_\Delta^{(\delta)}$, which belongs to $\mathcal{S}(\R)$ and which was defined in \cite{MR1755100} as follows:

%Despite this nice property, one can argue that the scaling function $\phi$ is not well-adapted for approximation purposes of FBM. Indeed, such an approximation would involve fractional primitive of $\phi$ that, in fact, does not belong to the Schwartz class, see Section \ref{sect:mainandstrat} for a discussion about this fact and references. For this reason, the authors in \cite{MR1755100} introduce the fractional scaling function.

\begin{Def}\label{def:fractscal}
The \textit{fractional scaling function} of order $\delta \in \R$ of a univariate Meyer scaling function $\phi$ is the function $\Phi_\Delta^{(\delta)}\in\mathcal{S}(\R)$ defined through its Fourier transform by:
\[ \forall \, \xi\in\R\setminus\{0\},\,\,\widehat{\Phi}_\Delta^{(\delta)}(\xi) = \left( \frac{1-e^{-i \xi}}{i \xi} \right)^\delta \widehat{\phi}(\xi) \quad \text{ and } \quad \widehat{\Phi}_\Delta^{(\delta)}(0)=1.\]
Similarly to $\widehat{\phi}$, the function $\widehat{\Phi}_\Delta^{(\delta)}$ has a compact support satisfying 
\[
{\rm supp}\,\widehat{\Phi}_\Delta^{(\delta)}\subseteq \left [-\frac{4\pi}{3},\frac{4\pi}{3}\right]\,.
\]
\end{Def}

\begin{Rmk}
\label{rem:frac-mey}
Let $\delta$ and $h$ be two arbitrary and fixed real numbers. One can check, from elementary properties of the Fourier transform (see e.g. the seminal book \cite{schwartz:78}), that the fractional scaling function $\Phi_\Delta^{(\delta)}$ and the fractional primitive or derivative $\psi_h$, of the univariate Meyer mother wavelet $\psi$, belong to $\mathcal{S}(\R)$, which means that they are infinitely differentiable functions whose derivatives of any order rapidly decay at infinity, in other words one has, for all fixed $m \in \N_0$ and $L\in (0,+\infty)$, 
\begin{equation}\label{maj:psih}
\sup_{x\in\R} \left\lbrace (3+|x|)^L \left (\Big | \frac{d^m}{dx^m}\Phi_\Delta^{(\delta)}(x)\Big|+\Big | \frac{d^m}{dx^m} \psi_h(x)\Big| \right) \right\rbrace < +\infty. 
\end{equation}
\end{Rmk}
 
Apart the fact that $\Phi_\Delta^{(\delta)}$ is a very smooth and very well-localized function, another major advantage in expressing the low frequency part of FBM in terms of it is to draw connections between the latter process and FARIMA random walk time series (i.e. partial sums of FARIMA sequence (see Definition \ref{def:farima} below)), as shown by the following theorem of \cite{MR1755100} which provides the most classical wavelet-type random series representation of FBM clearly separating its low and high frequency parts. 

\begin{Thm}[\textbf{Meyer, Sellan and Taqqu}] For each fixed $J\in\Z$, the FBM $\{ B_h(t) \}_{t \in \R}$ can be expressed as the following random series, which converges almost surely and uniformly in $t$ on each compact interval of $\R$, 
\begin{align}\label{eqn:expFBM1alter}
B_h(t) & = \sum_{k \in \Z} 2^{-Jh} S_{J,k}^{(h)} \left(\Phi_\Delta^{(h+1/2)}(2^Jt-k)-\Phi_\Delta^{(h+1/2)}(-k) \right) \nonumber \\
& \quad  + \sum_{j=J}^{+ \infty} \sum_{ k \in \Z} 2^{-j h} g^\psi_{j,k} \Big( \psi_{h+1}(2^{j} t-k)-\psi_{h+1}(-k) \Big),
\end{align}
where:
\begin{itemize}
\item $(g^\psi_{j,k})_{(j,k)\in \Z^2}$ is the sequence of the i.i.d. $\mathcal{N}(0,1)$ Gaussian random variables defined, for all $(j,k)\in\Z^2$, by
\begin{equation}\label{eqn:n01psi}
g^\psi_{j,k} := 2^{j/2}\int_\R  \psi(2^jx-k) \, dB(x);
\end{equation}
\item given the sequence $(g^\phi_{J,k})_{k\in\Z}$ of the i.i.d. $\mathcal{N}(0,1)$ Gaussian random variables defined, for all $k\in\Z$, by
\begin{equation}\label{eqn:n01phi}
g^\phi_{J,k} := 2^{J/2} \int_\R  \phi(2^J x-k) \, dB(x),
\end{equation}
$(S_{J,k}^{(h)})_{k\in\Z}$ is the Gaussian FARIMA random walk time series defined, for every $k\in\Z$, by
\[ 
S_{J,k}^{(h)}:=\left\{ \begin{array}{cl}
                       \sum_{\ell=1}^k Z_{J,\ell}^{(h-\frac{1}{2})} & \text{if } k > 0 \\ [2ex]
                       0 & \text{if } k = 0\\[2ex]
                       -\sum_{\ell=k+1}^0 Z_{J,\ell}^{(h-\frac{1}{2})}  & \text{if } k <0
                     \end{array} \right.
\]
with $(Z_{J,\ell}^{(h-\frac{1}{2})})_{\ell \in \N}$ the Gaussian FARIMA $(0,h-\frac{1}{2},0)$ sequence associated to $(g^\phi_{J,k})_{k}$, see the next definition.
\end{itemize} 
\end{Thm}

%The final ingredient that we need to introduce in order to present the expansion of FBM given in \cite{MR1755100} are the  so-called FARIMA (autogressive fractionally integrated moving average) sequences.

\begin{Def}\label{def:farima}
Let $(g_k)_{k \in \Z}$ be an arbitrary sequence of i.i.d. centred Gaussian random variables (for instance the sequence $(g^\phi_{J,k})_{k\in\Z}$ in the previous theorem). For each fixed $\delta \in (-1/2,1/2)$, the Gaussian FARIMA $(0,\delta,0)$ sequence associated to $(g_k)_{k \in \Z}$ is denoted by $(Z^{(\delta)}_l)_{l \in \Z}$ and defined, for all $l \in \Z$, as:
\begin{equation}\label{eqn:farima}
 Z_l^{(\delta)} :=\gamma_0^{(\delta)} g_l+\sum_{p=1}^{+ \infty} \gamma_p^{(\delta)} g_{l-p}\,, \quad \text{with $\gamma_0^{(\delta)}:=1$ and } \gamma_p^{(\delta)}:= \frac{\delta \,\Gamma(p+\delta)}{\Gamma(p+1)\Gamma(\delta+1)}.
 \end{equation}
 \end{Def}

\begin{Rmk}
\label{rem1:farima}
Observe that, for the constant $a_{\delta}:=\delta/\Gamma(\delta+1)$, it can be derived from the Stirling's formula that 
 \begin{equation}
 \label{eqn:stirling}
  \gamma_p^{(\delta)}\sim a_{\delta}\, p^{\delta-1},\quad\mbox{when $p$ goes to $+\infty$},
 \end{equation}
 which implies that the random series in \eqref{eqn:farima} is convergent in $L^2(\Omega)$, where $\Omega$ is the underlying probability space. Also notice that the latter series is almost surely convergent as well, thanks to the Kolmogorov's Three-Series theorem. % see \cite[Remark 2.2.13]{theseyassine}. 
\end{Rmk}

%One of the main result of the paper \cite{MR1755100} can then be stated as follows.

\begin{Rmk}
The FBM $\{ B_h(t) \}_{t \in \R}$ can also be expressed as
\begin{equation} \label{eqn:expFBM2}
B_h(t) = \sum_{j \in \Z} \sum_{ k \in \Z} 2^{-j h} g^\psi_{j,k} \Big ( \psi_{h+1}(2^{j} t-k)-\psi_{h+1}(-k) \Big),
\end{equation}
where the series is convergent almost surely and uniformly in $t$ on each compact interval of $\R$. Representations of the type \eqref{eqn:expFBM2} have turned out to be very useful in the study of local and global sample path behavior of various stochastic processes and fields extending FBM. Also it is worth mentioning that, even in the case of the FBM itself, whose sample path behavior was widely studied in the literature prior to wavelet theory, in the very recent article \cite{esserloosveldt} the representation \eqref{eqn:expFBM2} has allowed to show that FBM sample paths have dense subsets of $\R$ of slow points and rapid points. 

However, as explained in \cite{MR1755100,MR1420505,MR2147061}, the representation \eqref{eqn:expFBM1alter} is much more convenient than \eqref{eqn:expFBM2} for approximating the FBM $\{ B_h(t) \}_{t \in \R}$. Indeed, according to \eqref{eqn:expFBM1alter}, when $J$ is large enough, $\{ B_h(t) \}_{t \in \R}$ can be approximated by its low frequency part 
\[
B_{h,J}(t)=\sum_{k \in \Z} 2^{-Jh} S_{J,k}^{(h)} \left(\Phi_\Delta^{(h+1/2)}(2^Jt-k)-\Phi_\Delta^{(h+1/2)}(-k) \right),
\]
whose coefficients $S_{J,k}^{(h)}$, $k\in\Z$, can be rather easily obtained from the coefficients $S_{J-1,k}^{(h)}$, $k\in\Z$, of $\{ B_{h,J-1}(t) \}_{t \in \R}$ by induction (pyramidal Mallat-type scheme); roughly speaking, this is due to the fact that the fractional scaling function $\Phi_\Delta^{(h+1/2)}$ generates a multiresolution analysis of $L^2(\R)$ (see \cite{MR1755100}). 
\end{Rmk}

In fact, FBM belongs to a much larger class of chaotic processes, the so-called Hermite processes. They are self-similar with stationary increments possessing a long-range dependence property. They first appeared in a natural way as limits of normalized partial sums of "strongly" correlated stationary Gaussian random sequences, in the so-called Non-Central Limit theorems established long time ago by Taqqu, Dobrushin and Major \cite{MR400329,MR550123,MR550122}. Apart the FBM, which is the Hermite process of order $1$, any other Hermite process of arbitrary integer order $d\ge 2$ is non-Gaussian; in fact it belongs to the $d$th Wiener chaos, and it is even considered to be a paradigmatic example of a stochastic process in this chaos whose many properties are still unknown, though the second order chaos has turned out to be less difficult to study than the higher order chaoses. %This lack of Gaussianity for the other Hermite processes makes them an alternative candidate for modelling purposes. %For this reason, they have attracted the attention of various authors in recent years, see for instance 
This fact have motivated many authors, interested in 
"conquering" non-Gaussian Wiener chaoses, to explore various issues related with them, we refer for instance to \cite{MR2820000,MR2808541,MR2578831,MR2352951,MR2105535,MR2573552,MR3079303} to cite but a few works in this area.

By the end of the introduction of the paper \cite{MR2105535} (see page 602 in it) published in 2004, Pipiras raised the problem to know whether wavelet-type random series representations with a well-localized smooth scaling function, reminiscent to the representation \eqref{eqn:expFBM1alter} of FBM, can be obtained for a Hermite process of any order $d$. He solved it in this same paper in the particular case $d=2$ in which the Hermite process is called the Rosenblatt process. Moreover, some further advances have recently been made in this particular case $d=2$ in the article \cite{MR4110623} in which a rather sharp estimate of the almost 
sure uniform rate of convergence of the wavelet-type random series representing the Rosenblatt process has been obtained, and has even been shown to be valid in the extended framework of the generalized Rosenblatt process. For deriving this sharp estimate, the article \cite{MR4110623} has introduced a new strategy which basically consists in expressing in a non-classical new way the approximation errors related with the approximation spaces of a multiresolution analysis of $L^2 (\R^2)$, namely in terms of bivariate wavelet functions having two distinct dilation indices $j_1$ and $j_2$ (see Section \ref{sect:mainandstrat} for more precision).

So far, the challenging problem presented in the previous paragraph has remained completely open in the general case $d\ge 3$. In fact, for solving it, one has to face at least the following two major difficulties:
\begin{itemize}
\item[(a)] To find in which way the low frequency part of an arbitrary Hermite process can be expressed in terms of FARIMA sequences and fractional scaling functions belonging to the Schwartz class.
\item[(b)] To show that a wavelet-type random series representation of an arbitrary Hermite process is almost surely uniformly convergent on compact intervals, and to estimate its almost sure uniform rate of convergent; the method introduced in \cite{MR4110623}  for reaching such a goal in the particular case of the generalized Rosenblatt process seems to be also useful in the general case of a Hermite process, yet some parts of it need to be significantly modified, in particular the crucial equality (2.33) in \cite{MR4110623} fails to be true in the general case since, for $d\ge 3$, as far as we know, there is no generalized Plancherel formula which, loosely speaking, would be of the type:
$\int_{\R^d} \prod_{l=1}^d f_l (x)\, dx=b_d \int_{\R^d} \prod_{l=1}^d \widehat {f}_l(\xi)\, d\xi$, where $b_d$ is a universal constant only depending on $d$, and $\widehat {f}_l$ is the Fourier transform of the function $f_l$.
\end{itemize}

The main aim of our present article is to propose a solution for this open problem, not only for usual Hermite processes but also for the generalized Hermite processes, of any integer order $d\ge 3$, which were introduced by Bai and Taqqu in \cite{MR3163219} and which extend the generalized Rosenblatt processes ($d=2$) due to  Maejima and Tudor \cite{MR2959876}. Also, with this article, we hope to open the door to future development of simulation methods for such generalized chaotic processes for which no simulation method is available so far. We hope as well to open the door to that of new strategies allowing to study in depth their erratic local sample path behavior, as for instance to show the existence of slow points and rapid points, in the same spirit of what has been very recently done for FBM in \cite{esserloosveldt} and for generalized Rosenblatt process in \cite{dawloosveldt}. 

The generalized Hermite process of an arbitrary integer order $d\ge 2$ is denoted by $\{X_{\mathbf{h}}^{(d)}(t)\}_{t \in \R_+}$, because it depends on a vector-valued Hurst parameter $\mathbf{h}:=(h_1,\ldots, h_d)$ whose coordinates $h_l$ satisfy 
\begin{equation}\label{eqn:hermi:cond}
h_1,\cdots,h_d \in (1/2,1) \text{ and } \sum_{\ell=1}^d h_\ell > d-\frac{1}{2}.
\end{equation}
This process belongs to the non-Gaussian $d$th Wiener chaos since it is defined, for each $t \in \R_+$, through the multiple Wiener integral:
\begin{equation}\label{eqn:def:ghp}
X^{(d)}_{\mathbf{h}}(t) := \int_{\R^d}' K^{(d)}_{\mathbf{h}}(t,x_1,\ldots,x_d) dB(x_1) \cdots dB(x_d),
\end{equation}
where $\{B(x)\}_{x \in \R}$ is a usual Brownian motion on the underlying probability space $(\Omega, \mathcal{F},\mathbb{P})$, and where the deterministic kernel function $K^{(d)}_{\mathbf{h}}$ is given, for every $(t,x_1,\ldots,x_d) \in \R_+ \times \R^d$, by
\begin{equation}\label{eqn:def:kernel}
K^{(d)}_{\mathbf{h}}(t,x_1,\cdots,x_d) := \frac{1}{\prod_{\ell=1}^d \Gamma(h_\ell-1/2)} \int_0^t \prod_{j=1}^d (s-x_\ell)_+^{h_\ell-3/2} \, ds,
\end{equation}
Observe that the symbol $\int_{\R^d}'$ in \eqref{eqn:def:ghp} denotes integration over $\R^d$ with diagonals $\{x_\ell=x_{\ell'}\}$, $\ell \neq \ell'$, excluded. Also observe that when all the coordinates $h_1,\ldots, h_d$ of the vector-valued parameter $\mathbf{h}$ are equal, then the process $\{X_{\mathbf{h}}^{(d)}(t)\}_{t \in \R_+}$ reduces to usual Hermite process.

The remaining of our article is organized as follows. In Section \ref{sect:mainandstrat}, we present the main lines of our strategies as well as some major ingredients in them including some preliminary proofs, and we state our three main theorems. Sections \ref{sect:approxproc}, \ref{sect:approxerro} and \ref{sec:altern-rep} are completely devoted to the proofs of our three main theorems. Some important results on multiple Wiener integrals, which are very useful for us, are given in Appendix \ref{sec:ape;profito}. At last the statements of some technical Lemmas, borrowed from the article \cite{MR4110623} and used in many our proofs, are recalled in Appendix \ref{sec:use-lemmata}.

\section{Strategies, main results and some major ingredients} \label{sect:mainandstrat}

Let us start by briefly recalling some fundamental definitions and facts from wavelet analysis in $L^2(\R^d)$ which will be useful for justifying our strategies.

\begin{Def}
\label{def:mraL}
A \textit{multiresolution analysis} of the Hilbert space $L^2(\R^d)$ is a sequence $(V_j^d)_{j \in \Z}$ of closed linear subspaces of $L^2(\R^d)$ such that
\begin{enumerate}[(a)]
\item for all $j \in \Z$, $V_j^d \subseteq V_{j+1}^d$;
\item $\bigcap_{j \in \Z} V_j^d = \{0\}$ and $\bigcup_{j \in \Z} V_j^d$ is dense in $L^2(\R^d)$;
\item for all $j \in \Z$, $V_j^d= \{f(2^j \cdot) \, : \, f \in V_0^d \}$;
\item there exists a function $\Phi \in V_0^d$, called scaling function, such that the sequence $\big (\Phi( \cdot -\mathbf{k})\big)_{\mathbf{k} \in \Z^d}$ is an orthonormal basis of $V_0^d$. Notice that in the univariate case $d=1$, this function $\Phi$ is denoted by $\phi$ as in the previous Section \ref{sec:intro}. 
\end{enumerate}
\end{Def}

\begin{Rmk}
It clearly results from (c) and (d) in Definition \ref{def:mraL}, that, for all fixed $j \in \Z$, the sequence $\big (2^{jd/2}\Phi( 2^{j} \cdot -\mathbf{k})\big)_{\mathbf{k} \in \Z^d}$ is an orthonormal basis of $V_j^d$.
\end{Rmk}

%A multiresolution analysis of $L^2(\R^d)$ provides a natural way to approximate any function $f\in L^2(\R^d)$; namely $f$ is approximated by $f_J$, its (orthogonal) projection of $f$ onto a space $V_J^d$, with $J$ large enough. Note that, in this case, $f-f_J$ belongs to $(V_J^d)^\bot$, the orthogonal complement of $V_J^d$ in $L^2(\R^d)$. In particular, one can expand $f_J$ in the basis $(2^{Jd/2}\Phi ( 2^{J} \cdot -\mathbf{k}))_{\mathbf{k} \in \Z^d}$ of $V_J^d$. In order to obtain an orthogonal basis of the space $(V_J^d)^\bot$, 
Usually, one denotes by $W_J^d$ the orthogonal complement of $V_J^d$ in $V_{J+1}^d$. 
%that is 
%\[
%V_{J+1}^d=V_J^d \overset{\bot } {\oplus} W_J^d.
%\]
Then, it follows from (a) and (b) in Definition \ref{def:mraL} that, for all fixed $J \in \Z$, the following very important equalities hold:
\begin{equation}
\label{eq:fundVW}
V_J^d = \overset{\bot}{\bigoplus_{- \infty<j<J}} W_j^d \text{ and } L^2(\R^d) =  V_J^d \overset{\bot } {\oplus} \left(\overset{\bot}{\bigoplus_{J \leq j < + \infty}} W_j^d \right) = \overset{\bot}{\bigoplus_{j \in \Z}} W_j^d \,.
\end{equation}
Using \eqref{eq:fundVW}, with $d=1$ and an arbitrary $J$, one can derive from the following fundamental theorem (see e.g. \cite{MR1162107,MR1085487}) orthonormal bases for the subspace $V_J^1\subset L^2(\R)$ and for the whole space $L^2(\R)$.

\begin{Thm}\label{thm:wavbas1}
There is a function $\psi\in W_0^1$, called mother wavelet, such that, for all fixed $j \in \Z$, the sequence of functions $\big (2^{j/2} \psi( 2^{j/2}\cdot -k)\big)_{k \in \Z}$ is an orthonormal basis of $W_j^1$. Then, the important equalities  (\ref{eq:fundVW}), imply, for all fixed $J \in \Z$, that:
\begin{enumerate}[(a)]
\item the sequence of functions $\big (2^{j/2}\psi( 2^{j}\cdot -k)\big)_{j <J,\,k \in \Z}$ is an orthonormal basis for the space $V_J^1$;
\item the sequences  of functions $\big (2^{J/2}\phi( 2^J\cdot -k)\big)_{ k \in \Z} \cup \big (2^{j/2}\psi( 2^{j}\cdot -k)\big)_{j \geq J,\, k \in \Z}$ and $(2^{j/2}\psi( 2^{j}\cdot -k))_{(j,k) \in \Z^2}$ are two orthonormal bases for the space $L^2(\R)$.
\end{enumerate}
Such bases are called orthonormal wavelet bases.
\end{Thm}
Thanks to the tensor product method (see e.g. \cite{MR1162107,MR1085487}), for any integer $d>1$ one can construct from a multiresolution analysis $(V_j^1)_{j\in\Z}$ of $L^2(\R)$ a multiresolution analysis $(V_j^d)_{j\in\Z}$ for $L^2 (\R^d)$. Namely, for each $j\in\Z$, the space $V_j^d$ is defined as $V_j^d:=(V_j^{1})^{\otimes_d}$ the tensor product of the space $V_j^1$, $d$ times with itself. Then a scaling function $\Phi$, which can be associated in a natural way to such a multiresolution analysis $(V_j^d)_{j\in\Z}$, is $\Phi:=\phi^{\otimes_d}$, the tensor product of the univariate scaling function $\phi$, $d$ times with itself. In such a setting, it is well known that, for any fixed $J\in\Z$, an orthonormal wavelet basis of the space $(V_J^d)^\bot$ (the orthogonal complement of $V_J^d$ in $L^2 (\R^d)$) is: 
\begin{eqnarray*}
&&\Big\{2^{j d/2} \prod_{l=1}^d\psi^{(\eta_l)} (2^j x_l-k_l)\,:\, j\in\Z \mbox{ and } j\ge J,\\
&&\hspace{3cm} (\eta_1,\ldots, \eta_d)\in \{0,1\}^d \setminus \{0\}^d, \, (k_1,\ldots, k_d)\in\Z^d\Big\}, 
\end{eqnarray*}
where $\psi^{(0)}$ and $\psi^{(1)}$ respectively denote the univariate scaling function and mother wavelet $\phi$ and $\psi$ (see Theorem \ref{thm:wavbas1}).
Nevertheless, a major ingredient of strategies of our article consists in making use of another much less classical orthonormal wavelet basis of $(V_J^d)^\bot$. This idea comes from the article \cite{MR2105535} in which $d=2$ and whose main goal was to estimate almost sure rate of convergence of wavelet-type random series representation of the generalized Rosenblatt process.

In order to precisely define the non classical orthonormal wavelet basis of $(V_J^d)^\bot$ we intend to use, we need to give some further notations. For all multi-indices $\mathbf{j}=(j_1,\ldots, j_d)\in \Z^d$ and $\mathbf{k}=(k_1,\ldots,k_d)\in \Z^d$, we denote by $\psi_{\mathbf{j},\mathbf{k}}$ the multivariate wavelet function belonging to $L^2 (\R^d)$ defined as the tensor product: 
\begin{equation}
\label{eq:tensor-psi}
\psi_{\mathbf{j},\mathbf{k}}:=\bigotimes_{\ell=1}^{d}  \psi_{{j}_\ell,{k}_\ell},
\end{equation}
where the univariate wavelet functions $\psi_{j_l,k_l}$ are defined, for every $x\in\R$, as: $\psi_{j_l,k_l}(x):=2^{j_{l}/2} \psi(2^{j_l} x-k_l)$. Observe that the previous definition $V_J^d$ through tensor product, and the point $(a)$ in Theorem \ref{thm:wavbas1} implies that the collection of functions 
\[ \left \{\psi_{\mathbf{j},\mathbf{k}} \, : \, \mathbf{j},\mathbf{k} \in \Z^d \text{ and } \max_{\ell \in [\![1,d]\!]} j_\ell <J \right\}   \]
is an orthonormal basis of the subspace $V_J^d\subset L^2 (\R^d)$; while the point $(b)$ in this same theorem entails that the collection of functions
$
\big\{\psi_{\mathbf{j},\mathbf{k}} \, : \, \mathbf{j},\mathbf{k} \in \Z^d\big\} 
$ 
is an orthonormal basis of the whole space $L^2 (\R^d)$, since $L^2 (\R^d)=\big (L^2 (\R)\big)^{\otimes_d}$. Combining these two results, it turns out that the collection of functions
\begin{equation}\label{eqn:wavbasvjort}
\left \{\psi_{\mathbf{j},\mathbf{k}} \, : \, \mathbf{j},\mathbf{k} \in \Z^d \text{ and } \max_{\ell \in [\![1,d]\!]} j_\ell \geq J \right\}  
\end{equation}
is an orthonormal basis of the subspace $(V_J^d)^\bot\in L^2 (\R^d)$ which is the orthogonal complement of $V_J^d$ in $L^2(\R^d)$.

Let us now precisely explain the connection between the latter basis and the error of approximation of a generalized Hermite process by the scaling function part of its wavelet-type random series representation. For each fixed $t\in\R_+$ and integer $J\ge 1$, the two functions of $L^2(\R^d)$ $(x_1,\ldots,x_d) \mapsto K^{(d)}_{\mathbf{h},J}(t,x_1,\cdots,x_d)$ and $(x_1,\ldots,x_d) \mapsto K^{(d,\bot)}_{\mathbf{h},J}(t,x_1,\cdots,x_d)$  
respectively denote the two orthogonal projections of the function $(x_1,\ldots,x_d) \mapsto  K^{(d)}_{\mathbf{h}}(t,x_1,\cdots,x_d)$ (see \eqref{eqn:def:ghp} and \eqref{eqn:def:kernel}) onto $V_J^d$ and $(V_J^d)^\bot$. One clearly has that
\[ K^{(d)}_{\mathbf{h}}(t,\bullet)-K^{(d)}_{\mathbf{h},J}(t,\bullet)=K^{(d,\bot)}_{\mathbf{h},J}(t,\bullet),\]
which leads us to define the approximation and details processes associated with the generalized Hermite process $\{X_\mathbf{h}^{(d)}(t) \}_{t \in\R_+}$ in the following way:
\begin{Def}
Let $d \in \N$ and $\mathbf{h}$ satisfying the conditions \eqref{eqn:hermi:cond}. For all $J \in \N$, the \textit{approximation process at scale $J$} of the generalized Hermite process $\{X_\mathbf{h}^{(d)}(t) \}_{t \in \R_+}$ is the process defined, for all $t \in\R_+$, by the multiple integral:
\begin{equation}\label{eqn:approxprocess}
X^{(d)}_{\mathbf{h},J}(t):= \int_{\R^d}' K^{(d)}_{\mathbf{h},J}(t,x_1,\ldots,x_d) dB(x_1) \ldots dB(x_d);
\end{equation}
in fact $\{X_{\mathbf{h},J}^{(d)}(t) \}_{t \in \R_+}$ can be viewed as the scaling function part of the wavelet-type random series representation of $\{X_\mathbf{h}^{(d)}(t) \}_{t \in\R_+}$.
The \textit{details process at scale $J$} is defined, for all $t \in\R_+$, as:
\begin{equation}\label{eqn:detailsprocess}
X^{(d,\bot)}_{\mathbf{h},J}(t):=X^{(d)}_{\mathbf{h}}(t)-X^{(d)}_{\mathbf{h},J}(t)= \int_{\R^d}' K^{(d,\bot)}_{\mathbf{h},J}(t,x_1,\ldots,x_d) dB(x_1) \ldots dB(x_d);
\end{equation}
in fact $\{X^{(d,\bot)}_{\mathbf{h},J}(t)\}_{t\in\R_+}$ can be viewed as the error stemming from the approximation of $\{X_\mathbf{h}^{(d)}(t) \}_{t \in\R_+}$ by $\{X_{\mathbf{h},J}^{(d)}(t) \}_{t \in\R_+}$.
\end{Def}
Observe that combining \eqref{eqn:approxprocess} with the Wiener isometry and the fact that $\big(2^{J \frac{d}{2}}\Phi(2^J\cdot-\mathbf{k} )\big)_{ \mathbf{k} \in \Z^d}$ is an orthonormal basis of $V_J^d$, one gets, for each fixed $t\in\R_+$, that
\begin{equation}\label{eq:XJ}
X^{(d)}_{\mathbf{h},J}(t) = \sum_{\mathbf{k} \in \Z^d} \mu_{J,\textbf{k}} \mathfrak{K}^{(d,\mathbf{h})}_{J,\textbf{k}}(t), 
\end{equation} 
where the chaotic random variables $\mu_{J,\textbf{k}}$ and the deterministic coefficients $\mathfrak{K}^{(d,\mathbf{h})}_{J,\textbf{k}}(t)$ are given by: 
\begin{equation} \label{eqn:rv:mu}
\mu_{J,\mathbf{k}} := 2^{J \frac{d}{2}}\int_{\R^d}' \phi(2^J x_1-k_1) \cdots \phi(2^J x_d-k_d) \, dB(x_1) \ldots dB(x_d)
\end{equation}
and 
\begin{equation}\label{eqn:coefkbis}
\mathfrak{K}^{(d,\mathbf{h})}_{J,\textbf{k}}(t) :=2^{J \frac{d}{2}} \int_{\R^d} K^{(d)}_{\mathbf{h}}(t,x_1,\ldots,x_d) \phi(2^J x_1-k_1) \cdots \phi(2^J x_d-k_d) \, dx_1\cdots dx_d.
\end{equation}
Also observe that combining \eqref{eqn:detailsprocess} with the Wiener isometry and the fact that the collection of functions in \eqref{eqn:wavbasvjort} is an orthonormal basis of $(V_J^d)^\bot$, one obtains, for each fixed $t\in\R_+$, that
\begin{equation}\label{eq:X-XJ}
X^{(d,\bot)}_{\mathbf{h},J}(t)=  \sum_{\substack{(\mathbf{j},\mathbf{k}) \in (\Z^d)^2 \\  \displaystyle\max_{\ell \in [\![1,d]\!]} j_\ell \geq J }} \varepsilon_{\textbf{j},\textbf{k}}  \, \mathcal{K}^{(d,\mathbf{h})}_{\textbf{j},\textbf{k}}(t),
\end{equation}
where the chaotic random variables $\varepsilon_{\textbf{j},\textbf{k}}$ and the deterministic coefficients $\mathcal{K}^{(d,\mathbf{h})}_{\textbf{j},\textbf{k}}(t)$ are given by: 
\begin{equation}\label{eqn:rv:eps}
 \varepsilon_{\textbf{j},\textbf{k}} := \int_{\R^d}' \psi_{\mathbf{j},\mathbf{k}}(x_1, \ldots,x_d) \, dB(x_1)\ldots dB(x_d)
\end{equation}
and
\begin{equation}\label{eqn:coefk}
\mathcal{K}^{(d,\mathbf{h})}_{\textbf{j},\textbf{k}}(t) := \int_{\R^d} K^{(d)}_{\mathbf{h}}(t,x_1,\ldots,x_d) \psi_{\mathbf{j},\mathbf{k}}(x_1, \ldots,x_d) \, dx_1\cdots dx_d.
\end{equation}
One mentions in passing that, so far, one only knows that the random series in \eqref{eq:XJ} and 
\eqref{eq:X-XJ} are unconditionally convergent in $L^2(\Omega)$, for each fixed $t\in\R_+$.

\begin{Rmk}
\label{rem:hyp-meyer}
Similarly to the article \cite{MR1755100}, from now and till the end of our article we always assume that the univariate scaling function and mother wavelet $\phi$ and $\psi$ are associated with an orthonormal Meyer wavelet basis of $L^2(\R)$ (see Remark \ref{rem:mey-bas}). Then, it results from \eqref{eqn:fractantideri}, \eqref{eqn:def:kernel}, \eqref{eqn:coefk}, Fubini theorem and the changes of variable $y_\ell=2^{j_\ell}x_\ell -k_\ell$ (for all $\ell \in [\![1,d]\!]$), that
\begin{equation}\label{eqn:coefk2}
\mathcal{K}^{(d,\mathbf{h})}_{\textbf{j},\textbf{k}}(t) = 2^{ j_1 (1-h_1)+\cdots + j_d (1-h_d)} \int_0^t \prod_{\ell=1}^d \psi_{h_\ell}(2^{j_\ell}s-k_\ell) \, ds,
\end{equation}
where $\psi_{h_l}$ is the fractional primitive or order $h_l-1/2$ of $\psi$. Also, one can derive from \eqref{eqn:coefkbis} and similar arguments that
\begin{equation}\label{eqn:coefk2bis}
\mathfrak{K}^{(d,\mathbf{h})}_{J,\textbf{k}}(t) = 2^{-J(h_1+ \cdots+h_d-d)} \int_0^t \prod_{\ell=1}^d \phi_{h_\ell}(2^{J}s-k_\ell) \, ds,
\end{equation}
where $\phi_{h_l}$ is the fractional primitive or order $h_l-1/2$ of $\phi$. Notice that one knows from \eqref{def:fractwav:eq1} that the Fourier transform of $\phi_{h_l}$ satisfies
\begin{equation}
\label{eqn:coefk2bis-f}
\widehat{\phi}_{h_l}(\xi)=(i \xi)^{1/2-h_l}\widehat{\phi}(\xi),\quad\mbox{for almost all $\xi\in\R$.} 
\end{equation}
\end{Rmk}

In view of the fact that the $\phi_{h_l}$'s fail to belong to the Schwartz class $\mathcal{S}(\R)$ and are even badly localized functions, one of the main goal of our article will be to introduce, in the same spirit of what has been done for the approximation process of FBM in \cite{MR1755100} and for that of Rosenblatt process in \cite{MR2105535}, a modified version of the random series representation \eqref{eq:XJ} in which the deterministic coefficients are expressed in terms of "nice" fractional scaling functions (see Definition \ref{def:fractscal}) belonging to the  Schwartz class. In order to adapt ideas of \cite{MR1755100,MR2105535} to the framework of the generalized Hermite process $\{X_\mathbf{h}^{(d)}(t) \}_{t \in\R_+}$, which is much more complex than those of FBM and Rosenblatt process, we need to introduce, for each fixed $J\in\Z$, the sequence of random variables $(\sigma_{J,\textbf{k}}^{(\mathbf{h})})_{\textbf{k}\in\Z^d}$, defined, for all $\textbf{k}\in\Z^d$, as:
\begin{equation}
\label{eq:ext-farima}
\sigma_{J,\textbf{k}}^{(\mathbf{h})}:=\sum_{\textbf{p}\in\N_0^d}\Big (\prod_{l=1}^d \gamma_{p_l}^{(h_l-1/2)}\Big)\mu_{J,\mathbf{k}-\mathbf{p}},
\end{equation}
where the deterministic coefficients $\gamma_{p_l}^{(h_l-1/2)}$ are given by the third and the second equalities in \eqref{eqn:farima} with $p=p_l$ and $\delta=h_l-1/2$. Notice that the following Proposition \ref{prop:gen-farima} shows, among other things, that the definition \eqref{eq:ext-farima} makes sense.
Roughly speaking, the sequence $(\sigma_{J,\textbf{k}}^{(\mathbf{h})})_{\textbf{k}\in\Z^d}$ can be viewed as a generalized FARIMA sequence. In fact, it can be expressed in terms of usual FARIMA sequences (see Proposition \ref{prop:gen-farima} below). In order to provide the latter expression of $\sigma_{J,\textbf{k}}^{(\mathbf{h})}$ we need the following definition:

\begin{Def}
\label{def:setPS}
Let $S$ be an arbitrary finite subset of $\N$ whose cardinality is denoted by $\# S$. Then, for any integer $m$ such that $0 \leq m \leq \lfloor \# S /2 \rfloor$, one denotes by $\mathcal{P}_m^{S}$  the finite set of the partitions of $S$ with $m$ (non ordered) pairs and $\# S-2m$ singletons. Moreover, for the sake of simplicity, when $S=[\![1,n]\!]$ with $n \in \N$ being arbitrary, one sets $\mathcal{P}_m^{(n)}=\mathcal{P}_m^{[\![1,n]\!]}$.
\end{Def}

\begin{Prop}
\label{prop:gen-farima}
For all fixed $(J,\textbf{k}) \in \Z \times \Z^d$, the random series in \eqref{eq:ext-farima} is convergent almost surely and in $L^\gamma (\Omega)$, for any $\gamma\in (0,+\infty)$. Moreover, one has that 
\begin{align}\label{eqn:defofsigma}
\sigma_{J,\textbf{k}}^{(\mathbf{h})} = \sum_{m=0}^{\lfloor d/2 \rfloor } (-1)^m \sum_{P \in \mathcal{P}_m^{(d)}} \prod_{r=1}^m \mathbb{E}[ Z_{J,k_{\ell_r}}^{(h_{\ell_r}-1/2)}Z_{J,k_{\ell_r'}}^{(h_{\ell_r'}-1/2)}] \prod_{s=m+1}^{d-m} Z_{J,k_{\ell_s}}^{(h_{\ell_s}-1/2)},
\end{align}
where the indices $\ell_r$, $\ell_r'$ and $\ell_{s}''$ are such that 
\[ P =\Big\{\{\ell_1,\ell_1 '\},\ldots,\{\ell_m,\ell_m '\},\{\ell_{m+1}''\},\ldots,\{\ell_{d-m}''\}\Big\},\]
and where, for all $\delta \in (0,1/2)$, $(Z^{(\delta)}_{J,q})_{q \in \Z}$, is the FARIMA $(0,\delta,0)$ sequence (see Definition \ref{def:farima}) associated with the sequence $(g^\phi_{J,k})_{k\in\Z}$ of i.i.d. $\mathcal{N}(0,1)$ Gaussian random variables introduced in \eqref{eqn:n01phi}.
%given, for all $q \in \Z$, by 
%\[
% Z_{J,q}^{(\delta)} := \sum_{p=0}^{+ \infty} \gamma_p^{(\delta)} g_{J,q-p}^\phi.
% \]
\end{Prop}

Before proving Proposition \ref{prop:gen-farima}, let us state the first main theorem of our article which provides a modified version of the random series representation \eqref{eq:XJ} obtained through the generalized FARIMA sequence $(\sigma_{J,\textbf{k}}^{(\mathbf{h})})_{\textbf{k}\in\Z^d}$ (see \eqref{eqn:defofsigma} and \eqref{eq:ext-farima}) as well as "nice" fractional scaling functions (see Definition \ref{def:fractscal}) belonging to the  Schwartz class.

\begin{Thm}\label{thm:basefreq}
The approximation process $\{X_{\mathbf{h},J}^{(d)}(t)\}_{t\in\R_+}$, defined in \eqref{eqn:approxprocess}, can be expressed, for all $t\in\R_+$, as:
\begin{equation}\label{eqn:basfreq:toprove}
X_{\mathbf{h},J}^{(d)}(t)= 2^{-J(h_1+\ldots+h_d-d)} \sum_{\mathbf{k} \in \Z^d} \sigma_{J,\textbf{k}}^{(\mathbf{h})} \int_0^t \prod_{\ell=1}^d \Phi_\Delta^{(h_\ell-1/2)}(2^J u-k_\ell) \, du,
\end{equation}
where the series is almost surely uniformly convergent in $t$ on each compact interval of $\R_+$.
\end{Thm}

\begin{Rmk}
Let $f$ be an arbitrary function in the Schwartz class $\mathcal{S}(\R)$ and let $(a_p)_{p \in \Z}$ be an arbitrary slowly increasing sequence of real numbers, that is we have, for some constants $\kappa>0$ and $\mu >0$ and for every $p\in\Z$,  $|a_p| \leq \kappa (1+|p|)^\mu$. It is known (see for instance \cite{MR1755100}) that, if we set $A_0=0$ and $A_q-A_{q-1}=a_q$ for all $q\in \Z\setminus\{0\}$ and $\widetilde{f}(y) = \int_{y-1}^y f(v) \, dv$ for all $y \in \R$, then the function $\widetilde{f}$ belongs to $\mathcal{S}(\R)$ and the sequence $\{A_k\}_{k \in \Z}$ is slowly increasing. Moreover, using an Abel transform, for all $t \in \R$, we have
\begin{equation}
\label{eq:ablel}
\sum_{k \in \Z} a_k\int_0^t f(v-k) \, dv = \sum_{q \in \Z} A_q (\widetilde{f}(t-q)-\widetilde{f}(-q)).
\end{equation} 
In order to apply \eqref{eq:ablel} in the framework of Theorem \ref{thm:basefreq}, we define, for each $(J,q, \mathbf{n}) \in \N \times \Z \times \Z^{d-1}$, the random variable $S_{J,q,\mathbf{n}}^{(\mathbf{h})}$ as:
\begin{equation}
\label{eq:gen-rw-farima}
S_{J,q,\mathbf{n}}^{(\mathbf{h})}=\left\{ \begin{array}{cl}
                      \sum_{p=1}^q \sigma_{J,(p,\mathbf{n}+p)}^{(\mathbf{h})} & \text{if } q > 0\\ [2ex]
                       0 & \text{if } q = 0\\ [2ex]
                       -\sum_{p=q+1}^0 \sigma_{J,(p,\mathbf{n}+p)}^{(\mathbf{h})} & \text{if } q <0,
                     \end{array} \right.
\end{equation}
with the convention that $\mathbf{n}+p:=(n_1+p,\ldots, n_{d-1}+p)$. Also, for every $\mathbf{n}=(n_1,\ldots, n_{d-1})\in\Z^{d-1}$, we define the function $\widetilde{\Phi}_{\Delta,\mathbf{n}}^{(\mathbf{h})}$, belonging to $\mathcal{S}(\R)$, as: 
\[ \widetilde{\Phi}_{\Delta,\mathbf{n}}^{(\mathbf{h})} (y):= \int_{y-1}^{y} \Phi_{\Delta}^{(h_1-1/2)}(v) \prod_{\ell=1}^{d-1} \Phi_{\Delta}^{(h_{\ell+1}-1/2)}(v-n_\ell) \, dv,\quad\mbox{for all $y\in\R$.} \]
Then, using Theorem \ref{thm:basefreq}, Fubini theorem, the change of variable $v=2^J u$, the change of indices $k_1=p$ and $n_{l-1}=k_{l}-k_1$ (for all $l\in [\![2,d]\!]$), the slow increase property for the sequence $\big(\sigma_{J,(p,\mathbf{n}+p)}^{(\mathbf{h})}\big)_{p\in\Z}$ provided by \eqref{rem:logbousig:eq1}, \eqref{eq:ablel}, a slow increase property (derived from \eqref{rem:logbousig:eq1} and \eqref{eq:gen-rw-farima}) for the sequence $\big(S_{J,q,\mathbf{n}}^{(\mathbf{h})}\big)_{q\in \Z}$ with a random constant\footnote{All along this paper, if $\mathbf{n} \in \Z^d$, we use the notation $|\mathbf{n}|=\sum_{\ell=1}^d |n_\ell|$.} $\kappa (\mathbf{n})=\mathcal{O}\big(\log^{d/2}(3+|\mathbf{n}|)\big)$, and the inequality 
\[
\sup_{y\in [0,Y]}\sup_{(q,\mathbf{n})\in \Z\times\Z^{d-1}}\Big\{\big (3+|p|+|\mathbf{n}|\big)^{L} \big|\widetilde{\Phi}_{\Delta,\mathbf{n}}^{(\mathbf{h})} (y)\big|\Big\}<\infty,\quad\mbox{for all fixed $Y,L>0$,}
\]
we obtain that
\begin{equation}
\label{eq:abel-appoxp}
X_{\mathbf{h},J}^{(d)}(t)= 2^{-J(h_1+\ldots+h_d+1-d)} \sum_{\mathbf{n} \in \Z^{d-1}} \sum_{q \in \Z} S_{J,q,\mathbf{n}}^{(\mathbf{h})}\left( \widetilde{\Phi}_{\Delta,\mathbf{n}}^{(\mathbf{h})}(2^Jt-q)-\widetilde{\Phi}_{\Delta,\mathbf{n}}^{(\mathbf{h})}(-q) \right) , 
\end{equation}
where the convergence of the random series holds almost surely and uniformly in $t$ on each compact interval of $\R_+$. Notice that the random series representation \eqref{eq:abel-appoxp} for the approximation $\{X_{\mathbf{h},J}^{(d)}(t)\}_{t\in\R_+}$ of the generalized Hermite process is reminiscent to that of the low frequency part (that is the scaling function part) in the representation of FBM in  \eqref{eqn:expFBM1alter}.
\end{Rmk}

 The proof of Theorem \ref{thm:basefreq} will be given in Section \ref{sect:approxproc}. Let us now focus on the proof of the fundamental Proposition \ref{prop:gen-farima}. Its starting point consists in an expression of the chaotic random variable $\mu_{J,\mathbf{k}}$ (see \eqref{eqn:rv:mu}) in terms of the i.i.d Gaussian random variables $g^\phi_{J,k}$ and Hermite polynomials $H_n$. We mention in passing that a rather similar expression also holds for the chaotic random variable $\varepsilon_{\textbf{j},\textbf{k}}$ (see \eqref{eqn:rv:eps}); it will be useful for us later. For giving these expressions for $\mu_{J,\mathbf{k}}$ and $\varepsilon_{\textbf{j},\textbf{k}}$ it is convenient to make use of the very common notation for multiple Wiener integral: for any $n\in\N$ and $f \in L^2(\R^n)$,
\[ I_n(f) =  \int_{\R^n}' f(x_1, \ldots, x_n) \, dB(x_1) \ldots dB(x_n).\]
It is known (see e.g. equation (1) in \cite{MR1106271}) that, for any univariate functions $\varphi_1,\ldots, \varphi_p$  of $L^2(\R)$ which are orthonormal and for every $n_1, \ldots, n_p\in \N$, one has
\begin{equation}\label{eqn:itoprod}
I_{n_1 + \cdots +n_p} \left( \varphi_{1}^{\otimes_{n_1}} \otimes \cdots \otimes \varphi_{p}^{\otimes_{n_p}} \right) = \prod_{\ell=1}^{p} H_{n_{\ell}} \left( \int_\R \varphi_\ell(x) \, dB(x) \right). 
\end{equation}
We recall that:
\begin{Def}
\label{def:hermP}
For all $n \in \Z_+$, the $n$th Hermite polynomial is the polynomial of degree $n$ denoted  by $H_n$ and defined, for every $x\in\R$, as:
\[ H_n(x)=(-1)^n e^{x^2 /2} \frac{d^n}{dx^n} e^{-x^2 /2} \,.\]
For instance, the first four Hermite polynomials are $H_0(x)=1$, $H_1(x)=x$, $H_2(x)=x^2-1$ and $H_3(x)=x^3-3x$. 
\end{Def}
The equality \eqref{eqn:itoprod} will play a crucial role in the sequel; for the sake of completeness its proof is given in Appendix \ref{sec:ape;profito}.
%Let us start by expressing the random variables \eqref{eqn:rv:mu} and \eqref{eqn:rv:eps} in a more explicit way. 
%If $f \in L^2(\R)$, we note $f^{\otimes_n}$ for the tensor product \[ \underset{\text{n times}}{\underbrace{f \otimes \cdots \otimes f}} \, : \, (x_1,\ldots,x_n) \mapsto f(x_1). \ldots. f(x_n) \] which is, of course, a $L^2(\R^n)$ (symmetric) function. 
In order to apply it to the multiple Wiener integrals in \eqref{eqn:rv:mu} and \eqref{eqn:rv:eps}, we need to introduce some notations. In fact any $(\mathbf{j},\mathbf{k}) \in (\Z^d)^2$ can be viewed as a finite sequence $\big ((j_m ,k_m)\big)_{1\le  m \le d}$ whose $d$ terms $(j_m ,k_m)$ belong to $\Z^2$ and some of them can be equal to each other. The positive integer $p(\mathbf{j},\mathbf{k})\le d$ denotes the number of the distinct terms of the sequence $(\mathbf{j},\mathbf{k})=\big ((j_m ,k_m)\big)_{1\le  m \le d}$, and the latter terms are denoted by $(\widetilde{j_\ell},\widetilde{k_\ell})$, $1\le l \le p(\mathbf{j},\mathbf{k})$; moreover the notation $(\widetilde{j_\ell},\widetilde{k_\ell})_{n_\ell}$, where $n_\ell\in\{1,\ldots, d\}$, means that $(\widetilde{j_\ell},\widetilde{k_\ell})$ has the multiplicity $n_\ell$, that is there are exactly $n_\ell$ terms of the sequence $\big ((j_m ,k_m)\big)_{1\le  m \le d}$ which are equal to $(\widetilde{j_\ell},\widetilde{k_\ell})$. At last, it is clear that 
$\sum_{\ell=1}^{p(\mathbf{j},\mathbf{k})} n_\ell =d$. Using these notations and \eqref{eqn:n01psi}, we can derive from \eqref{eqn:rv:eps} and \eqref{eqn:itoprod} that, for all $(\mathbf{j},\mathbf{k}) \in (\Z^2)^d$, 
\begin{equation}\label{eqn:rv:prodher}
\varepsilon_{\mathbf{j},\mathbf{k}} = \prod_{\ell=1}^{p(\mathbf{j},\mathbf{k})} H_{n_\ell} \left( g^\psi_{\widetilde{j}_\ell,\widetilde{k}_\ell}\right).
\end{equation}
Similar arguments and \eqref{eqn:n01phi} allow to shown that, for all $J \in \Z$ and $\mathbf{k} \in \Z^d$, 
\begin{equation}\label{eqn:rv:prodherbis}
\mu_{J,\mathbf{k}} = \prod_{\ell=1}^{p(\{J\}^d,\mathbf{k})} H_{n_\ell} \left( g^\phi_{J,\widetilde{k}_\ell}\right);
\end{equation}
 observe that the positive integer $n_\ell$ in \eqref{eqn:rv:prodherbis} is the multiplicity of $\widetilde{k}_\ell$ in $\mathbf{k}$.
 In order to connect the random variables $\mu_{J,\mathbf{k}}$ to FARIMA sequences (see Definition \ref{def:farima}), we have to rewrite the expression \eqref{eqn:rv:prodherbis} in a way that gives us an easier "access" to the i.i.d Gaussian random variables $g^\phi_{J,k}$ in it. %compared to the expression \eqref{eqn:rv:prodherbis}.
To this end, we recall that, for any $n\in\N$, the $n$th Hermite polynomial $H_n$ satisfies, for all $x\in\R$, the equality:
\begin{equation}
\label{eq:herm-part}
H_n (x)= \sum_{m=0}^{\lfloor n/2 \rfloor} (-1)^m a_m^{(n)} x^{n-2m},
\end{equation}
where $a_m^{(n)}$ is the number of partitions of $[\![1,n]\!]$ with $m$ (non ordered) pairs and $n-2m$ singletons.

\begin{Lemma}\label{prop:rv:formpourapprox}
Using notations already introduced in Definition \ref{def:setPS}, for all $J \in \Z$ and $\mathbf{k} \in \Z^d$, the random variable $\mu_{J,\mathbf{k}}$ in \eqref{eqn:rv:prodherbis} can be rewritten as:
\begin{align} \label{rv:formpourapprox}
\mu_{J,\mathbf{k}} = \sum_{m=0}^{\lfloor d/2 \rfloor} (-1)^m \sum_{P \in \mathcal{P}_m^{(d)}} \prod_{r=1}^m \mathbb{E}[g_{J,k_{\ell_r}}^\phi g_{J,k_{\ell_r'}}^\phi] \prod_{s=m+1}^{d-m} g_{J,k_{\ell_{s}''}}^\phi,
\end{align}
where the indices $\ell_r$, $\ell_r'$ and $\ell_{s}''$ are such that 
\[ P =\Big\{\{\ell_1,\ell_1 '\},\ldots,\{\ell_m,\ell_m '\},\{\ell_{m+1}''\},\ldots,\{\ell_{d-m}''\}\Big\}.\]
\end{Lemma}

\begin{proof}
Let us proceed by induction on the positive integer $d$. It easily follows from \eqref{eqn:rv:prodherbis} and Definition \ref{def:hermP} that the equality \eqref{rv:formpourapprox} is satisfied in the two particular cases $d=1$ and $d=2$. In the sequel, one assumes that $d>2$ and that \eqref{rv:formpourapprox} holds for any positive integer $n$ such that $n <d$. Let us first show that these assumptions allow to prove \eqref{rv:formpourapprox} when the $d$ indices forming the multi-index $\mathbf{k}$ are all equal together, that is $\mathbf{k}=(k_1,\ldots,k_1)$. Indeed, the latter equality implies, for all $\ell_r, \ell_r' \in [\![1,d]\!]$, that $\mathbb{E}[g_{J,k_{\ell_r}}^\phi g_{J,k_{\ell_r'}}^\phi]=1$, which in turn entails that 
\begin{align*}
\sum_{m=0}^{\lfloor d/2 \rfloor} (-1)^m \sum_{P \in \mathcal{P}_m^{(d)}} \prod_{r=1}^m \mathbb{E}[g_{J,k_{\ell_r}}^\phi g_{J,k_{\ell_r'}}^\phi] \prod_{s=m+1}^{d-m} g_{J,k_{\ell_{s}''}}
&= \sum_{m=0}^{\lfloor d/2 \rfloor} (-1)^m a_m^{(d)} (g_{J,k_1}^\phi)^{d-2m}\\
&= H_d(g_{J,k_1}^\phi)=\mu_{J,\mathbf{k}}\,, 
\end{align*}
where the second and the third equalities respectively follow from \eqref{eq:herm-part} and \eqref{eqn:rv:prodherbis}. From now on, we focus on the most general case in which the $d$ indices forming the multi-index $\mathbf{k}$ are not necessarily equal together. Generally speaking, there exists a unique integer $a$ satisfying $1\le a<d$ such that one has $\mathbf{k}=(k_1,\ldots,k_1,k_{a+1},\ldots,k_d)$ with $k_1 \neq k_\ell$, for all $ a<\ell \leq d$; in fact $a$ is nothing else than the multiplicity of the first index of $\mathbf{k}$. Then, one can derive from \eqref{eqn:rv:prodherbis}, \eqref{eq:herm-part} and the induction hypothesis that
\begin{align*}
\mu_{J,\mathbf{k}} &= H_a(g_{J,k_1}^\phi) \prod_{\ell=2}^p H_{n_\ell}(\mu_{J,\widetilde{k}_\ell}) = \left (\sum_{m=0}^{\lfloor a/2 \rfloor} (-1)^m a_m^{(d)} (g_{J,k_1}^\phi)^{d-2m}\right)  \prod_{\ell=2}^p H_{n_\ell}(\mu_{J,\widetilde{k}_\ell})\\
& = \left( \sum_{m=0}^{ \lfloor a/2 \rfloor} (-1)^m \sum_{P_1  \in \mathcal{P}_m^{(a)}} \prod_{r=1}^m \mathbb{E}[g_{J,k_{\ell_r}}^\phi g_{J,k_{\ell_r'}}^\phi] \prod_{s=m+1}^{a-m} g_{J,k_{\ell_{s}''}}^\phi \right)\times\ldots \\ & \qquad  \ldots\times \left( \sum_{n=0}^{ \lfloor (d-a)/2 \rfloor} (-1)^n \sum_{P_2  \in \mathcal{P}_n^{[\![a+1,d]\!]}} \,\prod_{t=1}^n \mathbb{E}[g_{J,k_{\ell_t}}^\phi g_{J,k_{\ell_t'}}^\phi] \prod_{u=n+1}^{d-a-n} g_{J,k_{\ell_{u}''}}^\phi \right) \\
& = \sum_{v=0}^{\lfloor d/2 \rfloor} (-1)^v \sum_{m,n \, : \, m+n=v} \left( \sum_{P  \in \mathcal{P}_{v,[m,n]}^{(d,a)}} \prod_{r=1}^v \mathbb{E}[g_{J,k_{\ell_r}}^\phi g_{J,k_{\ell_r'}}^\phi] \prod_{s=v+1}^{d-v} g_{J,k_{\ell_{s}''}}^\phi  \right),
\end{align*}
where $\mathcal{P}_{v,[m,n]}^{(d,a)}$ is the subset of $\mathcal{P}_v^{(d)}$ of the partitions of $[\![1,d]\!]$ with $m$ (non ordered) pairs of integers in $[\![1,a]\!]$ and $n$ (non ordered) pairs of integers in $[\![a+1,d]\!]$; notice that when $\lfloor a/2 \rfloor + \lfloor (d-a)/2 \rfloor<m+n\le \lfloor d/2 \rfloor$ then $\mathcal{P}_{v,[m,n]}^{(d,a)}$ becomes an empty set, therefore the sum over it reduces to zero.

Finally notice that when $P' \in \mathcal{P}_v^{(d)}$ is a partition with at least a (non ordered) pair $\{\ell,\ell'\}$ such that $\ell \in [\![1,a]\!]$ and $\ell' \in[\![a+1,d]\!]$, then $\mathbb{E}[g_{J,k_\ell}^\phi g_{J,k_{\ell'}}^\phi]=0$, thus, using the fact that $\mathcal{P}_{v,[m',n']}^{(d,a)}\cap \mathcal{P}_{v,[m'',n'']}^{(d,a)}=\emptyset$ when $(m',n')\ne (m'',n'')$, one gets that
\begin{align*}
\sum_{v=0}^{\lfloor d/2 \rfloor} & (-1)^v \sum_{m,n \, : \, m+n=v} \left( \sum_{P  \in \mathcal{P}_{v,[m,n]}^{(d,a)}} \prod_{r=1}^v \mathbb{E}[g_{J,k_{\ell_r}}^\phi g_{J,k_{\ell_r'}}^\phi] \prod_{s=v+1}^{d-v} g_{J,k_{\ell_{s}''}}^\phi  \right) \\ & = \sum_{v=0}^{\lfloor d/2 \rfloor} (-1)^v \sum_{P \in \mathcal{P}_v^{(d)}} \prod_{r=1}^v \mathbb{E}[g_{J,k_{\ell_r}}^\phi g_{J,k_{\ell_r'}}^\phi] \prod_{s=v+1}^{d-v} g_{J,k_{\ell_{s}''}}^\phi,
\end{align*}
which shows that \eqref{rv:formpourapprox} is valid.
\end{proof}

We are now in position to prove Proposition \ref{prop:gen-farima}.

\begin{proof}[Proof of Proposition \ref{prop:gen-farima}] Combining Lemma \ref{prop:rv:formpourapprox} with Remark \ref{rem1:farima}, it can easily be shown that, when the integer $n$ goes to $+\infty$, the partial sum of order $n$ of the random series in \eqref{eq:ext-farima}, that is the $d$th order Wiener chaos random variable 
\[
\sum_{\textbf{p}\in[\![0,n]\!]^d}\Big (\prod_{l=1}^d \gamma_{p_l}^{(h_l-1/2)}\Big)\mu_{J,\mathbf{k}-\mathbf{p}}\,,
\] 
converges almost surely to 
\[
\sum_{m=0}^{\lfloor d/2 \rfloor } (-1)^m \sum_{P \in \mathcal{P}_m^{(d)}} \prod_{r=1}^m \mathbb{E}[ Z_{J,k_{\ell_r}}^{(h_{\ell_r}-1/2)}Z_{J,k_{\ell_r'}}^{(h_{\ell_r'}-1/2)}] \prod_{s=m+1}^{d-m} Z_{J,k_{\ell_s}}^{(h_{\ell_s}-1/2)}.
\]
The fact that the convergence also holds in $L^\gamma (\Omega)$, for any $\gamma\in (0,+\infty)$, can be derived from a general result in \cite{MR1474726} according to which any sequence of random variables belonging to a finite order Wiener chaos converges in $L^\gamma (\Omega)$ as soon as it converges in probability.
\end{proof}

The following theorem, which provides, for $\|\cdot \|_{I,\infty}$ the uniform norm on any compact interval $I \subset \R_+$, an almost sure estimate of the error stemming from the approximation of $\{X_{\mathbf{h}}^{(d)}(t)\}_{t\in I}$ by $\{X_{\mathbf{h},J}^{(d)}(t)\}_{t\in I}$ is the second main result of our article. This theorem will be proved in Section \ref{sect:approxerro}.

\begin{Thm}\label{thm:main1}
For any compact interval $I \subset \R_+$, there exists an almost surely finite random variable $C$ (depending on $I$) for which one has, almost surely, for each $J \in \N$,
\begin{equation}\label{eqn:thm:main1}
\| X^{(d)}_{\mathbf{h}}-X^{(d)}_{\mathbf{h},J} \|_{I,\infty} = \| X^{(d,\bot)}_{\mathbf{h},J} \|_{I,\infty} \leq C J^{\frac{d}{2}} 2^{-J(  h_1+ \cdots + h_d -d + 1/2)}.
\end{equation}
\end{Thm}

Before stating the third and the last main result of our article, let us explain the motivation behind it. As the collection of functions $ \left \{\psi_{\mathbf{j},\mathbf{k}} \, : \, \mathbf{j},\mathbf{k} \in \Z^d  \right\}$
is an orthonormal basis of $L^2(\R^d)$, one can also wish to give a random series representation for the generalized Hermite process $\{X_{\mathbf{h}}^{(d)}(t)\}_{t\in \R_+}$ using this basis. Indeed, similarly to \eqref{eq:X-XJ}, it can be shown that 
\begin{equation} \label{eqn:fullserie}
X^{(d)}_{\mathbf{h}}(t)=\sum_{(\mathbf{j},\mathbf{k}) \in (\Z^d)^2} \varepsilon_{\mathbf{j},\mathbf{k}} \mathcal{K}^{(d,\mathbf{h})}_{\textbf{j},\textbf{k}}(t),
\end{equation}
where the random series is unconditionally convergent in $L^2(\Omega)$, for each fixed $t \in \R_+$. Roughly speaking, our third main result shows that when the partial sums of the random series in \eqref{eqn:fullserie} are well-chosen, then its convergence  holds in a much stronger sense: almost surely for the uniform norm $\|\cdot \|_{[0,T],\infty}$, where the fixed real number $T>2$ is arbitrary. Also, our third main result provides an almost sure estimate of the rate of convergence of the series for the the uniform norm $\|\cdot \|_{[0,T],\infty}$. In order to precisely explain how the partial sums have to be chosen, we need the following definition:

\begin{Def}\label{def:forthm2}
Let $T>2$, $b>0$, $b'>0$ and $g>0$ be four fixed arbitrary real numbers. For all $N \in \N$, we define the two disjoint finite subsets of $(\Z^d)^2$
\[ \mathcal{S}_N^{+} := \{(\mathbf{j},\mathbf{k}) \in (\Z^d)^2 \, : \, -2^{Nb} \leq \min_{\ell \in [\![1,d]\!]} j_\ell, \, 0 \leq \max_{\ell \in [\![1,d]\!]} j_\ell < N, \, \max_{\ell \in [\![1,d]\!]}|k_\ell| \leq 2^{N+1}T \} \] 
and
\[ \mathcal{S}_N^{-} := \{(\mathbf{j},\mathbf{k}) \in (\Z^d)^2 \, : \, -2^{Nb'} \leq \min_{\ell \in [\![1,d]\!]} j_\ell \leq \max_{\ell \in [\![1,d]\!]} j_\ell <0, \max_{\ell \in [\![1,d]\!]}|k_\ell| \leq 2^{Ng} \}. \]
\end{Def}

We are now in position to state our third and last main result.

\begin{Thm}\label{thm:main2}
Let $T>2$, $b>0$, $b'>0$ and $g>0$ be four fixed arbitrary real numbers. For all $t \in \R^+$ and $N \in \N$, let $\widetilde{X}^{(d)}_{\mathbf{h},N}(t)$ be the $d$th order Wiener chaos random variable defined by
\begin{equation}
\label{anteq1:thm:main2}
\widetilde{X}^{(d)}_{\mathbf{h},N}(t) := \sum_{(\mathbf{j},\mathbf{k}) \in \mathcal{S}_N^{+} \cup \mathcal{S}_N^{-}} \varepsilon_{\mathbf{j},\mathbf{k}} \mathcal{K}^{(d,\mathbf{h})}_{\textbf{j},\textbf{k}}(t).
\end{equation}
There exists an almost surely finite random variable $C$ (depending on $T, b, b', g$) for which one has, almost surely, for all $N \in \N$,
\begin{equation}\label{eqn:toshowinthm2}
\| X^{(d)}_{\mathbf{h}}-\widetilde{X}^{(d)}_{\mathbf{h},N}\|_{[0,T],\infty} \leq C N^{\frac{d}{2}} 2^{-N(  h_1+ \cdots + h_d -d + 1/2)}.
\end{equation}
\end{Thm}

To prove Theorems \ref{thm:main1} and \ref{thm:main2}, we will need a logarithmic bound for the the sequence of random variables $(\varepsilon_{\mathbf{j},\mathbf{k}})_{(\mathbf{j},\mathbf{k}) \in (\Z^d)^2}$. We get it from the following lemma which is is a straightforward consequence of Lemma 2 in \cite{MR2027888} and of the fact that the 
$g_{j,k}^\psi:=I_1 (\psi_{j,k})$, $(j,k)\in\Z^2$, are $\mathcal{N}(0,1)$ Gaussian random variables.

\begin{Lemma}
\label{lem:bound-gjk}
There are $\Omega^*$ an event of probability $1$ and $C_1^*$ a positive random variable of finite moment of any order, such that, for all $\omega \in \Omega^*$ and for each $(j,k) \in \Z^2$, one has
\begin{equation}\label{eqn:boundn01}
\big| g_{j,k}^{\psi}(\omega)\big| \leq C_1^*(\omega) \sqrt{\log\big (3+|j|+|k|\big)}.
\end{equation}
\end{Lemma} 
Next, observe that, for any $n \in \N$, there exists a constant $a_n >0$ such that, for all $x \in \R$
\begin{equation}\label{rem:hermP:eq1}
|H_n(x)| \le a_n \big(1+|x|^n\big)\,;
\end{equation}
the latter inequality is a straightforward consequence of the fact that $H_n$ is a polynomial function of degree $n$. Then, combining \eqref{eqn:rv:prodher} with \eqref{eqn:boundn01} and \eqref{rem:hermP:eq1}, one obtains the following lemma.

\begin{Lemma}
\label{lem:eqn:bound:rv}
Let $\Omega^*$ be the same event of probability $1$ as in Lemma~\ref{lem:bound-gjk}. There is
$C_d^*$ a positive random variable of finite moment of any order, such that, on $\Omega^*$, one has, for all $(\mathbf{j},\mathbf{k}) \in (\Z^2)^d$,
\begin{equation}\label{eqn:bound:rv}
|\varepsilon_{\mathbf{j},\mathbf{k}}| \leq C_d^* \prod_{\ell=1}^{p(\mathbf{j},\mathbf{k})}  \left(\sqrt{\log(3+|\widetilde{j}_\ell|+|\widetilde{k}_\ell|}\right)^{n_\ell} = C_d^* \prod_{m=1}^d \sqrt{\log(3+|j_m|+|k_m|)}.
\end{equation}
\end{Lemma}

To prove Theorems \ref{thm:main1} we will also need to know precisely when two random variables $\varepsilon_{\mathbf{j},\mathbf{k}}$ and $\varepsilon_{\mathbf{r},\mathbf{s}}$ are correlated. For this purpose, it is useful to define the set $\calD(\mathbf{j},\mathbf{k})$.

\begin{Def}
\label{def:mutili-set}
Using the same notations as in \eqref{eqn:rv:prodher}, for all $(\mathbf{j},\mathbf{k}) \in (\Z^d)^2$, the set 
$\calD(\mathbf{j},\mathbf{k})$ is defined as:
\begin{equation}
\label{def:mutili-set:eq1}
\calD(\textbf{j},\textbf{k}):=\big\{(\widetilde{j_\ell},\widetilde{k_\ell})_{n_l}: 1\le l\le p(\mathbf{j},\mathbf{k})\big\}\,. 
\end{equation}
\end{Def}

\begin{Rmk}
\label{rem:eqa-multi-jk}
For any arbitrary two elements $(\mathbf{j},\mathbf{k})=\big ((j_m ,k_m)\big)_{1\le  m \le d}$ and $(\mathbf{r},\mathbf{s})=\big ((r_m ,s_m)\big)_{1\le  m \le d}$ of $(\Z^d)^2$, a necessary and sufficient condition for having $\calD(\textbf{j},\textbf{k})=\calD(\textbf{r},\textbf{s})$ is that there exists a permutation $\sigma$ of the set $\{1,\ldots, d\}$ for which one has $(j_m ,k_m)=(r_{\sigma(m)} , s_{\sigma(m)})$, for all $m\in\{1,\ldots, d\}$.  Thus, being given an arbitrary element $(\mathbf{j},\mathbf{k})$ of $(\Z^d)^2$, there are at most $d\,!-1$ other elements $(\mathbf{r},\mathbf{s})$ of $(\Z^d)^2$ which satisfy $\calD(\textbf{j},\textbf{k})=\calD(\textbf{r},\textbf{s})$. Notice that, in this case, as a consequence of equality \eqref{eqn:rv:prodher}, one has $\varepsilon_{\mathbf{j},\mathbf{k}} =\varepsilon_{\mathbf{r},\mathbf{s}}$.
\end{Rmk}

Let us also recall that, if $G$ is any arbitrary $\mathcal{N}(0,1)$ Gaussian random variable then, one has 
\begin{equation}\label{eqn:esp:prodher}
\E\big [H_m(G)H_n(G)\big] = \delta_{m,n} m!\,, \quad\mbox{for any $m,n \in \Z_+$,}
\end{equation}
where $\delta_{m,n}=1$ when $m=n$ and $\delta_{m,n}=0$ otherwise. A straightforward consequence of \eqref{eqn:esp:prodher} is that 
\begin{equation}\label{eqn:esp:her}
\E\big[H_n(G)\big]=0\,, \quad\mbox{for all integer $n\ge 1$.}
\end{equation}
Relation \eqref{eqn:esp:prodher} is the keystone of the proof of the following proposition.

\begin{Prop}\label{prop:corre}
For every $(\mathbf{j},\mathbf{k}) \in (\Z^2)^d$ and $(\mathbf{r},\mathbf{s}) \in (\Z^2)^d$, one has 
\begin{equation}
\label{prop:corre:eq1}
\E [\varepsilon_{\mathbf{j},\mathbf{k}}\varepsilon_{\mathbf{r},\mathbf{s}}]=
\left\{
\begin{array}{l}
\E [\varepsilon_{\mathbf{j},\mathbf{k}}^2]=\displaystyle\prod_{\ell=1}^{p(\mathbf{j},\mathbf{k})} n_l !\le d\, !\quad\mbox{if $\calD(\textbf{j},\textbf{k})=\calD(\textbf{r},\textbf{s})$,}\\
\\
0\quad\mbox{otherwise.}
\end{array}
\right.
\end{equation}
\end{Prop}

\begin{proof} First notice that, in view of Remark \ref{rem:eqa-multi-jk}, \eqref{eqn:rv:prodher}, the independence of the $\mathcal{N}(0,1)$ Gaussian random variables $g_{\,\widetilde{j}_\ell,\widetilde{k}_\ell}^\psi$ with ${ \ell \in\{1,\ldots, p(\mathbf{j},\mathbf{k})\}}$ and \eqref{eqn:esp:prodher}, the equality
\eqref{prop:corre:eq1} is clearly satisfied when ${\calD(\textbf{j},\textbf{k})=\calD(\textbf{r},\textbf{s})}$. So, from now on, one assumes that $\calD(\textbf{j},\textbf{k})=\{(\widetilde{j_1},\widetilde{k_1})_{n_1},\ldots,(\widetilde{j_p},\widetilde{k_p})_{n_p}\}$ (where $p=p(\textbf{j},\textbf{k})$) is not equal to $\calD(\textbf{r},\textbf{s})=\{(\widetilde{r_1},\widetilde{s_1})_{m_1},\ldots,(\widetilde{r_q},\widetilde{s_q})_{m_q}\}$ (where $q=p(\textbf{r},\textbf{s})$) which happens in two different cases.

The first case consists in the situation where one has \\ $\{(\widetilde{j_1},\widetilde{k_1}),\ldots,(\widetilde{j_p},\widetilde{k_p})\} \neq \{(\widetilde{r_1},\widetilde{s_1}),\ldots,(\widetilde{r_q},\widetilde{s_q})\}$, which implies that there exists at least one element of these two sets which does not belong to the other set. For sake of simplicity, one assumes that $(\widetilde{j_1},\widetilde{k_1}) \notin \{(\widetilde{r_1},\widetilde{s_1}),\ldots,(\widetilde{r_q},\widetilde{s_q})\}$. Then, using \eqref{eqn:rv:prodher}, the fact that the $\mathcal{N}(0,1)$ Gaussian random variable $g^\psi_{\,\widetilde{j}_1,\widetilde{k}_1}$ is independent of the Gaussian vector $\big (g^\psi_{\,\widetilde{j}_2,\widetilde{k}_2},\ldots, g^\psi_{\,\widetilde{j}_p,\widetilde{k}_p}, g^\psi_{\,\widetilde{r}_1,\widetilde{s}_1},\ldots , g^\psi_{\,\widetilde{r}_q,\widetilde{s}_q}\big)$, and \eqref{eqn:esp:her}, one gets that
\[ 
\E[ \varepsilon_{\mathbf{j},\mathbf{k}} \varepsilon_{\mathbf{r},\mathbf{s}}] = \underset{=0}{\underbrace{\E\big [H_{n_1}(g^\psi_{\,\widetilde{j}_1,\widetilde{k}_1})\big]}} \E \left[ \prod_{\ell=2}^p H_{n_\ell} (g^\psi_{\,\widetilde{j}_\ell,\widetilde{k}_\ell}) \prod_{\ell'=1}^q H_{m_{\ell'}} (g^\psi_{\,\widetilde{j}_{\ell'},\widetilde{k}_{\ell'}})\right]=0. 
\]

The second case consists in the situation where one has $p=q$, \\
$\{(\widetilde{j_1},\widetilde{k_1}),\ldots,(\widetilde{j_p},\widetilde{k_p})\} =\{(\widetilde{r_1},\widetilde{s_1}),\ldots,(\widetilde{r_p},\widetilde{s_p})\}$ and $n_{\ell_0}\ne m_{\ell_0}$ for some $\ell_0\in\{1,\ldots, p\}$. For sake of simplicity, one assumes that $\ell_0=1$. Then, using \eqref{eqn:rv:prodher}, the fact that the $\mathcal{N}(0,1)$ Gaussian random variable $g^\psi_{\,\widetilde{j}_1,\widetilde{k}_1}$ is independent of the Gaussian vector $\big (g^\psi_{\,\widetilde{j}_2,\widetilde{k}_2},\ldots, g^\psi_{\,\widetilde{j}_p,\widetilde{k}_p}\big)$, and \eqref{eqn:esp:prodher}, one obtains that
\[ 
\E[ \varepsilon_{\mathbf{j},\mathbf{k}} \varepsilon_{\mathbf{r},\mathbf{s}}] = \underset{=0}{\underbrace{\E\big [H_{n_1}(g^\psi_{\,\widetilde{j}_1,\widetilde{k}_1}) H_{m_1}(g^\psi_{\,\widetilde{j}_1,\widetilde{k}_1})\big]}} \E \left[ \prod_{\ell=2}^p H_{n_\ell} (g^\psi_{\,\widetilde{j}_{\ell},\widetilde{k}_{\ell}}) H_{m_\ell} (g^\psi_{\,\widetilde{j}_\ell,\widetilde{k}_\ell})\right]=0. 
\]
\end{proof}

\section{Proof of Theorem \ref{thm:basefreq}} \label{sect:approxproc}

In this section, we aim at proving Theorem \ref{thm:basefreq}. First, we need to focus on the sequence $(\gamma_p^{(\delta)})_{p \in \N_0}$, $\delta \in (-1/2,1/2)$, used in \eqref{eqn:farima} to define the FARIMA $(0,\delta,0)$ sequence. We recall that, if $\delta \neq 0$, the continuous function defined for all complex number $z \notin [1,+ \infty)$
\[ F_\delta \,: \, z \mapsto (1-z)^{-\delta} \]
is analytic on the disk $|z|<1$ with Taylor expansion given by
\begin{equation}\label{eqn:analytic}
F_\delta(z) = \sum_{p=0}^{+ \infty} \gamma_p^{(\delta)} z^p.
\end{equation}
In the sequel, for any $\delta \neq 0$, and $\xi \in \R \setminus 2 \pi \Z$, using the continuity property of the function  $F_\delta$ the quantity $(1-e^{i \xi})^{-\delta}=F_\delta(e^{i \xi})$ is expressed as:
\begin{equation}
\label{eq:Fexi}
(1-e^{i \xi})^{-\delta}=F_\delta(e^{i \xi})=\lim_{r\in\R,\,r \to 1^{-}} F_\delta(r e^{i \xi})=\lim_{r\in\R,\,r \to 1^{-}}(1-re^{i \xi})^{-\delta}.
\end{equation}

For the sake of convenience, for all $p \in \N$, we set $\gamma_{-p}^{(\delta)}=0$. The following lemma shows that the sequence $(\gamma_p^{(\delta)})_{p \in \Z}$ is nothing else than the sequence of the Fourier coefficient of the function $\xi\mapsto (1-e^{i \xi})^{-\delta}$ which belongs to $L^2 ([0,2\pi])$. Recall that $L^2 ([0,2\pi])$ is the space of the complex-valued functions defined on the real line which are $2 \pi$-periodic and Lebesgue square-integrable on the interval $[0,2 \pi]$.

\begin{Lemma} \label{lemma:fouriercoef}
For all $\delta \in (0,1/2)$ and $p \in \Z$, we have
\[ \gamma_p^{(\delta)} = \frac{1}{2 \pi} \int_0^{2 \pi} e^{-i p \xi} (1-e^{i \xi})^{-\delta} \, d\xi. \]
\end{Lemma}

\begin{proof}
Let $(r_j)_{j\in\N}$ be an arbitrary increasing sequence of real numbers in the open interval $(0,1)$ which converges to $1$. In view of \eqref{eq:Fexi}, we have
\[  \int_0^{2 \pi} e^{-i p \xi} (1-e^{i \xi})^{-\delta} \, d\xi = \int_0^{2 \pi} e^{-i p \xi} \lim_{r_j \to 1^{-}} (1-r_j e^{i \xi})^{-\delta} \, d\xi.\]
Let us first show that one can permute the limit and integration symbols. To this end, for all $j$, let us consider the subset of $[0,2 \pi]$
\[ A_j := \left \{ \xi \in [0, 2 \pi ] \, : \, |1-e^{i \xi}| \leq 2(1-r_j) \right \}. \]
Note that, for all $j$ large enough, if $\xi \in A_j$, then $$\xi \in [0,4(1-r_j)] \cup [2\pi-4(1-r_j),2 \pi].$$ Therefore, we can derive from the inequality $|1-r_j e^{i \xi} | \geq (1-r_j)$ that
\begin{align*}
\left | \int_{A_j} e^{-i p \xi}  (1-r_j e^{i \xi})^{-\delta} \, d\xi \right| & \leq (1-r_j)^{- \delta} \int_{A_{j}} \mathbbm{1}_{A_j}(\xi) \, d\xi \\
& \leq 8(1-r_j)^{1- \delta}.
\end{align*}
Since $1-\delta>0$, the latter inequality entails that 
\[ \lim_{j\rightarrow +\infty} \int_{A_j} e^{-i p \xi}  (1-r_j e^{i \xi})^{-\delta} \, d\xi=0. \]
Let us now assume that $\xi \in A_j^c:=[0,2\pi]\setminus A_j$, then we have that
\begin{align*}
|1-r_j e^{i \xi}| = |1-e^{i\xi} +(1-r_j)e^{i\xi}| \geq |1-e^{i \xi}|-(1-r_j) >\frac{1}{2}|1-e^{i \xi}|,
\end{align*}
which implies that
\begin{align*}
\left| e^{-i p \xi}  (1-r_j e^{i \xi})^{-\delta} \right| < 2^\delta |1-e^{i\xi}|^{-\delta}.
\end{align*}
As $\delta \in (0,1/2)$, the function $\xi \mapsto |1-e^{i\xi}|^{-\delta}$ is integrable on $[0,2 \pi]$, and since, for all $\xi \in (0,2 \pi)$,  $\mathbbm{1}_{A_j^c}(\xi) \to 1$, we conclude, by dominated convergence theorem, that
\[  \int_0^{2 \pi} e^{-i p \xi} (1-e^{i \xi})^{-\delta} \, d\xi =\lim_{r_j \to 1^{-}}  \int_0^{2 \pi} e^{-i p \xi} (1-r_j e^{i \xi})^{-\delta} \, d\xi.\]

Now, if $j$ is fixed, using \eqref{eqn:analytic}, we have
\begin{align*}
\int_0^{2 \pi} e^{-i p \xi} (1-r_j e^{i \xi})^{-\delta} \, d\xi &= \sum_{m=0}^{+ \infty} \gamma_m^{(\delta)} r_j^m \int_0^{2 \pi} e^{-ip\xi} e^{im \xi} \, d\xi = 2\pi r_j^p \gamma_p^{(\delta)} .
\end{align*}
The conclusion follows immediately.
\end{proof}

The following corollary is a straightforward consequence of Lemma \ref{lemma:fouriercoef} and a fundamental result of Fourier analysis.
\begin{Cor}\label{Cor:fouriercoef}
For all $\delta \in (0,1/2)$ we have
\[ (1-e^{i\xi})^{-\delta}=\sum_{p=0}^{+ \infty} \gamma_p^{(\delta)} e^{ip \xi},\]
with converge in $L^2([0,2\pi])$.
\end{Cor}

\begin{Rmk}
Let us emphasize that Corollary \ref{Cor:fouriercoef} shows that the expectations involved in the expression \eqref{eqn:defofsigma} of the random variables $\sigma_{J,\mathbf{k}}$, $(J,\mathbf{k}) \in \Z \times \Z^d$, are easily computable. Indeed, for all $J \in \Z$ and $k,p,k',p' \in \Z$, the expectation $\E [g^\phi_{J,k-p},g^\phi_{J,k'-p'} ]$ does not vanish only when $k-p=k'-p'$ and, in this case, it is equal to $1$. Thus, we can write
\begin{align*}
\E [g^\phi_{J,k-p},g^\phi_{J,k'-p'} ] &= \frac{1}{2\pi}\int_0^{2\pi} e^{i(k-p) \xi} e^{-i(k'-p') \xi} \, d\xi. 
\end{align*}
Then, using Definition \ref{def:farima} and Remark \ref{rem1:farima}, we obtain, for all $\delta, \delta'\in (0,1/2)$, $J \in \N$ and $k,k' \in \Z$, that
\begin{align*}
\E[Z_{j,k}^{(\delta)}Z_{j,k'}^{(\delta')}] & = \frac{1}{2\pi}\int_0^{2\pi} \left( \sum_{p \in \Z} \gamma_p^{(\delta)} e^{i(k-p) \xi} \right) \left(\sum_{p' \in \Z} \gamma_{p'}^{(\delta')}e^{-i(k'-p') \xi} \right) \, d\xi \\
&= \frac{1}{2\pi}\int_0^{2\pi}  e^{i(k-k')\xi} \left( \sum_{p \in \Z} \gamma_p^{(\delta)} e^{-i p \xi} \right)  \left(\sum_{p' \in \Z} \gamma_{p'}^{(\delta')}e^{ip' \xi} \right) \, d\xi \\
& =\frac{1}{2\pi}\int_0^{2\pi} e^{i(k-k')\xi} (1-e^{-i\xi})^{-\delta} (1-e^{i\xi})^{-\delta'} \, d\xi.
\end{align*}
In particular, if $\delta=\delta'$, a fact that always occurs when we restrict to standard Hermite processes, we get that
\begin{align*}
\E[Z_{j,k}^{(\delta)}Z_{j,k'}^{(\delta)}] &=\frac{1}{2\pi}\int_0^{2\pi} e^{i(k-k')\xi} |1-e^{-i\xi}|^{-2 \delta}  \, d\xi \\
&=\frac{1}{2\pi}\int_0^{2\pi} e^{i(k-k')\xi} \left|2\sin \left(\frac{\xi}{2}\right)\right|^{-2 \delta}  \, d\xi.
\end{align*}
\end{Rmk}

Let us now recall that $\{\phi( \cdot -k) \, : \, k \in \Z \}$ is an orthonormal basis of the subspace $V_0^1$ of the multiresolution analysis of $L^2(\R)$ of the univariate Meyer scaling function $\phi$. Therefore (see e.g. \cite{MR1162107}), an arbitrary function $f$ of $L^2(\R)$ belongs to its subspace $V_0^1$ if and only if there exists a unique function $m_f\in L^2([0,2 \pi])$, such that the Fourier transform of $f$, satisfies in $L^2(\R)$ (that is for almost all $\xi\in\R$) the equality $\widehat{f}(\xi)= m_f (\xi)\widehat{\phi}(\xi)$.

\begin{Def}\label{def:fouphidelta}
For all $\delta \in (0,\frac{1}{2})$, $\Phi^{(-\delta)}$ is the $V_0^1$ function given by
\[ \widehat{\Phi}^{(-\delta)}(\xi)=(1-e^{i\xi})^{-\delta}\,\widehat{\phi}(\xi),\]
where the equality holds in $L^2(\R)$.
\end{Def}

\begin{Lemma}\label{Lemma:expphidelta}
For all $\delta \in (0,\frac{1}{2})$, we have
\begin{equation}\label{eqn:expofphidelta}
\Phi^{(-\delta)}(x) = \sum_{p=0}^{+ \infty} \gamma_p^{(\delta)} \phi(x+p)
\end{equation}
with convergence in $L^2(\R)$.
\end{Lemma}
\begin{proof}
As $\delta \in (0,\frac{1}{2})$, equation \eqref{eqn:stirling} entails that $(\gamma_p^{(\delta)})_{p \in \Z}$ is a $\ell^2(\Z)$ sequence. Thus the right-hand side of \eqref{eqn:expofphidelta} is a well-defined $L^2(\R)$ function with Fourier transform given by
\[ \sum_{p=0}^{+ \infty} \gamma_p^{(\delta)} e^{ip \xi} \widehat{\phi}(\xi).\]
The conclusion follows from Corollary \ref{Cor:fouriercoef}.
\end{proof}

Let us come back to generalized Hermite processes. We start by giving,for all $J \in \N$ and $t\in\R_+$, a $L^2(\R^d)$ expansion of the the kernel function $K^{(d)}_{\mathbf{h},J}(t,\bullet)$ (see \eqref{eqn:approxprocess}) using the following functions $\widetilde{a}_{J,k}^{(\delta)}$.

\begin{Def}\label{Def:coeajdelta}
For all $J \in \N$, $k \in \Z$ and $\delta\in (0,1/2)$ and $s\in\R_+$, we define the function
\[ \widetilde{a}_{J,k}^{(\delta)}(s) := \int_\R e^{i(2^Js-k)\xi}(1-e^{-i \xi})^{-\delta}\,\widehat{\Phi}^{(\delta)}_\Delta(\xi) \, d\xi,\]
where the fractional scaling function $\Phi^{(\delta)}_\Delta  \in \mathcal{S}(\R)$ is as in Definition \ref{def:fractscal}.
\end{Def}

\begin{Lemma} \label{lemma:firstexp}
For all $t \in \R_+$ and $J \in \N$, we have
\begin{align*}
&K^{(d)}_{\mathbf{h},J}(t,\mathbf{u}) \\ &= (2 \pi)^{-d/2}\sum_{\mathbf{k} \in \Z^d} 2^{-J(h_1+\cdots+h_d-d)}  \left( \int_0^t \prod_{\ell=1}^d \widetilde{a}_{J,k_\ell}^{(h_\ell-1/2)}(s) \, ds\right) 2^{J \frac{d}{2}} \prod_{\ell=1}^d \phi(2^J u_\ell-k_\ell),
\end{align*}
where the convergence holds in $L^2(\R^d)$, with respect to $\mathbf{u}$.
\end{Lemma}

\begin{proof}
We know that, for all $t \in \R_+$ and $J \in \N$, we have
\begin{align*}
&K^{(d)}_{\mathbf{h},J}(t,\mathbf{u})   = \sum_{\mathbf{k} \in \Z^d} \mathfrak{K}^{(d,\mathbf{h})}_{J,\textbf{k}}(t) \, 2^{J\frac{d}{2}}\prod_{\ell=1}^d \phi(2^J u_\ell-k_\ell),
\end{align*}
where $\mathfrak{K}^{(d,\mathbf{h})}_{J,\textbf{k}}(t)$ is as in \eqref{eqn:coefk2bis} and where the convergence holds in $L^2(\R^d)$ with respect to $\mathbf{u}$. Moreover, using Definition \ref{def:fractscal}, \eqref{eqn:coefk2bis-f} and the inverse Fourier transform, we get, for all $\ell \in [\![1,d]\!]$ and almost every $s \in \R$, that
\[ \phi_{h_\ell}(2^J s-k_\ell)  = \frac{1}{\sqrt{2\pi}} \int_\R e^{i(2^Js-k_\ell)\eta} (1-e^{-i\xi})^{-(h_\ell-1/2)} \widehat{\Phi}_\Delta^{(h_\ell-1/2)}(\xi) d\xi.\] 
Then combining the latter equality with \eqref{eqn:coefk2bis} and Definition \ref{Def:coeajdelta}, we obtain the lemma.
\end{proof}

In the sequel, for simplifying notations, we sometimes omit in them to make reference to the fixed indices $J \in \N$ and $t\in (0,\infty)$.  The previous lemma leads us the following definition:

\begin{Def}\label{def:alpha}
For all $\mathbf{k} \in \Z^d$, we define the coefficient
\[ \alpha_{\textbf{k}}^{(\mathbf{h})} := \int_0^t \prod_{\ell=1}^d \widetilde{a}_{J,k_\ell}^{(h_\ell-1/2)}(s) \, ds.\]
\end{Def}

\begin{Lemma}
For all $\mathbf{k} \in \Z^d$, the following equality holds:
\begin{equation}\label{eqn:avantconvo}
\alpha_{\textbf{k}}^{(\mathbf{h})} = (2 \pi)^\frac{d}{2}\sum_{\mathbf{q} \in \N_0^d} \int_0^t \prod_{\ell =1}^d \gamma_{q_\ell}^{(h_\ell-1/2)} \Phi_\Delta^{(h_\ell-1/2)}(2^J s - k_\ell-q_\ell) \, ds.
\end{equation}
\end{Lemma}

\begin{proof}
First, observe that one can easily derive from Definitions \ref{Def:coeajdelta} and \ref{def:fractscal}, Corollary \ref{Cor:fouriercoef}, Cauchy-Schwarz inequality and elementary properties of Fourier transform that, for all $\ell \in [\![1,d]\!]$, $k_l\in\Z$ and $s\in\R_+$,
\[ \sqrt{2 \pi}\sum_{q_\ell \in \N_0} \gamma_{q_\ell}^{(h_\ell-1/2)} \Phi_\Delta^{(h_\ell-1/2)}(2^J s - k_\ell-q_\ell)= \widetilde{a}_{J,k_\ell}^{(h_\ell-1/2)}(s).\]
Then the lemma can be derived from the dominated convergence theorem. Notice that the latter theorem can be used since the fact that the functions $\Phi_\Delta^{(h_\ell-1/2)}$, $\ell \in [\![1,d]\!]$, belong to Schwartz class $\mathcal{S}(\R)$ implies that 
\[  \sup_{s\in [0,t]}\left\{\sum_{\mathbf{q} \in \N_0^d} \prod_{\ell =1}^d \gamma_{q_\ell}^{(h_\ell-1/2)} \big|\Phi_\Delta^{(h_\ell-1/2)}(2^J s - k_\ell-q_\ell)\big|\right\} < \infty.\]
\end{proof}

\begin{Rmk}
In order to provide a more synthetic and convenient expression for the right-hand side of the equality \eqref{eqn:avantconvo}, let us introduce the two sequences of real numbers $(\beta^{(\mathbf{h})}_{\mathbf{k}})_{\mathbf{k} \in \Z^d}$ and $(\Gamma^{(\mathbf{h})}_{\mathbf{k}})_{\mathbf{k} \in \Z^d}$ defined as:
\begin{equation}
\label{eq-seqbe}
\beta^{(\mathbf{h})}_{\mathbf{k}} := \int_0^t \prod_{\ell =1}^d \Phi_\Delta^{(h_\ell-1/2)}(2^J s - k_\ell) \, ds
\end{equation}
and
\begin{equation}
\label{eq-seqGa}
\Gamma^{(\mathbf{h})}_{\mathbf{k}} :=  \prod_{\ell =1}^d \gamma_{-k_\ell}^{(h_\ell-1/2)},
\end{equation}
with the convention that $\gamma_{-k_\ell}^{(h_\ell-1/2)}=0$ as soon as $-k_\ell<0$. Then, in view of \eqref{eqn:avantconvo}, it turns out that, up to the multiplicative factor $(2 \pi)^\frac{d}{2}$, the sequence $(\alpha_{\textbf{k}}^{(\mathbf{h})})_{\textbf{k}\in\Z^d}$ is nothing else than the convolution product of the two sequences $(\beta^{(\mathbf{h})}_{\mathbf{k}})_{\mathbf{k} \in \Z^d}$ and $(\Gamma^{(\mathbf{h})}_{\mathbf{k}})_{\mathbf{k} \in \Z^d}$, that is:
\begin{equation}\label{eqn:convo}
\alpha_{\textbf{k}}^{(\mathbf{h})}=  (2 \pi)^\frac{d}{2} \sum_{\mathbf{q} \in \Z^d} \beta^{(\mathbf{h})}_{\mathbf{k}-\mathbf{q}} \Gamma^{(\mathbf{h})}_{\mathbf{q}} = (2 \pi)^\frac{d}{2} (\beta^{(\mathbf{h})} \ast \Gamma^{(\mathbf{h})})_{\textbf{k}}.
\end{equation}
 Moreover, one can easily observe from \eqref{eq-seqbe} and the fact that the functions $\Phi_\Delta^{(h_\ell-1/2)}$, $\ell \in [\![1,d]\!]$, belong to Schwartz class, that one has, for any fixed $J\in\N$, $t\in\R_+$ and arbitrarily large $\mu>0$,
 \begin{equation}\label{eqn:fastdecayforbeta}
\sup_{\mathbf{k} \in \Z^d} \left( \prod_{\ell=1}^d (1+|k_\ell|)^\mu |\beta^{(\mathbf{h})}_{\mathbf{k}}| \right) < \infty,
\end{equation}
which in particular implies that $(\beta^{(\mathbf{h})}_{\mathbf{k}})_{\mathbf{k} \in \Z^d} \in \ell^1(\Z^d)$. On another hand, we know from \eqref{eq-seqGa} and \eqref{eqn:stirling} that $(\Gamma^{(\mathbf{h})}_{\mathbf{k}})_{\mathbf{k} \in \Z^d} \in \ell^2(\Z^d)$. Thus, one can derive from \eqref{eqn:convo} that the sequence $(\alpha^{(\mathbf{h})}_{\mathbf{k}})_{\mathbf{k} \in \Z^d}$ belongs to $\ell^2(\Z^d)$, and that its Fourier transform $\widehat{\alpha}^{(\mathbf{h})}$, namely the function in $L^2([0,2 \pi]^d)$ defined through the series
\begin{equation}
\widehat{\alpha}^{(\mathbf{h})} (\mathbf{\xi}) := \sum_{\mathbf{k} \in \Z^d} \alpha_{\textbf{k}}^{(\mathbf{h})} e^{-i \langle \mathbf{k}, \mathbf{\xi} \rangle}.
\end{equation}
which converge in $L^2([0,2 \pi]^d)$, satisfies 
\begin{equation}\label{eqn:expressionalpha}
\widehat{\alpha}^{(\mathbf{h})} (\mathbf{\xi}) = (2 \pi)^\frac{d}{2} \widehat{\beta}^{(\mathbf{h})} (\mathbf{\xi}) \widehat{\Gamma}^{(\mathbf{h})} (\mathbf{\xi}),\quad\mbox{for almost all $\xi\in\R^d$,}
\end{equation}
where $\widehat{\beta}^{(\mathbf{h})}\in L^\infty ([0,2 \pi]^d)$ and $\widehat{\Gamma}^{(\mathbf{h})}\in L^2([0,2 \pi]^d)$ respectively denote the Fourier transforms of the sequences $(\beta^{(\mathbf{h})}_{\mathbf{k}})_{\mathbf{k} \in \Z^d} \in \ell^1(\Z^d)$ and $(\Gamma^{(\mathbf{h})}_{\mathbf{k}})_{\mathbf{k} \in \Z^d} \in \ell^2(\Z^d)$. We mention that $\widehat{\beta}^{(\mathbf{h})}$ is defined through the series 
\begin{equation}\label{eqn:beta}
\widehat{\beta}^{(\mathbf{h})} (\mathbf{\xi}) := \sum_{\mathbf{k} \in \Z^d} \beta_{\textbf{k}}^{(\mathbf{h})} e^{-i \langle \mathbf{k}, \mathbf{\xi} \rangle},
\end{equation}
which converges uniformly in $\xi\in\R^d$, and $\widehat{\Gamma}^{(\mathbf{h})}$ is defined through the series 
\begin{equation}\label{eqn:exprgammabis}
\widehat{\Gamma}^{(\mathbf{h})} (\mathbf{\xi}) := \sum_{\mathbf{k} \in \Z^d} \Gamma_{\textbf{k}}^{(\mathbf{h})} e^{-i \langle \mathbf{k}, \mathbf{\xi} \rangle} = \prod_{\ell =1}^d  \left(  \sum_{k_\ell=0}^{+ \infty}\gamma_{k_\ell}^{(h_\ell-1/2)}e^{i k_\ell  \mathbf{\xi}_\ell }  \right),
\end{equation}
with convergence in $L^2([0,2 \pi]^d)$. Observe that it follows from \eqref{eqn:exprgammabis}, \eqref{eq-seqGa} and Corollary \ref{Cor:fouriercoef}
\begin{equation}\label{eqn:exprgamma}
\widehat{\Gamma}^{(\mathbf{h})} (\mathbf{\xi})=\prod_{\ell =1}^d (1-e^{i\xi_\ell})^{1/2-h_\ell}, \quad\mbox{for almost all $\xi\in\R^d$.}
\end{equation}
\end{Rmk}

\begin{Lemma}
The series
\begin{equation}\label{eqn:serievj}
 \mathcal{V}^{(\mathbf{h})}(\mathbf{u}) := \sum_{\mathbf{k} \in \Z^d} \beta_{\mathbf{k}}^{(\mathbf{h})} 2^{J \frac{d}{2}} \prod_{\ell=1}^d \Phi^{(1/2 -h_\ell)}(2^J u_\ell-k_\ell)
\end{equation}
is normally convergent in $L^2(\R^d)$. Thus the function $\mathcal{V}^{(\mathbf{h})}$ belongs to $L^2(\R^d)$. Moreover, its Fourier transform satisfies, for almost every $\xi \in \R^d$,
\begin{equation} \label{eqn:transfoserivj}
 \widehat{\mathcal{V}}^{(\mathbf{h})}(\mathbf{\xi}) =  2^{-J \frac{d}{2}} \widehat{\beta}^{(\mathbf{h})} (2^{-J}\mathbf{\xi}) \prod_{\ell=1}^d \widehat{\Phi}^{(1/2-h_\ell)}(2^{-J} \xi_\ell).
\end{equation}
\end{Lemma}

\begin{proof} The normal convergence of the series in \eqref{eqn:serievj} easily follows from \eqref{eqn:fastdecayforbeta} and the straightforward inequality, for all $\mathbf{k} \in \Z^d$, 
\[ \left \|  \beta_{\mathbf{k}}^{(\mathbf{h})} 2^{J \frac{d}{2}} \prod_{\ell=1}^d \Phi^{(1/2-h_\ell)}(2^J u_\ell-k_\ell) \right \|_{L^2(\R^d)} \leq \prod_{\ell=1}^d  \| \Phi^{(1/2-h_\ell)} \|_{L^2(\R^d)} |\beta_{\mathbf{k}}^{(\mathbf{h})} |. \]
Next, for all $N \in \N$, let $\mathcal{V}^{(\mathbf{h})}_N$ be the finite sum, defined, for each $\mathbf{u}\in\R^d$, as:
\[  \mathcal{V}^{(\mathbf{h})}_N(\mathbf{u}) := \sum_{|\mathbf{k}| \leq N} \beta_{\mathbf{k}}^{(\mathbf{h})} 2^{J \frac{d}{2}} \prod_{\ell=1}^d \Phi^{(1/2-h_\ell)}(2^J u_\ell-k_\ell).\]
We already know that $\mathcal{V}_N^{(\mathbf{h})} \to \mathcal{V}^{(\mathbf{h})}$ in $L^2(\R^d)$ as $N \to + \infty$. Therefore, we also have the convergence $\widehat{\mathcal{V}}_N^{(\mathbf{h})} \to \widehat{\mathcal{V}}^{(\mathbf{h})}$ in $L^2(\R^d)$ as $N \to + \infty$, which implies the existence of a subsequence $(N_r)_{r \in \N}$ such that
\begin{equation}\label{eqn:convvj}
\lim_{r\rightarrow+\infty}\widehat{\mathcal{V}}_{N_r}^{(\mathbf{h})}(\xi)=\widehat{\mathcal{V}}^{(\mathbf{h})}(\xi), \quad\mbox{for almost every $\xi \in \R^d$.}
\end{equation}
Then, noticing that the Fourier transform of $\mathcal{V}^{(\mathbf{h})}_N$ satisfies, for all $N \in \N$ and $\xi\in\R^d$,
\[ \widehat{\mathcal{V}}_N^{(\mathbf{h})}(\xi)= 2^{-J \frac{d}{2}} \left( \sum_{|\mathbf{k}| \leq N} \beta_{\mathbf{k}}^{(\mathbf{h})} e^{-i2^{-J} \langle \mathbf{k},\xi \rangle} \right) \prod_{\ell=1}^d \widehat{\Phi}^{(1/2-h_\ell)}(2^{-J} \xi_\ell),\]
it turns out that \eqref{eqn:transfoserivj} is a consequence of \eqref{eqn:convvj} and \eqref{eqn:beta}.
\end{proof}

\begin{Lemma}\label{lemma:kjvj}
For all $t \in\R_+$ and $J \in \N$, we have
\[
K^{(d)}_{\mathbf{h},J}(t,\mathbf{u}) = 2^{-J(h_1+\cdots+h_d-d)}  \mathcal{V}^{(\mathbf{h})}(\mathbf{u}),\quad\mbox{for almost all $\mathbf{u}\in\R^d$.}
\]
\end{Lemma}

\begin{proof} For each fixed $t\in\R_+$, the function $\xi\mapsto\widehat{K^{(d)}_{\mathbf{h},J}}(t,\xi)$ denotes the Fourier transform in $L^2(\R^d)$ of the function $\mathbf{u}\mapsto K^{(d)}_{\mathbf{h},J}(t,\mathbf{u})$. It follows 
from Lemma \ref{lemma:firstexp}, Definition \ref{def:alpha}, \eqref{eqn:expressionalpha}, \eqref{eqn:exprgamma}, Definition \ref{def:fouphidelta} and \eqref{eqn:transfoserivj} that, for almost every $\xi \in \R^d$,
\begin{align*}
\widehat{K^{(d)}_{\mathbf{h},J}}(t,\xi) &= (2 \pi)^{-d/2}\, 2^{-J(h_1+\cdots+h_d-d/2)} \widehat{\alpha}^{(\mathbf{h})} (2^{-J}\mathbf{\xi})  \prod_{\ell=1}^d \widehat{\phi}(2^{-J} \xi_\ell) \\
& = 2^{-J(h_1+\cdots+h_d-d/2)} \widehat{\beta}^{(\mathbf{h})} (2^{-J}\mathbf{\xi}) \widehat{\Gamma}^{(\mathbf{h})} (2^{-J}\mathbf{\xi}) \prod_{\ell=1}^d \widehat{\phi}(2^{-J} \xi_\ell) \\
&= 2^{-J(h_1+\cdots+h_d-d/2)} \widehat{\beta}^{(\mathbf{h})} (2^{-J}\mathbf{\xi}) \prod_{\ell=1}^d \widehat{\Phi}^{(1/2-h_\ell)}(2^{-J} \xi_\ell) \\
&= 2^{-J(h_1+\cdots+h_d-d)} \widehat{\mathcal{V}}^{(\mathbf{h})}(\mathbf{\xi}).
\end{align*}
Then taking the inverse Fourier transforms of the two functions in both sides of the equality $\widehat{K^{(d)}_{\mathbf{h},J}}(t,\bullet)=2^{-J(h_1+\cdots+h_d-d)} \widehat{\mathcal{V}}^{(\mathbf{h})}(\bullet)$, we obtain the lemma.
\end{proof}
The following proposition provides, for each $J\in\N$, a useful representation of the generalized FARIMA sequence $(\sigma_{J,\textbf{k}}^{(\mathbf{h})})_{\textbf{k}\in\Z^d}$, introduced in \eqref{eq:ext-farima}, in terms of multiple Wiener integrals and the functions $2^{J \frac{d}{2}} \prod_{\ell=1}^d \Phi^{(1/2 -h_\ell)}(2^J u_\ell-k_\ell)$ in \eqref{eqn:serievj}. 
\begin{Prop}
\label{prop:intexpsig}
One has almost surely, for all $J\in\N$ and $\textbf{k}\in\Z^d$,
\begin{equation}
\label{prop:intexpsig:eq1}
\sigma_{J,\textbf{k}}^{(\mathbf{h})}=\int_{\R^d}' 2^{J \frac{d}{2}} \prod_{\ell=1}^d \Phi^{(1/2-h_\ell)}(2^J u_\ell-k_\ell) \, dB(u_1) \cdots dB(u_d).
\end{equation}
\end{Prop}

\begin{proof} 
The proposition easily results from Lemma \ref{Lemma:expphidelta}, the Wiener isometry and  \eqref{eq:ext-farima}.
\end{proof}

\begin{Rmk}
\label{rem:logbousig}
Thanks to the representation \eqref{prop:intexpsig:eq1}, using Theorem 6.7 in \cite{MR1474726}, the Wiener isomety and arguments similar to those in the proofs of Lemmas 1 and 2 in \cite{MR2027888}, it can be shown that there exist $\widetilde{C}$ a positive finite random variable and $\widetilde{\Omega}$ an event of probability $1$, such that, one has on $\widetilde{\Omega}$, for all $J\in\N$ and $\textbf{k}\in\Z^d$,
\begin{equation}
\label{rem:logbousig:eq1}
\big |\sigma_{J,\textbf{k}}^{(\mathbf{h})}\big|\le \widetilde{C}\left (\log \Big (3+J+|\mathbf{k}|\Big)\right)^{d/2}. 
\end{equation} 
\end{Rmk}

We have enough materials to complete the proof of Theorem \ref{thm:basefreq}.

\begin{proof}[End of the Proof of Theorem \ref{thm:basefreq}] It follows from \eqref{eqn:approxprocess}, Lemma \ref{lemma:kjvj}, \eqref{eqn:serievj}, \eqref{eq-seqbe}, the Wiener isometry and \eqref{prop:intexpsig:eq1}, that, for each $t\in\R_+$, the random series in the right-hand side of \eqref{eqn:basfreq:toprove} is convergent in $L^2(\Omega)$ and is equal almost surely to $X_{\mathbf{h},J}^{(d)}(t)$. Moreover using \eqref{rem:logbousig:eq1}, the fact that for all $\delta\in (0,1/2)$ the $\Phi_{\Delta}^{(\delta)}$ belongs to the Schwartz class, and classical computations, one can show that the convergence of the latter random series also holds almost surely and uniformly $t$ on each compact interval of $\R_+$.
\end{proof}

\section{Proof of Theorem \ref{thm:main1}} \label{sect:approxerro}

In this section, we aim at proving Theorem \ref{thm:main1}. We will need a number of intermediary results which mainly consist in bounding in convenient ways well-chosen parts of the random series in \eqref{eq:X-XJ}. We mention that the event $\Omega^*$ of probability $1$ (see Lemma \ref{lem:bound-gjk}) will appear in the statements of many of them. In what follows, we write, for all $(\mathbf{j},\mathbf{k})\in  (\Z^d)^2$ and $t\in\R_+$,
\begin{equation}\label{def:AJK}
\mathcal{A}_{\mathbf{j},\mathbf{k}}(t):= \int_0^t \prod_{\ell=1}^d \psi_{h_\ell}(2^{j_\ell}s-k_\ell) \, ds,
\end{equation}
where $\psi_{h_\ell}$ is the fractional primitive of order $h_l-1/2$ of the univariate Meyer mother wavelet $\psi$ (see Definition \ref{def:fractwav}). We start by defining the following subsets of $\Z^d$. 
\begin{Def} For all $n \in [\![1,d]\!]$ and all $J \in \N$, we define the infinite subset of $\Z^d $
\begin{equation}\label{ens:alephnj}
\aleph_{n,J} := \left\lbrace \mathbf{j} \in \Z^d : j_n \ge J \mbox{ and }  \max_{\ell \in [\![1,d]\!]} j_\ell = j_n \right\rbrace.
\end{equation}
and, for all $T>0$,
\begin{equation}\label{ens:bethnjt}
\beth_{n,T} := \left\lbrace \mathbf{k}\in\Z^d : k_n\in\Z; \quad \exists\,\ell\in [\![1,d]\!] \setminus \{n\},\quad |k_\ell| \ge 2^{j_n+1}T \right\rbrace.
\end{equation}
\end{Def}

\begin{Lemma}\label{lemmepour22:1}
Let $T >2$ and $L \geq 3/2$ be two fixed real numbers. There exits a positive almost surely finite random variables $C$ such that, for all $n \in [\![1,d]\!]$ and $J \in \N$, on $\Omega^*$, the random variable
\begin{equation}\label{def:hnj0}
\mathcal{H}^0_{n,J} := \sum_{\mathbf{j} \in \aleph_{n,J}} 2^{ j_1 (1-h_1)+\cdots + j_d (1-h_d)} \sum_{\mathbf{k} \in \beth_{n,T} } \left|\varepsilon_{\mathbf{j},\mathbf{k}}\right| \sup_{t\in [0,T]} \left| \mathcal{A}_{\mathbf{j},\mathbf{k}}(t) \right|
\end{equation}
is bounded from above by $ C J^{\frac{d}{2}} \left(\log(3+J)\right)^{\frac{d-1}{2}} 2^{-J(h_1+ \cdots + h_d + L -d-1)}$.
\end{Lemma}

\begin{proof}
Let $\mathbf{j} \in \aleph_{n,J}$. Let us define the set of boolean vectors
\[ \mathcal{B}_n := \Big\{v=(v_l)_{\ell \in [\![1,d]\!]} \in \{0,1\}^d \, : \, v_n =1 \text{ and } \exists\,\ell' \neq n \, : \, v_{\ell'}=0 \Big\}\]
and, for all $v \in \mathcal{B}_n$, the set
\[ \beth_{n,T}^v := \left\lbrace \mathbf{k}\in\Z^d : k_\ell \in\Z \text{ if } v_\ell=1 \text{ and } |k_\ell| \ge 2^{j_n+1}T \text{ otherwise} \right\rbrace. \]
Of course, we have
\begin{equation}\label{eqn:ho:cutsum}
\sum_{\mathbf{k} \in \beth_{n,T} } \left|\varepsilon_{\mathbf{j},\mathbf{k}}\right| \sup_{t\in [0,T]} \left| \mathcal{A}_{\mathbf{j},\mathbf{k}}(t) \right| \leq \sum_{v \in \mathcal{B}_n} \sum_{\mathbf{k} \in \beth_{n,T}^v } \left|\varepsilon_{\mathbf{j},\mathbf{k}}\right| \sup_{t\in [0,T]} \left| \mathcal{A}_{\mathbf{j},\mathbf{k}}(t) \right|.
\end{equation}

Fix $v \in \mathcal{B}_n$, using together  the inequality \eqref{log:ineg1}, the triangular inequality, the fact that the function $y \mapsto (2+y)^{-L}\sqrt{\log(2+y)}$ is decreasing on $\R_+$ and the inequality \eqref{log:ineg2}, we have, for all $s \in [0,T]$
\begin{align}
&\prod_{\ell=1}^d \frac{\sqrt{\log(3+|j_\ell|+|k_\ell|)}}{(3+|2^{j_\ell}s-k_\ell|)^L} \nonumber\\
&= \left(\prod_{\ell : v_\ell=1} \sum_{k_\ell \in \Z} \frac{\sqrt{\log(3+|j_\ell|+|k_\ell|)}}{(3+|2^{j_\ell}s-k_\ell|)^L} \right)\times \ldots \nonumber \\ & \qquad \ldots \times \left(\prod_{\ell' : v_{\ell'}=0} \sum_{|k_{\ell'}|> 2^{j_n+1}T} \frac{\sqrt{\log(3+|j_{\ell'}|+|k_{\ell'}|)}}{(3+|2^{j_{\ell'}}s-k_{\ell'}|)^L} \right) \nonumber \\
& \leq c_0 2^L \left(\prod_{\ell : v_\ell=1}\sqrt{\log(3+|j_\ell|+2^{j_\ell}T)} \right) \times \ldots \nonumber\\ & \qquad \ldots \times\left( \prod_{\ell' : v_{\ell'}=0} \sum_{|k_{\ell'}|> 2^{j_n+1}T} \frac{\sqrt{\log(3+|j_{\ell'}|)}\sqrt{\log(3+|k_{\ell'}|)}}{(3+|k_{\ell'}|)^L} \right) \nonumber \\
& \leq c_1 \left(\prod_{\ell : v_\ell=1}\sqrt{\log(3+|j_\ell|+2^{j_n}T)} \right)\times \ldots \nonumber \\ & \qquad \ldots \times \left( \prod_{\ell' : v_{\ell'}=0} \sqrt{\log(3+|j_{\ell'}|)}\int_{2^{j_n+1}T}^{+\infty} \frac{\sqrt{\log(2+y)}}{(2+y)^L} dy \right) \nonumber \\
& \leq c_2  \sqrt{\prod_{\ell=1, \ell\neq n}^d \log(3+|j_\ell|)} \times j_n^{\frac{d}{2}} \times 2^{-j_n(L-1)(\# \{ \ell' \, : \, v_{\ell'}=0 \})} \nonumber \\
& \le c_2  \sqrt{\prod_{\ell=1, \ell\neq n}^d \log(3+|j_\ell|)} \times j_n^{\frac{d}{2}} \times 2^{-j_n(L-1)}, \label{maj:gnj0:avant}
\end{align}
with $c_0$, $c_1$ and $c_2$ positive deterministic constants not depending on $n$, $\mathbf{j}$, $v$ nor $J$. Then, the expression \eqref{def:AJK},  the bound \eqref{eqn:bound:rv}, the fast decay property \eqref{maj:psih} and inequality \eqref{maj:gnj0:avant} give
\begin{align}
 &\sum_{\mathbf{k} \in \beth_{n,T}^v } \left|\varepsilon_{\mathbf{j},\mathbf{k}}\right| \sup_{t\in [0,T]} \left| \mathcal{A}_{\mathbf{j},\mathbf{k}}(t) \right| \leq C_1  \sqrt{\prod_{\ell=1, \ell\neq n}^d \log(3+|j_\ell|)} \times j_n^{\frac{d}{2}} \times 2^{-j_n(L-1)}, \label{maj:gnj0}
\end{align}
with $C_1$ a positive almost surely finite random variables not depending on $n$, $\mathbf{j}$, $v$ nor $J$. Then, using \eqref{def:hnj0}, \eqref{eqn:ho:cutsum}, \eqref{maj:gnj0} and the triangular inequality, we get
\begin{align}
&\mathcal{H}^0_{n,J}  \le C_3 \sum_{j_n\ge J} \quad \sum_{\substack{j_\ell \le j_n \\ \ell \in [\![1,d]\!] \setminus \{n\}}} \sqrt{\prod_{\substack{\ell=1 \\ \ell\neq n}}^d \log(3+|j_\ell|)} \times j_n^{\frac{d}{2}} \times 2^{-j_n(L-1)} \prod_{\ell=1}^d 2^{j_\ell(1-h_\ell)} \nonumber \\
& =C_3 \sum_{j_n\ge J} \quad \sum_{\substack{j_\ell \le j_n \\ \ell \in [\![1,d]\!] \setminus \{n\}}} \left( \prod_{\substack{\ell=1 \\ \ell\neq n}}^d \sqrt{\log(3+|j_\ell|)} 2^{j_\ell(1-h_\ell)} \right) \times j_n^{\frac{d}{2}} \times 2^{-j_n(h_n+L-2)} \nonumber \\
& = C_3  \sum_{j_n\ge J} j_n^{\frac{d}{2}} 2^{-j_n(h_n+L-2)} \prod_{\substack{\ell=1 \\ \ell\neq n}}^d \sum_{j_\ell=-\infty}^{j_n} \sqrt{\log(3+|j_\ell|)} 2^{j_\ell(1-h_\ell)} \nonumber
\end{align} 
and since, by inequality \eqref{log:ineg1}, we have
\begin{align*}
 \prod_{\substack{\ell=1 \\ \ell\neq n}}^d & \sum_{j_\ell=-\infty}^{j_n} \sqrt{\log(3+|j_\ell|)} 2^{j_\ell(1-h_\ell)}\\
 &= \prod_{\substack{\ell=1 \\ \ell\neq n}}^d \log(3+|j_n|)\sum_{p=0}^{+\infty} \sqrt{\log(3+|j_n-p|)} 2^{(j_n-p)(1-h_\ell)} \\
& = \prod_{\substack{\ell=1 \\ \ell\neq n}}^d \log(3+|j_n|) 2^{j_n(1-h_\ell)} \left(\sum_{p=0}^{+\infty} \sqrt{\log(3+|p|)} 2^{-p(1-h_n)} \right) \\
& \leq c \left(\log(3+|j_n|)\right)^{\frac{d-1}{2}} \prod_{\substack{\ell=1 \\ \ell\neq n}}^d 2^{j_n(1-h_\ell)} ,
\end{align*}
for a deterministic constant $c>0$, we conclude that
\begin{align}
\mathcal{H}^0_{n,J} &\le C_2 \sum_{j_n\ge J} j_n^{\frac{d}{2}} \left(\log(3+|j_n|)\right)^{\frac{d-1}{2}} 2^{-j_n(h_n+L-2)} \prod_{\substack{\ell=1 \\ \ell\neq n}}^d 2^{j_n(1-h_\ell)} \nonumber \\ 
& = C_2 \sum_{j_n\ge J} j_n^{\frac{d}{2}} \left(\log(3+|j_n|)\right)^{\frac{d-1}{2}} 2^{-j_n(h_1+ \cdots + h_d + L-d-1)} \nonumber \\ 
& \le C_3 J^{\frac{d}{2}} \left(\log(3+J)\right)^{\frac{d-1}{2}} 2^{-J(h_1+ \cdots + h_d + L-d-1)},
\end{align}
where $C_2$ and $C_3$ are positive almost surely finite random variables.
\end{proof}

For the next Lemma, we need to define a partition of $\Z$. This is similar to what is done in \cite[Definition 2.5]{MR4110623} for the Rosenblatt process.

\begin{Def}
Let $a$ be a fixed real number satisfying $1/2 < a < 1$. For all $(j,k)\in \Z_+ \times \Z$, let us denote by $B_{j,k}$ the interval 
\begin{equation}\label{def:bjk}
B_{j,k}:= [ k2^{-j}-2^{-ja},k2^{-j}+2^{-ja} ].
\end{equation}
For all $j\in\N$ and for all $t\in\R_+$, we consider the three disjoint subsets of $\Z$:
\begin{align}
D_j^1(t) & := \lbrace k\in\Z : B_{j,k} \subseteq [0,t] \rbrace,  \label{def:dj1}\\
D_j^2(t) & := \lbrace k\in\Z\setminus D_j^1(t) : B_{j,k} \cap [0,t] \neq \emptyset \rbrace,\label{def:dj2}\\
D_j^3(t) & := \lbrace k\in\Z  : B_{j,k} \cap [0,t] = \emptyset \rbrace. \label{def:dj3}
\end{align}
These three sets, depending on $t$ and $a$, form a partition of $\Z$:
\[
\Z = \bigcup_{\ell=1}^3 D_j^{\ell}(t).
\]
\end{Def}

\begin{Rmk}
By definition, it is clear that $D_j^3(t)$ is always an infinite countable set while $D_j^1(t)$ and $D_j^2(t)$ are two finites sets, possibly empty. Moreover, for all strictly positive real number $T$ and all $j\in\Z_+$, we have 
\begin{align}
& \sup_{t\in [0,T]} \left\lbrace \card (D_j^1(t)) \right\rbrace \le c' 2^j, \label{maj:carddj1} \\
& \sup_{t\in [0,T]} \left\lbrace \card (D_j^2(t)) \right\rbrace \le c'' 2^{j(1-a)}, \label{maj:carddj2} 
\end{align}
where $c'\ge 1$ and $c''\ge 1$ are two finite positive constants. Let us also remark that $c'$ only depends on $T$ while $c''$ does not depend on it.
\end{Rmk}

\begin{Lemma}\label{lemmepour22:2}
Let $T >2$ and $L \ge 2^{-1}(1-a)^{-1}+1$ be two fixed real numbers. There exits a positive almost surely finite random variables $C$ such that, for all $n \in [\![1,d]\!]$ and $J \in \N$, on $\Omega^*$, the random variable
\begin{align*}
\mathcal{H}^3_{n,J} &:= \sum_{\mathbf{j} \in \aleph_{n,J}} 2^{ j_1 (1-h_1)+\cdots + j_d (1-h_d)} \times \ldots  \\ & \qquad \ldots \times\sup_{t \in [0,T]} \left \{\sum_{k_n \in D_{j_n}^3(t) } \sum_{ \substack{k_\ell \in \Z \\ \ell\in [\![1,d]\!] \setminus \{n\}}} \left|\varepsilon_{\mathbf{j},\mathbf{k}}\right|  \left| \mathcal{A}_{\mathbf{j},\mathbf{k}}(t) \right| \right\}    
\end{align*}
is bounded from above by
$ C J^{\frac{d+1}{2}} \left(\log(3+J)\right)^{\frac{d-1}{2}} 2^{-J(h_1+ \cdots + h_d + (L-1)(1-a) -d)}$.
\end{Lemma}

\begin{proof}
Let $t\in [0,T]$ and $\mathbf{j} \in \aleph_{n,J}$. Using together the expression \eqref{def:AJK}, the fast decay property \eqref{maj:psih}, the bound \eqref{eqn:bound:rv} as well as Lemmata \ref{lem:maj:somme1} and \ref{lem:maj:somme2}, we get
\begin{align}
& \sum_{k_n\in D_{j_n}^3(t)} \quad \sum_{\substack{k_\ell\in\Z \\ \ell \in [\![1,d]\!] \setminus \{n\}}} \left|\varepsilon_{\mathbf{j},\mathbf{k}}\right| \left| \mathcal{A}_{\mathbf{j},\mathbf{k}}(t) \right| \nonumber \\
& \le C_0 \sum_{k_n\in D_{j_n}^3(t)} \quad \sum_{\substack{k_\ell\in\Z \\ \ell \in [\![1,d]\!] \setminus \{n\}}} \int_0^T \prod_{\ell=1}^d \frac{\sqrt{\log(3+|j_\ell|+|k_\ell|}}{(3+|2^{j_\ell}s-k_\ell|)^L} ds \nonumber \\
& = C_0 \int_{0}^T \left( \sum_{k_n\in D_{j_n}^3(t)} \frac{\sqrt{\log(3+|j_n|+|k_n|}}{(3+|2^{j_n}s-k_n|)^L}  \right) \prod_{\substack{\ell=1 \\ \ell\neq n}}^d \sum_{k_\ell \in\Z} \frac{\sqrt{\log(3+|j_\ell|+|k_\ell|}}{(3+|2^{j_\ell}s-k_\ell|)^L} ds \nonumber \\
& \le C_1 \prod_{\substack{\ell=1 \\ \ell\neq n}}^d \sqrt{\log(3+|j_\ell|+2^{j_\ell}T)} \int_0^T \sum_{k_n\in D_{j_n}^3(t)} \frac{\sqrt{\log(3+|j_n|+|k_n|}}{(3+|2^{j_n}s-k_n|)^L} ds \nonumber \\
& \le C_2 \prod_{\substack{\ell=1 \\ \ell\neq n}}^d \sqrt{\log(3+|j_\ell|+2^{j_\ell}T)} (1+j_n) 2^{-j_n(L-1)(1-a)}. \label{maj:gnj3}
\end{align}
Now, let us remark that, if $j_n \geq J$, combining the triangular inequality and \eqref{log:ineg1}, it comes
\begin{align}
\sum_{\substack{j_\ell \leq j_n \\ \ell \in [\![1,d]\!] \setminus \{n\}}} & \prod_{\substack{\ell=1 \\ \ell\neq n}}^d \left( \sqrt{\log(3+|j_\ell|+2^{j_\ell}T)} 2^{j_\ell(1-h_\ell)} \right) \nonumber\\
& = \prod_{\substack{\ell=1 \\ \ell\neq n}}^d \left( \sum_{j_\ell=-\infty}^{j_n} \sqrt{\log(3+|j_\ell|+2^{j_\ell}T)} 2^{j_\ell(1-h_\ell)} \right) \nonumber \\
& = \prod_{\substack{\ell=1 \\ \ell\neq n}}^d \left( \sum_{p=0}^{+ \infty} \sqrt{\log(3+|j_n-p|+2^{j_n-p}T)} 2^{(j_n-p)(1-h_\ell)} \right) \nonumber \\
& \leq  \left(\sqrt{\log(3+2^{j_n}T)\log(3+|j_n|)}\right)^{d-1} \times \ldots \nonumber\\
& \qquad \ldots \times \prod_{\substack{\ell=1 \\ \ell\neq n}}^d 2^{j_n(1-h_\ell)} \left(  \sum_{p=0}^{+ \infty} \sqrt{\log(3+p)} 2^{-p(1-h_\ell)} \right) \nonumber \\
& \leq c (\sqrt{\log(3+2^{j_n}T)\log(3+|j_n|)})^{d-1} \prod_{\substack{\ell=1 \\ \ell\neq n}}^d 2^{j_n(1-h_\ell)} \label{maj:gnj3bis}
\end{align}
for a deterministic constant $c>0$. Then, by the the definition \eqref{ens:alephnj}, we get from \eqref{maj:gnj3} and  \eqref{maj:gnj3bis}
\begin{align*}
\mathcal{H}^3_{n,J} & \le C_3 \sum_{j_n\ge J} (1+j_n) 2^{-j_n(h_1+\cdots+h_d+(L-1)(1-a)-d)} \times \ldots \\
& \qquad \ldots \times (\sqrt{\log(3+2^{j_n}T)\log(3+|j_n|)})^{d-1} \\
& \leq C_4 \sum_{j_n\ge J} j_n^{\frac{d+1}{2}} (\log(3+|j_n|))^{\frac{d-1}{2}}  2^{-j_n(h_1+\cdots+h_d+(L-1)(1-a)-d)} \\
& \leq C_5 J^{\frac{d+1}{2}} (\log(3+|J|))^{\frac{d-1}{2}}  2^{-J(h_1+\cdots+h_d+(L-1)(1-a)-d)}
\end{align*}
where $C_3$, $C_4$ and $C_5$ are positive almost surely finite random variables.
\end{proof}

\begin{Lemma}\label{lemmepour22:3}
Let $T >2$ be a fixed real number. There exits a positive almost surely finite random variables $C$ such that, for all $n \in [\![1,d]\!]$ and $J \in \N$, on $\Omega^*$, the random variable
\begin{align*}
\mathcal{H}^2_{n,J} &:= \sum_{\mathbf{j} \in \aleph_{n,J}} 2^{ j_1 (1-h_1)+\cdots + j_d (1-h_d)} \times \ldots \\ & \qquad \ldots \times  \sup_{t \in [0,T]} \left \{\sum_{k_n \in D_{j_n}^2(t) } \sum_{\substack{k_\ell \in \Z \\ \ell\in [\![1,d]\!] \setminus \{n\}}} \left|\varepsilon_{\mathbf{j},\mathbf{k}}\right|  \left| \mathcal{A}_{\mathbf{j},\mathbf{k}}(t) \right| \right\}
\end{align*}
is bounded from above by
$ C J^{\frac{d}{2}} \left(\log(3+J)\right)^{\frac{d-1}{2}} 2^{-J(h_1+ \cdots + h_d +a -d)}.$
\end{Lemma}

\begin{proof}
let $L>1$ be a fixed real number, $t\in [0,T]$ and $\mathbf{j} \in \aleph_{n,J}$. Using the definition \eqref{def:AJK}, the fast decay property \eqref{maj:psih}, inequality \eqref{eqn:boundn01}, Lemma \ref{lem:maj:somme1}, the inequality $|k_n| \le 2^{j_n(1-a)} + 2^{j_n}T$, for all $k_n \in D_{j_n}^2(t)$, the change of variable $z=2^{j_n}s-k_n$, the bound \eqref{maj:carddj2} ant the inequality \eqref{log:ineg2}, we have 

\begin{align}
& \sum_{k_n\in D_{j_n}^2(t)} \quad \sum_{\substack{k_\ell\in\Z \\ \ell \in [\![1,d]\!] \setminus \{n\}}}  \left|\varepsilon_{\mathbf{j},\mathbf{k}}\right|  \left| \mathcal{A}_{\mathbf{j},\mathbf{k}}(t) \right|   \nonumber \\
& \le C_0 \int_0^T \sum_{k_n\in D_{j_n}^2(t)} \quad \sum_{\substack{k_\ell\in\Z \\ \ell \in [\![1,d]\!] \setminus \{n\}}} \prod_{\ell=1}^d \frac{\sqrt{\log(3+|j_\ell|+|k_\ell|)}}{(3+|2^{j_\ell}s-k_\ell|)^L} ds \nonumber \\
& = C_0 \int_0^T \sum_{k_n\in D_{j_n}^2(t)} \frac{\sqrt{\log(3+|j_n|+|k_n|)}}{(3+|2^{j_n}s-k_n|)^L}  \prod_{\substack{\ell=1 \\ \ell\neq n}}^d \sum_{k_\ell \in \Z} \frac{\sqrt{\log(3+|j_\ell|+|k_\ell|)}}{(3+|2^{j_\ell}s-k_\ell|)^L} ds \nonumber \\
& \le C_1 \prod_{\substack{\ell=1 \\ \ell\neq n}}^d \sqrt{\log(3+|j_\ell| + 2^{j_\ell}T)} \int_0^T \sum_{k_n\in D_{j_n}^2(t)} \frac{\sqrt{\log(3+|j_n|+|k_n|)}}{(3+|2^{j_n}s-k_n|)^L} ds \nonumber \\
& \le C_1 \sqrt{\log(3+|j_n|+2^{j_n(1-a)}+2^{j_n}T)} \prod_{\substack{\ell=1 \\ \ell\neq n}}^d \sqrt{\log(3+|j_\ell| + 2^{j_\ell}T)}  \times \ldots \nonumber \\ &  \qquad \ldots \times \qquad \sum_{k_n\in D_{j_n}^2(t)} \int_0^T \frac{ds}{(3+|2^{j_n}s-k_n|)^L} \nonumber \\
&= C_2 \left( \int_0^T \frac{dz}{(3+|z|)^L} \right) \sqrt{\log(3+|j_n|+2^{j_n(1-a)}+2^{j_n}T)} \nonumber \times \ldots \\ & \qquad \ldots \times \prod_{\substack{\ell=1 \\ \ell\neq n}}^d \sqrt{\log(3+|j_\ell| + 2^{j_\ell}T) } \,\card (D_{j_n}^2(t)) 2^{-j_n} \nonumber \\
& \le C_3 2^{-j_n a}\,\sqrt{\log(3+|j_n|+2^{j_n(1-a)}+2^{j_n}T)} \prod_{\substack{\ell=1 \\ \ell\neq n}}^d \sqrt{\log(3+|j_\ell| + 2^{j_\ell}T) } \nonumber \\
& \le C_4 2^{-j_n a}\,\sqrt{1+j_n} \prod_{\substack{\ell=1 \\ \ell\neq n}}^d \sqrt{\log(3+|j_\ell| + 2^{j_\ell}T) },   \label{maj:gnj2}
\end{align}
where $C_0$, $C_1$, $C_2$, $C_3$ and $C_4$ are positive almost surely finite random variables not depending on $n$, $t$, $\mathbf{j}$ and $J$. Then, combining the definition \eqref{ens:alephnj} and the inequalities \eqref{maj:gnj3bis}, $a > 1/2$  and $\sum_{\ell=1}^d h_\ell > d-1/2$, we get
\begin{align*}
& \mathcal{H}^2_{n,J}  \le C_4 \sum_{\mathbf{j} \in \aleph_{n,J}} \left( \prod_{\ell=1}^d 2^{j_\ell(1-h_\ell)} \right) 2^{-j_n a} \sqrt{1+j_n} \prod_{\substack{\ell=1 \\ \ell\neq n}}^d \sqrt{\log(3+|j_\ell| + 2^{j_\ell}T) } \\
& \le C_5 \sum_{j_n \geq J} 2^{-j_n (h_n+a-1)} \sqrt{1+j_n} (\sqrt{\log(3+2^{j_n}T)\log(3+|j_n|)})^{d-1} \times \ldots \\ & \qquad \ldots \times \prod_{\substack{\ell=1 \\ \ell\neq n}}^d 2^{j_n(1-h_\ell)} \\
& \leq C_6 \sum_{j_n \geq J} 2^{-j_n (h_1+ \cdots +h_n+a-d)} j_n^\frac{d}{2} (\log(3+|j_n|))^\frac{d-1}{2}  \\
& \leq C_7 2^{-J (h_1+ \cdots +h_n+a-d)} J^\frac{d}{2} (\log(3+|J|))^\frac{d-1}{2},
\end{align*}
where $C_5$, $C_6$ and $C_7$ are positive almost surely finite random variables.
\end{proof}

Let us now consider an integral which plays a crucial role in what follows.

\begin{Def}
For all $(\mathbf{j},\mathbf{k}) \in (\Z^d)^2$, one sets
\begin{equation}\label{eqn:fjk}
F_{\mathbf{j},\mathbf{k}} := \int_{\R} \prod_{\ell =1}^d \psi_{h_\ell} (2^{j_\ell}s-k_\ell) \, ds.
\end{equation}
\end{Def}
We start by giving an upper bound for $(F_{\mathbf{j},\mathbf{k}})^2$.

\begin{Lemma}\label{lem:boundF}
There exists a deterministic constant $c_\psi>0$ such that for all $(\mathbf{j},\mathbf{k}) \in (\Z^d)^2$, if $ \displaystyle j_n = \max_{\ell \in [\![1,d]\!]} j_\ell$, we have
\[ (F_{\mathbf{j},\mathbf{k}})^2 \leq c_\psi \, 2^{-j_n} \int_{\R}   \left|\psi_{h_n} (2^{j_n}s-k_n) \right|\prod_{\substack{\ell =1 \\ \ell \neq n}}^d  \left|\psi_{h_\ell} (2^{j_\ell}s-k_\ell) \right|^2 \, ds.\]
\end{Lemma}

\begin{proof}
Let $\mathbb{P}_{j_n,k_n}$ be the probability measure defined on the Borel $\sigma$-algebra of $\R$ and with density given by the function
\[ s \mapsto 2^{j_n} \|\psi_{h_n} \|_{L^1(\R)}^{-1} |\psi_{h_n}(2^{j_n}x-k_{n})|.\]
We clearly have that
\begin{align*}
(F_{\mathbf{j},\mathbf{k}})^2 &\leq  \left( \int_{\R}  \prod_{\ell =1}^d  \left|\psi_{h_\ell} (2^{j_\ell}s-k_\ell) \right| \, ds \right)^2 \\
& = 2^{-2j_n} \|\psi_{h_n} \|_{L^1(\R)}^2 \left( \int_{\R}  \prod_{\substack{\ell =1 \\ \ell \neq n}}^d  \left|\psi_{h_\ell} (2^{j_\ell}s-k_\ell) \right| \, {\rm d}\mathbb{P}_{j_n,k_n}(s) \right)^2.
\end{align*}
Then, it results from Jensen's inequality that 
\begin{align*}
(F_{\mathbf{j},\mathbf{k}})^2 &\leq 2^{-2j_n} \|\psi_{h_n} \|_{L^1(\R)}^2  \int_{\R}  \prod_{\substack{\ell =1 \\ \ell \neq n}}^d  \left|\psi_{h_\ell} (2^{j_\ell}s-k_\ell) \right|^2 \, {\rm d}\mathbb{P}_{j_n,k_n} (s) \\
& =2^{-j_n} \|\psi_{h_n} \|_{L^1(\R)} \int_{\R}     \left|\psi_{h_n} (2^{j_n}s-k_n) \right|\prod_{\substack{\ell =1 \\ \ell \neq n}}^d  \left|\psi_{h_\ell} (2^{j_\ell}s-k_\ell) \right|^2 \, ds.
\end{align*}
Then, setting $ \displaystyle c_\psi:= \max_{\ell \in [\![1,d]\!]} \|\psi_{h_\ell} \|_{L^1(\R)}$ one obtains the lemma.
\end{proof}

\begin{Lemma}\label{lemmepour22:4}
Let $T >2$ and $L>d+2$ be two fixed real numbers. There exits a positive almost surely finite random variable $C$ such that, for all $n \in [\![1,d]\!]$ and $J \in \N$, on $\Omega^*$, the random variable
\begin{align*}
\mathcal{H}^1_{n,J} &:= \sum_{\mathbf{j} \in \aleph_{n,J}} 2^{ j_1 (1-h_1)+\cdots + j_d (1-h_d)} \times \ldots \\
& \qquad \ldots \times \sup_{t \in [0,T]} \left \{\sum_{k_n \in D_{j_n}^1(t) } \sum_{\substack{k_\ell \in \Z \\ \ell\in [\![1,d]\!] \setminus \{n\}}} \left|\varepsilon_{\mathbf{j},\mathbf{k}}\right|  \left| \mathcal{A}_{\mathbf{j},\mathbf{k}}(t)- F_{\mathbf{j},\mathbf{k}}\right| \right\}    
\end{align*}
is bounded from above by $ C J^{d-1}\, 2^{-J((L-2)(1-a)+h_1+ \cdots + h_d  -d+1)}$.
\end{Lemma}

\begin{proof}
Let us fix $t\in [0,T]$ and $\mathbf{j} \in \aleph_{n,J}$. Using the definitions \eqref{def:AJK}, \eqref{eqn:fjk}, \eqref{def:bjk} and \eqref{def:dj1}, the fast decay property \eqref{maj:psih}, \eqref{eqn:bound:rv}, the inequality $|k_n| \leq 2^{j_n} T$, for all $k_n \in D_{j_n}^1(t)$, the inequality $2^{j_n} T \geq j_n$, Lemma \ref{lem:maj:somme1}, the fact that $ \displaystyle j_n = \max_{\ell \in [\![1,d]\!]} j_\ell$, the triangular inequality and finally inequality \eqref{log:ineg1}, it comes
\begin{align}
&\sum_{k_n \in D_{j_n}^1(t) } \sum_{\substack{k_\ell \in \Z \\ \ell\in [\![1,d]\!] \setminus \{n\}}} \left|\varepsilon_{\mathbf{j},\mathbf{k}}\right|  \left| \mathcal{A}_{\mathbf{j},\mathbf{k}}(t)- F_{\mathbf{j},\mathbf{k}}\right| \nonumber \\
& \leq C_0 \sqrt{\log(3+2 ^{j_n+1} T)} \int_{\R \setminus [0,t]} \left( \sum_{k_n \in D_{j_n}^1(t) } \frac{1}{(3+|2^{j_n}s-k_n|)^L} \right) \times \ldots \nonumber \\ & \qquad \qquad  \qquad \qquad \qquad \qquad \qquad  \ldots \times\prod_{\substack{\ell =1;\, \ell \neq n}}^d \left( \sum_{k_\ell \in \Z } \frac{\sqrt{\log(3+|j_\ell|+|k_\ell|)}}{(3+|2^{j_\ell}s-k_\ell|)^L} \right) \, ds \nonumber \\
& \leq C_1 \sqrt{\log(3+2 ^{j_n+1} T)} \int_{\R \setminus [0,t]}  \sum_{k_n \in D_{j_n}^1(t) } \frac{\prod_{\substack{\ell =1;\,\ell \neq n}}^d  \sqrt{\log(3+|j_\ell|+|2^{j_\ell} s|)}}{(3+|2^{j_n}s-k_n|)^L}   \, ds \nonumber \\
& \leq C_1 \sqrt{\log(3+2 ^{j_n+1} T)} \int_{\R \setminus [0,t]} \ \sum_{k_n \in D_{j_n}^1(t) } \frac{\prod_{\substack{\ell =1;\,\ell \neq n}}^d  \sqrt{\log(3+|j_\ell|+|2^{j_n} s|)}}{(3+|2^{j_n}s-k_n|)^L}   \, ds \nonumber \\
& \leq C_1 \int_{\R \setminus [0,t]}  \sum_{k_n \in D_{j_n}^1(t) } \frac{ \prod_{\substack{\ell =1;\,\ell \neq n}}^d \sqrt{\log(3+|j_\ell|+|2^{j_n} s-k_n|)}}{(3+|2^{j_n}s-k_n|)^L}   \, ds \times  \ldots \nonumber \\
& \qquad \ldots \times (\log(3+2 ^{j_n+1} T))^\frac{d}{2} , \label{eqn:int}
\end{align}
where $C_0$ and $C_1$ are positive almost surely finite random variables not depending on $n$, $t$, $\mathbf{j}$ and $J$. Let us estimate the last integral in \eqref{eqn:int}. First we bound it by the sum of the integrals $I^1_{\mathbf{j},\mathbf{k}}(t)$ and $I^2_{\mathbf{j},\mathbf{k}}$ where
\[ I^1_{\mathbf{j},\mathbf{k}}(t):= \int_{t}^{+ \infty}  \sum_{k_n \leq 2 ^{j_n}t-2^{j_n(1-a)} } \frac{\prod_{\substack{\ell =1;\, \ell \neq n}}^d \sqrt{\log(3+|j_\ell|+2^{j_n} s-k_n)}}{(3+2^{j_n}s-k_n)^L}   \, ds\]
and
\[ I^2_{\mathbf{j},\mathbf{k}} := \int_{- \infty}^{0}  \sum_{k_n \geq 2^{j_n(1-a)}} \frac{ \prod_{\substack{\ell =1;\, \ell \neq n}}^d \sqrt{\log(3+|j_\ell|+k_n-2^{j_n} s)}}{(3+k_n-2^{j_n}s)^L}   \, ds.\]
Concerning, $I^1_{\mathbf{j},\mathbf{k}}(t)$, we use the change of variable $y=2^{j_n}(s-t)$ and the fact that, for all $j \in \Z$, the function $y \mapsto (2+y)^{-L/(d-1)} \sqrt{\log(2+|j|+y)}$ is decreasing on $\R_+$ to get
\begin{align}
& I^1_{\mathbf{j},\mathbf{k}}(t) = 2^{-j_n} \int_0^{+ \infty} \sum_{k_n \leq 2 ^{j_n}t-2^{j_n(1-a)} } \prod_{\substack {\ell =1 \\ \ell \neq n}}^d\frac{ \sqrt{\log(3+|j_\ell|+y+2^{j_n} t-k_n)}}{(3+y+2^{j_n} t-k_n)^{L/(d-1)}}   \, dy \nonumber \\
&\leq 2^{-j_n} \int_0^{+ \infty} \sum_{p=0}^{+ \infty} \prod_{\substack {\ell =1 \\ \ell \neq n}}^d\frac{ \sqrt{\log(3+|j_\ell|+y+2^{j_n(1-a)}+p)}}{(3+y+2^{j_n(1-a)}+p)^{L/(d-1)}}   \, dy \nonumber \\
&\leq 2^{-j_n} \int_0^{+ \infty} \left(\int_0^{+ \infty} \prod_{\substack {\ell =1 \\ \ell \neq n}}^d \frac{ \sqrt{\log(2+|j_\ell|+y+2^{j_n(1-a)}+z)}}{(2+y+2^{j_n(1-a)}+z)^{L/(d-1)}} \, dz\right) \, dy \nonumber \\
&\leq 2^{-j_n} \int_0^{+ \infty} \left(\int_0^{+ \infty}  \frac{\prod_{\substack {\ell =1;\, \ell \neq n}}^d \sqrt{\log(2+|j_\ell|+y+2^{j_n(1-a)}+z)}}{(2+y+2^{j_n(1-a)}+z)^{L}} \, dz\right) \, dy . \label{eqn:I1bis}
\end{align}
We estimate the integral over $z$ in \eqref{eqn:I1bis} using integration by parts. Notice that there is no restriction to assume that $J$ is large enough so that the inequality $d-1\le\log(2+2^{J(1-a)})$ holds. Then, using the inequality $j_n\ge J$, we get, for all $(y,z)\in \R_+^2$ and for every $j_\ell \in \Z$ (with $l\ne n$) that 
\[\frac{d-1}{\sqrt{\log(2+|j_\ell|+y+2^{j_n(1-a)}+z)})} \leq  \sqrt{\log(2+|j_\ell|+y+2^{j_n(1-a)}+z)}.\]
Thus, denoting by $D_z$ the partial derivative operator with the variable $z$ and using the latter inequality, we get that
\begin{align*}
& D_z \prod_{\substack{\ell =1;\, \ell \neq n}}^d \sqrt{\log(2+|j_\ell|+y+2^{j_n(1-a)}+z)} \\
& = \sum_{\substack{\ell=1;\, \ell \neq n }}^d \,\frac{\prod_{\substack {i =1; \, i \neq n, \ell}}^d  \sqrt{\log(2+|j_i|+y+2^{j_i(1-a)}+z)}}{2  \sqrt{\log(2+|j_\ell|+y+2^{j_n(1-a)}+z)}(2+|j_\ell|+y+2^{j_n(1-a)}+z)} \\
& \leq \frac{\prod_{\substack {\ell =1;\, \ell \neq n}}^d\sqrt{\log(2+|j_\ell|+y+2^{j_n(1-a)}+z)}}{2 (2+y+2^{j_n(1-a)}+z)} 
\end{align*}
and consequently that
\begin{align*}
&\int_0^{+ \infty} \frac{\prod_{\substack {\ell =1;\, \ell \neq n}}^d \sqrt{\log(2+|j_\ell|+y+2^{j_n(1-a)}+z)}}{(2+y+2^{j_n(1-a)}+z)^L} \, dz \\
& \leq 2\times\frac{\prod_{\substack {\ell =1;\, \ell \neq n}}^d\sqrt{\log(2+|j_\ell|+y+2^{j_n(1-a)})}}{(2+y+2^{j_n(1-a)})^{L-1}}.
\end{align*}
This leads to
\begin{align}
I^1_{\mathbf{j},\mathbf{k}}(t) & \leq 2^{1-j_n} \int_0^{+ \infty}    \frac{\prod_{\substack {\ell =1;\, \ell \neq n}}^d\sqrt{\log(2+|j_\ell|+y+2^{j_n(1-a)})}}{(2+y+2^{j_n(1-a)})^{L-1}} \,  \, dy \nonumber \\
& \leq  2^{2-j_n}\times \frac{ \prod_{\substack {\ell =1;\, \ell \neq n}}^d\sqrt{\log(2+|j_\ell|+2^{j_n(1-a)})}}{(2+2^{j_n(1-a)})^{L-2}},  \label{eqn:I1}
\end{align}
where the last inequality is obtained through an integration by parts and the same arguments as before. Observe that, by using the definition of $I^2_{\mathbf{j},\mathbf{k}}$ one can show, as we already did it for deriving \eqref{eqn:I1}, that
\begin{equation}\label{eqn:I2}
I^2_{\mathbf{j},\mathbf{k}}\leq 2^{2-j_n}\times\frac{ \prod_{\substack {\ell =1;\, \ell \neq n}}^d \sqrt{\log(2+|j_\ell|+2^{j_n(1-a)})}}{(2+2^{j_n(1-a)})^{L-2}}.
\end{equation}
Next, it follows from the definition \eqref{ens:alephnj}, the inequalities \eqref{eqn:int}, \eqref{eqn:I1} and \eqref{eqn:I2}, the triangle inequality, the inequalities  \eqref{log:ineg1} and \eqref{log:ineg2} and the assumptions \eqref{eqn:hermi:cond} that
\begin{align*}
&\mathcal{H}^1_{n,J} := \sum_{\mathbf{j} \in \aleph_{n,J}} 2^{ j_1 (1-h_1)+\cdots + j_d (1-h_d)} \times \ldots \nonumber \\
& \qquad  \ldots \times\sup_{t \in [0,T]} \left \{\sum_{k_n \in D_{j_n}^1(t) } \sum_{\substack{k_\ell \in \Z \\ \ell\in [\![1,d]\!] \setminus \{n\}}} \left|\varepsilon_{\mathbf{j},\mathbf{k}}\right|  \left| \mathcal{A}_{\mathbf{j},\mathbf{k}}(t)- F_{\mathbf{j},\mathbf{k}}\right| \right\} \\
& \leq C_2 \sum_{j_n=J}^{+ \infty} (j_n+1)^{\frac{d}{2}}\, 2^{-j_n((L-2)(1-a)+h_n)} \times \ldots\nonumber \\
& \qquad  \ldots \times \prod_{\substack{\ell =1;\, \ell \neq n}}^d \left( \sum_{j_\ell = - \infty}^{j_n} 2^{j_\ell(1-h_\ell)} \sqrt{\log(2+|j_\ell|+2^{j_n(1-a)})} \right) \\
& \leq C_2 \sum_{j_n=J}^{+ \infty} (j_n+1)^{\frac{d}{2}}\, 2^{-j_n((L-2)(1-a)+h_1+\cdots+h_d-d+1)} \times \ldots \nonumber \\
& \qquad  \ldots \times \prod_{\substack{\ell =1;\, \ell \neq n}}^d \left( \sum_{p = 0}^{+\infty} 2^{-p(1-h_\ell)} \sqrt{\log(2+j_n+2^{j_n(1-a)}+p)} \right) \\
&  \leq C_3 \sum_{j_n=J}^{+ \infty} j_n^{d-1}\, 2^{-j_n((L-2)(1-a)+h_1+\cdots+h_d-d+1)} \\
& \leq C_4 J^{d-1}\, 2^{-J((L-2)(1-a)+h_1+\cdots+h_d-d+1)}
\end{align*}
where $C_2$, $C_3$ and $C_4$ are positive almost surely finite random variables not depending on $n$, $t$ and $J$ large enough.
\end{proof}

\begin{Def}
\label{def:rvs-M}
Let $T>2$ be a fixed real number. For all $t \in [0,T]$, $n \in [\![1,d]\!]$ and $\mathbf{j} \in \aleph_{n,1}$ (see \eqref{ens:alephnj}), the random variable $\widetilde{\mathcal{M}}_{n,\mathbf{j}}  (t)$ is defined as:
\begin{equation}\label{eqn:tiMnJ}
\widetilde{\mathcal{M}}_{n,\mathbf{j}}  (t) := \sum_{ \textbf{k} \in \daleth_{n}(t)} F_{\textbf{j},\textbf{k}} \varepsilon_{\textbf{j},\textbf{k}}, 
\end{equation} 
where the $F_{\textbf{j},\textbf{k}}$'s are as in \eqref{eqn:fjk} and the finite set $\daleth_{n}(t)\subset\Z^d$ is such that
\begin{equation}\label{eqn:set-tiMnJ}
\daleth_{n}(t) := \{\mathbf{k} \in \Z^d \, : \, k_n \in D_{j_n}^1(t) ; \, \forall \ell\in [\![1,d]\!] \setminus \{n\}, \, |k_\ell| \leq 2^{j_n +1} T \};
\end{equation}
the finite set $D_{j_n}^1(t)\subset\Z$ being defined through \eqref{def:dj1}. Finally, for evey $J\in\N$, we set
\begin{equation}\label{eqn:MnJ}
\mathcal{M}_{n,J} := \sum_{\mathbf{j} \in \aleph_{n,J}} 2^{j_1(1-h_1)+ \cdots j_d(1-h_d)} \sup_{t \in [0,T]}  |\widetilde{\mathcal{M}}_{n,\mathbf{j}}  (t) |.
\end{equation}
We mention in passing that $\aleph_{n,J}\subseteq \aleph_{n,1}$. Also, we mention that later we will show (see Remark \ref{rem:finiD1}) that the supremum in \eqref{eqn:MnJ} is in fact a supremum on a well-chosen finite set; therefore $\mathcal{M}_{n,J}$ is a positive finite random variable.
\end{Def}

Our next goal is to obtain an appropriate upper bound for $\mathcal{M}_{n,J}$. To this end, we need to bound in a convenient way the random variables $\widetilde{\mathcal{M}}_{n,\mathbf{j}}  (t)$. In order to do so we will combine some Borel-Cantelli arguments with the following fundamental result \cite[Theorem 6.7]{MR1474726}.

\begin{Lemma}\label{lem:jas}
For any fixed integer $d \geq 1$, there exists a (strictly) positive finite universal deterministic constant $c_d$ such that, for every random variable $X$ belonging to the Wiener chaos of order $d$ and for each real number $y \geq 2$, one has
\begin{equation}
\mathbb{P}(X \geq y \|X\|_{L^2(\Omega)} ) \leq \exp\big (-c_d\, y^{2/d}\,\big ).
\end{equation}
\end{Lemma}
We will apply Lemma \ref{lem:jas} to the random variable $\widetilde{\mathcal{M}}_{n,\mathbf{j}}  (t)$. This is why it is useful to control their $L^2(\Omega)$ norms uniformly in $t\in [0,T]$.

\begin{Lemma}
There exists a finite constant $c>0$, depending on $T$, such that, for all $n \in [\![1,d]\!] $ and $\mathbf{j} \in \aleph_{n,1}$, we have 
\begin{equation}\label{eqn:majl2}
\sup_{t \in [0,T]} \| \widetilde{\mathcal{M}}_{n,\mathbf{j}}  (t) \|_{L^2(\Omega)} \leq c\, 2^{-j_n /2}.
\end{equation}
\end{Lemma}

\begin{proof}
The equivalence relation $\sim$ on the set $\daleth_{n}(t)$ is defined as:
\[
\forall \,(\textbf{k},\textbf{k}') \in \daleth_{n}(t)\times\daleth_{n}(t),\,\,\,\, \textbf{k}\sim \textbf{k}'  \Longleftrightarrow \varepsilon_{\textbf{j,k}} = \varepsilon_{\textbf{j,k}'}.
\]
Let us emphasize that, we know from Remark \ref{rem:eqa-multi-jk} and Proposition \ref{prop:corre} that 
\begin{equation}
\label{eq:equiv-R}
\forall \,(\textbf{k},\textbf{k}') \in \daleth_{n}(t)\times\daleth_{n}(t),\,\,\,\, \textbf{k}\sim \textbf{k}'  \Longleftrightarrow \E[\varepsilon_{\textbf{j,k}}\varepsilon_{\textbf{j,k}'}] \neq 0. 
\end{equation}
Since $\daleth_{n}(t)$ is a finite set, the equivalence classes for the equivalence relation $\sim$ are in finite number denoted by $M$. Let us then denote them by $\daleth_{n,1}(t),\ldots, \daleth_{n,M}(t)$. Then using a well-known result on equivalence relations, the set $\daleth_{n}(t)$ can be expressed as:
\begin{equation}
\label{eq:Uequi-cl}
\daleth_{n}(t)=\bigcup_{i=1}^M \daleth_{n,i}(t)\quad\mbox{(disjoint union).}
\end{equation} 
Let us also mention from Remark \ref{rem:eqa-multi-jk} that, for each $i\in  [\![1,M]\!]$, we have
\begin{align}\label{eqn:boundcardi}
\card (\daleth_{n,i}(t)) \le d\,!
\end{align}
Next, observe that it follows from \eqref{eq:equiv-R} that 
\begin{equation}
\label{eq:cov-ze}
{\rm Cov}\bigg (\,\sum_{\textbf{k}\in \daleth_{n,i}'(t)} F_{\textbf{j,k}} \varepsilon_{\textbf{j,k}} \,,\, \sum_{\textbf{k}\in \daleth_{n,i}''(t)} F_{\textbf{j,k}} \varepsilon_{\textbf{j,k}} \bigg)=0, \quad\mbox{when $i'\ne i''$}.
\end{equation}
Then, one can derive from \eqref{eqn:tiMnJ}, \eqref{eq:Uequi-cl}, \eqref{eq:cov-ze}, Proposition \ref{prop:corre} and the triangular inequality,  that
\begin{align*}
& \| \widetilde{\mathcal{M}}_{n,\mathbf{j}}  (t) \|_{L^2(\Omega)}^2 = \left\|\sum_{i=1}^M \sum_{\textbf{k}\in \daleth_{n,i}(t)} F_{\textbf{j,k}} \varepsilon_{\textbf{j,k}} \right\|_{L^2(\Omega)}^2 = \sum_{i=1}^M \left\|\sum_{\textbf{k}\in \daleth_{n,i}(t)} F_{\textbf{j,k}} \varepsilon_{\textbf{j,k}} \right\|_{L^2(\Omega)}^2 \\
& = \sum_{i=1}^M \sum_{\textbf{k}\in\daleth_{n,i}(t)} \sum_{\textbf{k}'\in\daleth_{n,i}(t)} F_{\textbf{j,k}} F_{\textbf{j,k}'} \E\left[\varepsilon_{\textbf{j,k}} \varepsilon_{\textbf{j,k}'}\right] \\
& \le d\,!\sum_{i=1}^M \sum_{\textbf{k}\in\daleth_{n,i}(t)} \sum_{\textbf{k}'\in\daleth_{n,i}(t)} \big | F_{\textbf{j,k}}\big |\big |F_{\textbf{j,k}'}\big | \\
&= d\,!\sum_{i=1}^M \left (\sum_{\textbf{k}\in\daleth_{n,i}(t)} \big | F_{\textbf{j,k}}\big |\right)^2\,.
\end{align*}
Then using the convexity of the function $x\mapsto x^2$, the inequality \eqref{eqn:boundcardi} and the equality \eqref{eq:Uequi-cl}, we get that
\begin{align}\label{eqn:412intermediaire}
\| \widetilde{\mathcal{M}}_{n,\mathbf{j}}  (t) \|_{L^2(\Omega)}^2 & \leq (d\, !)^2\sum_{i=1}^M \sum_{\textbf{k}\in\daleth_{n,i}(t)} F_{\textbf{j,k}}^2 =  (d\, !)^2 \sum_{\textbf{k}\in\daleth_{n}(t)} F_{\textbf{j,k}}^2\,.
\end{align}
Moreover, putting together Lemma \ref{lem:boundF}, \eqref{eqn:set-tiMnJ}, the fast decay property \eqref{maj:psih} with $L>1$, and the inequality $\displaystyle\sup_{x\in\R}\sum_{k \in \Z} \big (3+|x-k|\big)^{-L}<\infty$ we get that
\begin{align}
\sum_{\textbf{k}\in\daleth_{n}(t)} F_{\textbf{j,k}}^2 &\leq c_\psi\, 2^{-j_n} \sum_{k_n \in D_{j_n}^1(t)} \int_{\R}  \left|\psi_{h_n} (2^{j_n}s-k_n) \right| \prod_{\substack{\ell =1 \\ \ell \neq n}}^d \left(  \sum_{k_\ell \in \Z} \left|\psi_{h_\ell} (2^{j_\ell}s-k_\ell) \right|^2 \right)\, ds \nonumber \\
& \leq c_1  2^{-j_n} \sum_{k_n \in D_{j_n}^1(t)} \int_{\R}  \left|\psi_{h_n} (2^{j_n}s-k_n) \right| \prod_{\substack{\ell =1 \\ \ell \neq n}}^d \left(  \sum_{k_\ell \in \Z} \big(3+|2^{j_\ell}s-k_\ell|\big)^{-L} \right)\, ds \nonumber \\
& \leq c_2  2^{-j_n} \sum_{k_n \in D_{j_n}^1(t)} \int_{\R}  \left|\psi_{h_n} (2^{j_n}s-k_n) \right| \, ds\nonumber\\
&= c_2\| \psi_{h_n}\|_{L^1(\R)}\,  2^{-2 j_n}\, \card (D_{j_n}^1(t))\,,  \label{boundsuppfjk}
\end{align}
where $c_1,c_2$ are finite positive constants not depending on $n$, $t$ and $\mathbf{j}$. Then, \eqref{eqn:412intermediaire}, \eqref{boundsuppfjk},  and \eqref{maj:carddj1} entail that $\| \widetilde{\mathcal{M}}_{n,\mathbf{j}}  (t) \|_{L^2(\Omega)}^2\le (d\, !)^2 c_2 2^{-j_n}$.
\end{proof}
The following remark shows that the supremum in \eqref{eqn:MnJ} is in fact a supremum on a well-chosen finite set.
\begin{Rmk}
\label{rem:finiD1}
For each fixed $j \in \N$ and $t\in\R_+$, we denote by $m_{j,t}$ the integer part of the real number  $2^jt-2^{j(1-a)}$, that is $m_{j,t}:=\lfloor 2^jt-2^{j(1-a)} \rfloor$. Thus, in view of \eqref{def:dj1} of the set $D_j^1(t)$, it turns out that
\begin{equation}
\label{rem:finiD1bis}
D_j^1(t)=\left\{ \begin{array}{cl}
                       \emptyset & \text{if } t \in [0, 2^{1-ja})\\
                       D_j^1(m_{j,t}2^{-j}+2^{-ja}) & \text{if } t \in [2^{1-ja},\infty).
                     \end{array} \right.
\end{equation}
Then, we can derive from \eqref{rem:finiD1bis} that, for all $n \in [\![1,d]\!] $ and $\mathbf{j} \in \aleph_{n,1}$,
\begin{equation}\label{eqn:discret}
\sup_{t \in [0,T]} |\widetilde{\mathcal{M}}_{n,\mathbf{j}}(t)|=\sup_{m \in \mathcal{I}_{j_n}} |\widetilde{\mathcal{M}}_{n,\mathbf{j}}(m2^{-j_n}+2^{-j_na}) |,
\end{equation}
where the arbitrary real number $T>2$ is fixed and $\mathcal{I}_j$ stands for the finite set
\begin{equation}\label{eqn:Ij}
\mathcal{I}_j := \N \cap (2 ^{j(1-a)}-1,2^jT-2^{j(1-a)}].
\end{equation}
\end{Rmk}

\begin{Lemma}\label{lem:eventomega**}
Let $T >2$ be a fixed real number. There exist $\Omega^{**}$ an event of probability $1$ and a positive almost surely finite random variable $C^{**}$ such that,  for all $n \in [\![1,d]\!] $ and $\mathbf{j} \in \aleph_{n,1}$, on $\Omega^{**}$, we have
\begin{equation}\label{bound:mnjl2}
\sup_{t \in [0,T]} |\widetilde{\mathcal{M}}_{n,\mathbf{j}}(t)| \leq C^{**}  2^{-j_n /2}\,\log(3+|\mathbf{j}|+ 2^{j_n} T)^\frac{d}{2} .
\end{equation}
\end{Lemma}

\begin{proof}
Let us first show that if $(X_j)_{j \in \N}$ is an arbitrary sequence of random variables in the Wiener chaos of order $d$, there exist $\Omega_1$, an event of probability $1$, and a positive almost surely finite random variable $C_1$ such that, for all $j \in \N$, on $\Omega_1$, we have
\begin{equation}
\label{eq1:seqMj}
|X_j| \leq C_1 \log(3+j)^\frac{d}{2} \|X_j\|_{L^2(\Omega)}.
\end{equation}
Let $\kappa \geq 2$ be a constant which will be precisely defined later. Applying, for any $j \in \N$, Lemma \ref{lem:jas} to the random variable $X_j$, we get that
\begin{equation}\label{maj:probajas}
\mathbb{P} \left( |X_j| \geq \kappa \log(3+j)^\frac{d}{2} \|X_j\|_{L^2(\Omega)} \right) \leq \exp\big (-c_d \kappa^\frac{2}{d} \log(3+j)\big),
\end{equation}
where $c_d$ is the same universal positive constant as in Lemma \ref{lem:jas}. 
Thus, assuming that the constant $\kappa$ satisfies $\kappa > c_d^{-\frac{d}{2}}$, it turns out that the series
\[ \sum_{j \in \N} \mathbb{P} \left( |X_j| \geq \kappa \log(3+j)^\frac{d}{2} \|X_j\|_{L^2(\Omega)} \right)  \]
is convergent; then, the existence of $\Omega_1$ and $C_1$ follows from Borel-Cantelli Lemma.
Next, notice that, thanks to an indexation argument, the result obtained in \eqref{eq1:seqMj} can be applied to 
the sequence of random variables 
\[
\left\{\widetilde{\mathcal{M}}_{n,\mathbf{j}}(m2^{-j_n}+2^{-j_na})\,: \,n \in [\![1,d]\!],\, \textbf{j} \in \aleph_{n,1}, m\in I_{j_n}\right\}.
\]
By this way, we can show that there are $\Omega^{**}$ an event of probability $1$ and a positive almost surely finite random variable $C_2$ (depending on $T$) such that, on $\Omega^{**}$, we have, for all $n \in [\![1,d]\!] $, $\textbf{j} \in \aleph_{n,1}$ and  $m\in I_{j_n}$, that
\begin{align} \label{maj:Mnjm}
|\widetilde{\mathcal{M}}_{n,\mathbf{j}}(m2^{-j_n}+2^{-j_na})| \leq C_2 \log(3+|\mathbf{j}|+m)^\frac{d}{2} \|\widetilde{\mathcal{M}}_{n,\mathbf{j}}(m2^{-j_n}+2^{-j_na})\|_{L^2(\Omega)}. 
\end{align}
Then, putting together \eqref{maj:Mnjm}, \eqref{eqn:majl2}, and Remark \ref{rem:finiD1}, we obtain \eqref{bound:mnjl2}. 
\end{proof}

\begin{Lemma}\label{lemmepour22:5}
Let $T >2$ be a fixed real number. There exits a positive almost surely finite random variable $C$ such that, for all $n \in [\![1,d]\!] $ and $J \in \N$, on $\Omega^{**}$ (see Lemma \ref{lem:eventomega**}), the random variable $\mathcal{M}_{n,J}$ (see (\ref{eqn:MnJ})) is bounded from above by
$ C J^{\frac{d}{2}} 2^{-J(h_1+\cdots+h_d -d+ \frac{1}{2})}$.
\end{Lemma}

\begin{proof}
Let us fix $J \in \N$, using \eqref{eqn:MnJ}, \eqref{ens:alephnj}, \eqref{bound:mnjl2}, \eqref{log:ineg1}, the triangular inequality,  \eqref{log:ineg2} and \eqref{eqn:hermi:cond}, we obtain that
\begin{align*}
 \mathcal{M}_{n,J} &\leq C_0 \sum_{j_n=J}^{+ \infty} 2^{j_n(\frac{1}{2}-h_n)} \log(3+ d\,j_n+2^{j_n} T)^\frac{d}{2} \times \ldots \nonumber \\
& \ldots \times \prod_{\substack{\ell=1 \\\ell \neq d}}^d \left( \sum_{j_\ell =- \infty}^{j_n} 2^{j_\ell(1-h_\ell)} \log(3+|j_n- j_\ell|)^\frac{d}{2}  \right)\\
& \leq C_1 \sum_{j_n=J}^{+ \infty} 2^{-j_n(h_1+\cdots +h_d -d+ \frac{1}{2})} \log(3+ d\, j_n+2^{j_n} T)^\frac{d}{2} \times \ldots \nonumber \\
& \ldots \times \prod_{\substack{\ell=1 \\\ell \neq d}}^d \left( \sum_{p=0}^{+ \infty} 2^{-p(1-h_\ell)} \log(3+p +2^{j_n}T)^\frac{d}{2}  \right) \\
& \leq C_2  \sum_{j_n=J}^{+ \infty} 2^{-j_n(h_1+\cdots +h_d -d+ \frac{1}{2})} j_n^\frac{d}{2} \\
& \leq C_3 2^{-J(h_1+\cdots +h_d -d+ \frac{1}{2})} J^\frac{d}{2}.
\end{align*}
where $C_0$, $C_1$, $C_2$ and $C_3$ are positive almost surely finite random variables not depending on $n$ and $J$. 
\end{proof}

We are now in position to complete the proof of Theorem \ref{thm:main1}.

\begin{proof}[End of the Proof of Theorem \ref{thm:main1}]
Without loss of generality, one can assume that the compact interval $I$ in the statement of the theorem is of the form $I=[0,T]$ for a fixed real number $T>2$. Let $\widetilde{\Omega}$ be the event of probability $1$ defined as: $\widetilde{\Omega}:= \Omega^* \cap \Omega^{**}$, where $\Omega^*$ and $\Omega^{**}$ are as in Lemmata \ref{lem:bound-gjk} and  \ref{lem:eventomega**}. 

First, we will show that, for each fixed $\omega \in \widetilde{\Omega}$ and $\mathbf{j} \in \Z^d$, the series of continuous function $\sum_{\mathbf{k}\in \Z^d} {\cal A}_{\mathbf{j},\mathbf{k}}\varepsilon_{\mathbf{j},\mathbf{k}}(\omega)$ is normally convergent with respect to the uniform norm $\| \cdot \|_{I,\infty}$.
Using the fast decay property \eqref{maj:psih}, the definition \eqref{def:AJK} of the continuous function ${\cal A}_{\mathbf{j},\mathbf{k}}$, the bound \eqref{eqn:bound:rv} and the triangular inequality, one gets, for some positive finite random variable $C_1$, depending on $T$ and $\mathbf{j}\in \Z^d$ and that
\begin{align*}
&\sum_{\mathbf{k} \in \Z^d} \|{\cal A}_{\mathbf{j},\mathbf{k}}\|_{I,\infty} |\varepsilon_{\mathbf{j},\mathbf{k}}(\omega)| \\
& \leq C_1(\omega) \sum_{\mathbf{k} \in \Z^d} \int_0^T \prod_{\ell=1}^d \frac{\sqrt{\log(3+|j_\ell|+|k_\ell|)}}{(1+2^{j_\ell}T+|2^{j_\ell}s-k_\ell|)^2} \, ds \\
& \leq C_1(\omega) \sum_{\mathbf{k} \in \Z^d} \int_0^T \prod_{\ell=1}^d \frac{\sqrt{\log(3+|j_\ell|+|k_\ell|)}}{(1+2^{j_\ell}T+|k_\ell|-|2^{j_\ell}s|)^2} \, ds \\
& \leq  C_1(\omega) T \sum_{\mathbf{k} \in \Z^d} \frac{\sqrt{\log(3+|j_\ell|+|k_\ell|)}}{(1+|k_\ell|)^2} < \infty,
\end{align*}
which shows that the normal convergence holds.

Next, for each $\mathbf{j}\in \Z^d$, we denote by $\{X_\mathbf{j}(t)\}_{t \in I}$ the stochastic process with continuous paths vanishing outside of $\widetilde{\Omega}$ and defined on $I \times \widetilde{\Omega}$ as
\begin{equation}
\label{eqn:proof22ter}
X_\mathbf{j}(t,\omega)=\sum_{\mathbf{k}\in \Z^d} {\cal A}_{\mathbf{j},\mathbf{k}}(t) \varepsilon_{\mathbf{j},\mathbf{k}}(\omega). 
\end{equation}
Next observe that in order to complete the proof of the theorem, it is enough to show that there exists a positive finite random variable $\widetilde{C}$ such that, for every $J \in \N$, the following inequality holds on $\widetilde{\Omega}$:
\begin{equation}\label{eqn:proof22}
\sum_{\substack{\mathbf{j} \in \Z^d \\  \max_{\ell \in [\![1,d]\!]} j_\ell \geq J}} 2^{j_1(1-h_1)+ \cdots +j_d(1-h_d)} \| X_\mathbf{j}\|_{I,\infty} \leq \widetilde{C} J^{\frac{d}{2}} 2^{-J(h_1+\cdots+h_d-d+\frac{1}{2})}.
\end{equation}
Indeed, assuming that \eqref{eqn:proof22} is true, then it clearly entails that, for all fixed $J \in \N$ and every $\omega \in \widetilde{\Omega}$, one has
\[ \sum_{\substack{\mathbf{j} \in \Z^d \\  \max_{\ell \in [\![1,d]\!]} j_\ell \geq J}} 2^{j_1(1-h_1)+ \cdots +j_d(1-h_d)} \| X_\mathbf{j}(\cdot,\omega)\|_{I,\infty} <,\infty,\]
which means that the series of continuous function
\begin{equation}\label{eqn:proof22bis}
X_J(\cdot, \omega) := \sum_{\substack{\mathbf{j} \in \Z^d \\  \max_{\ell \in [\![1,d]\!]} j_\ell \geq J}} 2^{j_1(1-h_1)+ \cdots +j_d(1-h_d)} X_\mathbf{j}(\cdot,\omega)
\end{equation}
is normally convergent with respect to the uniform norm $\| \cdot \|_{I,\infty}$ and thus $X_J(\cdot, \omega)$ is a continuous function on $I$. Then, we denote by $\{X_J(t)\}_{t \in I}$ the stochastic process with continuous paths vanishing outside of $\widetilde{\Omega}$ and defined on $I \times \widetilde{\Omega}$ by \eqref{eqn:proof22bis}. Thus, \eqref{eqn:proof22} and the triangular inequality imply that, for all $J \in \N$, the following inequality holds on $\widetilde{\Omega}$:
\begin{equation}\label{eqn:prooofthm:main1}
\| X_J \|_{I,\infty} \leq \widetilde{C} J^{\frac{d}{2}} 2^{-J(h_1+\cdots+h_d-d+\frac{1}{2})}.
\end{equation}
On another hand, we know from equality \eqref{eq:X-XJ} that, for all fixed $J \in \N$ and $t \in I$, the random series
\[ \sum_{\substack{\mathbf{j} \in \Z^d \\  \max_{\ell \in [\![1,d]\!]} j_\ell \geq J}} 2^{j_1(1-h_1)+ \cdots +j_d(1-h_d)}X_\mathbf{j}(t)\]
converges to $X^{(d,\bot)}_{\mathbf{h},J}(t):=X^{(d)}_{\mathbf{h}}(t)-X^{(d)}_{\mathbf{h},J}(t)$ in $L^2(\Omega)$. Combining this fact with \eqref{eqn:proof22bis} one concludes that, for all $t \in I$, almost surely,
\[ X_J(t)=X^{(d)}_{\mathbf{h}}(t)-X^{(d)}_{\mathbf{h},J}(t).\]
This latter equality and the fact that the two stochastic processes $\{X_J(t)\}_{t \in I}$ and $\{ X^{(d)}_{\mathbf{h}}(t)-X^{(d)}_{\mathbf{h},J}(t)\}_{t \in I}$ have continuous paths imply that these two processes are indistinguishable. Thus, the inequality \eqref{eqn:prooofthm:main1} is nothing else than the inequality \eqref{eqn:thm:main1}.

It remains us to show that \eqref{eqn:proof22} holds. In fact, it results from Lemmata \ref{lemmepour22:1}, \ref{lemmepour22:2}, \ref{lemmepour22:3}, \ref{lemmepour22:4}, \ref{lemmepour22:5} and the inequality:
\[ \sum_{\mathbf{j} \in \Z^d \, : \,  \displaystyle\max_{\ell \in [\![1,d]\!]} j_\ell \geq J} 2^{j_1(1-h_1)+ \cdots +j_d(1-h_d)} \| X_\mathbf{j}\|_{I,\infty} \leq \sum_{n=1}^{d} \left( \mathcal{M}_{n,J}+ \sum_{m=0}^3 \mathcal{H}_{n,J}^m \right), \]
which is obtained by using \eqref{eqn:proof22ter}, the triangular inequality, standard computations, and the definitions of the random variables $\mathcal{M}_{n,J}$, and $\mathcal{H}_{n,J}^m$ with $m\in [\![0,3]\!]$.  
\end{proof}

\section{Proof of Theorem \ref{thm:main2}}
\label{sec:altern-rep}

In this final Section, we aim at proving Theorem \ref{thm:main2}. We will need a number of intermediary results which mainly consist in bounding in convenient ways well-chosen parts of the random series in \eqref{eqn:fullserie}. In fact, most of the job has already been done in the previous Section \ref{sect:approxerro} but there still remain some parts of the series which need to be conveniently bounded. It is the purpose of the following lemmata. We mention that the event $\Omega^*$ of probability $1$ (see Lemma \ref{lem:bound-gjk}) will appear in their statements.

\begin{Lemma}\label{lemmepour24:1}
Let $T >2$ and $L>3/2$ be two fixed real numbers. There exits a positive almost surely finite random variables $C$ such that, for all $n \in [\![1,d]\!] $ and $N \in \N$, on $\Omega^*$, the random variable
\begin{align}
&\mathcal{L}^1_{n,N} := \sum_{\mathbf{j} \in \Z^d \, : \,   \max_{\ell \in [\![1,d]\!]} j_\ell <N} 2^{ j_1 (1-h_1)+\cdots + j_d (1-h_d)} \times \ldots \nonumber \\
& \qquad \ldots \times\sum_{|k_n| > 2^{N+1} T } \sum_{\substack{k_\ell \in \Z \\ \ell\in [\![1,d]\!] \setminus \{n\}}} \left|\varepsilon_{\mathbf{j},\mathbf{k}}\right|  \sup_{t \in [0,T]} \left \{\left| \mathcal{A}_{\mathbf{j},\mathbf{k}}(t)\right| \right\} \label{def:lnN1}
\end{align}
is bounded from above by $ C  2^{-N(h_1+ \cdots + h_d+L -d-1)} N^\frac{d}{2} \log(3+N)^\frac{d}{2}.$
\end{Lemma}

\begin{proof}
Let us fix $N \in \N$ and $\mathbf{j}\in \Z^d$ such that $\max_{\ell \in [\![1,d]\!]} j_\ell <N$. Using  the definition \eqref{def:AJK}, the fast decay property \eqref{maj:psih}, the bound \eqref{eqn:bound:rv}, the triangular inequality, inequalities \eqref{log:ineg1} and \eqref{maj:somme1} and the fact that the function $y \mapsto(2+y)^{-L} \sqrt{\log(2+y)}$ is decreasing on $\R_+$, we get
\begin{align}
& \sum_{|k_n| > 2^{N+1} T } \sum_{\substack{k_\ell \in \Z \\ \ell\in [\![1,d]\!] \setminus \{n\}}} \left|\varepsilon_{\mathbf{j},\mathbf{k}}\right|  \sup_{t \in [0,T]} \left \{\left| \mathcal{A}_{\mathbf{j},\mathbf{k}}(t)\right| \right\} \nonumber \\ 
& \leq C_0 \int_0^T \sum_{ |k_n| > 2^{N+1}T} \frac{\sqrt{\log(3+|j_n|+|k_n|)}}{(3+|2^{j_n}s-k_n|)^L} \times \ldots \nonumber \\
& \qquad \ldots \times \prod_{\substack{\ell=1 \\ \ell\neq n}}^d \left( \sum_{k_\ell \in \Z} \frac{\sqrt{\log(3+|j_\ell|+|k_\ell|)}}{(3+|2^{j_\ell}s-k_\ell|)^L} \right) \, ds \nonumber \\
& \leq  C_1 \int_0^T \sum_{ |k_n| > 2^{N+1}T} \frac{\sqrt{\log(3+|j_n|+|k_n|)}}{(3+|k_n|- 2^{j_n}s)^L} \prod_{\substack{\ell=1 \\ \ell\neq n}}^d \left( \sqrt{\log(3+|j_\ell|+2^{j_\ell} s)} \right) \, ds \nonumber \\
&\leq  C_1 T 2^L  \sqrt{\log(3+|j_n|)} \left( \prod_{\substack{\ell=1 \\ \ell\neq n}}^d  \sqrt{\log(3+|j_\ell|+2^{j_\ell} T)} \right) \times \ldots \nonumber \\
&\qquad \ldots \times  \left( \sum_{ |k_n| > 2^{N+1}T} \frac{\sqrt{\log(3+|k_n|)}}{(3+|k_n|)^L} \right) \nonumber \\
& \leq  C_1 T 2^{L+1}  \sqrt{\log(3+|j_n|)} \left(\prod_{\substack{\ell=1 \\ \ell\neq n}}^d  \sqrt{\log(3+|j_\ell|+2^{j_\ell} T)} \right) \times \ldots \nonumber \\
& \qquad\ldots \times  \left( \int_{2^{N+1}T}^{+ \infty} \frac{\sqrt{\log(2+y)}}{(2+y)^L} \right) \nonumber \\
& \leq C_2 \sqrt{\prod_{\ell=1}^d \log(3+|j_\ell|)} N^\frac{d}{2} 2^{-N(L-1)},\label{eqn:boundl1}
\end{align}
where $C_0$, $C_1$ and $C_2$ are positive almost surely finite random variables not depending on $\mathbf{j}$ and $N$. It follows from \eqref{eqn:boundl1} that
\begin{align*}
\mathcal{L}^1_{n,N} & \leq C_2  N^\frac{d}{2} 2^{-N(L-1)} \sum_{j_1 =- \infty}^{N-1} \cdots \sum_{j_d =- \infty}^{N-1} 2^{ j_1 (1-h_1)+\cdots + j_d (1-h_d)} \sqrt{\prod_{\ell=1}^d \log(3+|j_\ell|)} \\
& \leq C_3 N^\frac{d}{2} \log(3+N)^\frac{d}{2} 2^{-N(h_1+\cdots+h_d+L-d-1)},
\end{align*}
where $C_3$ is a positive almost surely finite random variable not depending on $N$.
\end{proof}

\begin{Lemma}\label{lemmepour24:2}
Let $T >2$, $L>1$ and $g>0$ be three fixed real numbers. There exits a positive almost surely finite random variable $C$ such that, for all $n\in[\![1,d]\!] $ and $N \in \N$, on $\Omega^*$, the random variable
\begin{align}
\mathcal{L}^2_{n,N} :=  \sum_{\mathbf{j} \in -\N^d} 2^{ j_1 (1-h_1)+\cdots + j_d (1-h_d)} \sum_{|k_n| > 2^{Ng} } \sum_{\substack{k_\ell \in \Z \\ \ell \in [\![1,d]\!], \ell \neq n}} \left|\varepsilon_{\mathbf{j},\mathbf{k}}\right|  \sup_{t \in [0,T]} \left \{\left| \mathcal{A}_{\mathbf{j},\mathbf{k}}(t)\right| \right\} \label{def:lnN2}
\end{align}
is bounded from above by $ C  2^{-N(L-1)g} \sqrt{N}. $
\end{Lemma}
\begin{proof}
Let ux fix $N \in \N$ and $\mathbf{j}\in - \N^d$. Using the definition \eqref{def:AJK}, the fast decay property \eqref{maj:psih}, the bound \eqref{eqn:bound:rv}, the inequality \eqref{log:ineg1}, the triangular inequality and the fact that the function $y \mapsto(2+y)^{-L} \sqrt{\log(2+y)}$ is decreasing on $\R_+$, we get
\begin{align}
& \sum_{|k_n| > 2^{Ng} } \sum_{\substack{k_\ell \in \Z \\ \ell\in [\![1,d]\!] \setminus \{n\}}} \left|\varepsilon_{\mathbf{j},\mathbf{k}}\right|  \sup_{t \in [0,T]} \left \{\left| \mathcal{A}_{\mathbf{j},\mathbf{k}}(t)\right| \right\} \nonumber \\
& \leq C_0 \int_0^T \sum_{ |k_n| > 2^{N g}} \frac{\sqrt{\log(3+|j_n|+|k_n|)}}{(3+T+|2^{j_n}s-k_n|)^L} \times \ldots \nonumber \\
& \qquad \ldots \times \prod_{\substack{\ell=1 \\ \ell\neq n}}^d \left( \sum_{k_\ell \in \Z} \frac{\sqrt{\log(3+|j_\ell|+|k_\ell|)}}{(3+T+|2^{j_\ell}s-k_\ell|)^L} \right) \, ds \nonumber \\
& \leq C_0 \sqrt{\prod_{\ell=1}^d \log(3+|j_\ell|)} \int_0^T \sum_{ |k_n| > 2^{N g}} \frac{\sqrt{\log(3+|k_n|)}}{(3+|k_n|)^L} \times \ldots \nonumber \\
& \qquad \ldots \times \prod_{\substack{\ell=1 \\ \ell\neq n}}^d \left( \sum_{k_\ell \in \Z} \frac{\sqrt{\log(3+|k_\ell|)}}{(3+|k_\ell|)^L} \right) \, ds \nonumber \\ 
& \leq C_1  \sqrt{\prod_{\ell=1}^d \log(3+|j_\ell|)}  \sum_{ |k_n| > 2^{N g}} \frac{\sqrt{\log(3+|k_n|)}}{(3+|k_n|)^L} \nonumber \\
& \leq 2 C_1  \sqrt{\prod_{\ell=1}^d \log(3+|j_\ell|)}  \int_{2^{N g}}^{+ \infty} \frac{\sqrt{\log(2+y)}}{(2+y)^L} \nonumber\\
& \leq C_2 \sqrt{\prod_{\ell=1}^d \log(3+|j_\ell|)} 2^{-N(L-1)g} \sqrt{N}, \label{eqn:boundl2}
\end{align}
where $C_0$, $C_1$ and $C_2$ are positive almost surely finite random variables not depending on $\mathbf{j}$ and $N$. It follows from \eqref{eqn:boundl2} that
\begin{align*}
\mathcal{L}^2_{n,N} & \leq C_2 2^{-N(L-1)g} \sqrt{N} \sum_{\mathbf{j} \in -\N^d} 2^{ j_1 (1-h_1)+\cdots + j_d (1-h_d)} \sqrt{\prod_{\ell=1}^d \log(3+|j_\ell|)} \\
& \leq C_3 2^{-N(L-1)g} \sqrt{N},
\end{align*}
where $C_3$ is a positive almost surely finite random variables not depending on $N$.
\end{proof}

\begin{Def}
If $b >0$ is a fixed real numbers, we define the set of boolean vector
\[ \mathcal{B} := \Big\{v=(v_l)_{\ell \in [\![1,d]\!]}  \in \{0,1\}^d : \exists\, (l',l'')\in [\![1,d]\!]^2  \text{ such that } l'\ne l'' \text{ and } v_{l'}\ne v_{l''} \Big\} \]
and, for all $v \in \mathcal{B}$ and $N \in \N$, the subset of $\Z^d$
\begin{multline*}
\gimel_{v,N}^3 := \Big\{ \mathbf{j} \in \Z^d \, : \, \forall \,\ell \in  [\![1,d]\!], j_\ell \in  [\![0,N]\!]\text{ if } v_\ell=0 \text{ and, otherwise, } j_\ell \leq -2^{N b} \Big\}.     
\end{multline*}

\end{Def}

\begin{Lemma}\label{lemmepour24:3}
Let $T >2$ and $b>0$ be two fixed real numbers. There exits a positive almost surely finite random variable $C$ such that, for all $v \in \mathcal{B}$ and $N \in \N$, on $\Omega^*$, the random variable
\begin{align}
\mathcal{L}^3_{v,N} :=  \sum_{\mathbf{j} \in \gimel_{v,N}^3} 2^{ j_1 (1-h_1)+\cdots + j_d (1-h_d)}  \sum_{\mathbf{k} \in \Z^d} \left|\varepsilon_{\mathbf{j},\mathbf{k}}\right|  \sup_{t \in [0,T]} \left \{\left| \mathcal{A}_{\mathbf{j},\mathbf{k}}(t)\right| \right\} \label{def:lnN3}
\end{align}
is bounded from above by $C  N^\frac{d}{2} 2^{N \left(\sum_{(\ell \, : \, v_\ell =0)} (1-h_\ell)\right)-2^{Nb} \left( \sum_{(\ell' \, : \, v_{\ell'} =1)}(1-h_{\ell'})\right)}.$
\end{Lemma}

\begin{proof}
Let ux fix $N \in \N$ and $ \mathbf{j} \in \gimel_{v,N}^3$. Using the definition \eqref{def:AJK}, the fast decay property \eqref{maj:psih}, the bound \eqref{eqn:bound:rv}, the inequality \eqref{log:ineg1}, the triangular inequality and the inequalities \eqref{maj:somme1} and \eqref{log:ineg2}, we get
\begin{align}
&\sum_{\mathbf{k} \in \Z^d}  \left|\varepsilon_{\mathbf{j},\mathbf{k}}\right|  \sup_{t \in [0,T]} \left \{\left| \mathcal{A}_{\mathbf{j},\mathbf{k}}(t)\right| \right\} \nonumber \\
& \leq C_0 \int_0^T \prod_{\ell \, : \, v_\ell =0} \left( \frac{\sqrt{\log(3+j_\ell+|k_\ell|)}}{(3+|2^{j_\ell}s-k_\ell|)^2} \right) \times \ldots \nonumber \\
& \qquad \ldots \times\prod_{\ell' \, : \, v_{\ell'} =1} \left( \frac{\sqrt{\log(3+|j_{\ell'}|+|k_{\ell'}|)}}{(3+|2^{j_{\ell'}}s-k_{\ell'}|)^2} \right) \, ds \nonumber \\
& \leq C_0 \int_0^T \prod_{\ell \, : \, v_\ell =0} \left( \sqrt{\log(3+j_\ell+2^{j_\ell}T)} \right) \times \ldots \nonumber \\
& \qquad \ldots \times \prod_{\ell' \, : \, v_{\ell'} =1} \left( \frac{\sqrt{\log(3+|j_{\ell'}|)}\sqrt{\log(3+|k_{\ell'}|)}}{(3+|k_{\ell'}|)^2} \right) \, ds \nonumber \\
& \leq C_1 N^{ \# \{\ell \, : \, v_\ell=0\} /2}  \prod_{\ell' \, : \, v_{\ell'} =1} \left(\sqrt{\log(3+|j_{\ell'}|)} \right), \label{eqn:boundl3}
\end{align}
where $C_0$ and $C_1$ are positive almost surely finite random variables not depending on $\mathbf{j}$ and $N$. It follows from \eqref{eqn:boundl3} that
\begin{align*}
\mathcal{L}^3_{v,N}  &\leq C_1 N^{ \# \{\ell \, : \, v_\ell=0\} /2}  \prod_{\ell \, : \, v_\ell=0} \left( \sum_{j_\ell =0}^{N} 2^{j_\ell (1-h_\ell)}\right) \times \ldots \nonumber \\
& \qquad \ldots \times\prod_{\ell' \, : \, v_{\ell'} =1} \left(\sum_{j_{\ell'}=- \infty}^{\lfloor-2^{Nb}\rfloor} 2^{j_{\ell'}(1-h_{\ell'})} \sqrt{\log(3+|j_{\ell'}|)} \right) \\
& \leq C_2 N^\frac{d}{2} 2^{N \left(\sum_{(\ell \, : \, v_\ell =0)} (1-h_\ell)\right)-2^{Nb} \left( \sum_{(\ell' \, : \, v_{\ell'} =1)}(1-h_{\ell'})\right)},
\end{align*}
where $C_2$ is a positive almost surely finite random variable not depending on $N$.
\end{proof}

\begin{Def}
Let $b' >0$ be a fixed real number. For all boolean vector $v =(v_l)_{\ell \in [\![1,d]\!]}\in \{0,1\}^d$ and $N \in \N$, we define the subset of $\Z^d$
\[ \gimel_{v,N}^4 := \big\{ j \in -\N^d \, : \, \forall \,\ell \in [\![1,d]\!],\, j_\ell  \leq -2^{N b' }\text{ if } v_\ell=1  \big\}. \]
\end{Def}

\begin{Lemma}\label{lemmepour24:4}
Let $T >2$ and $b'>0$ be two fixed real numbers. There exits a positive almost surely finite random variable $C$ such that, for all $v \in \{0,1\}^d$ and $N \in \N$, on $\Omega^*$, the random variable
\begin{align}
\mathcal{L}^4_{v,N} :=  \sum_{\mathbf{j} \in \gimel_{v,N}^4} 2^{ j_1 (1-h_1)+\cdots + j_d (1-h_d)}  \sum_{\mathbf{k} \in \Z^d} \left|\varepsilon_{\mathbf{j},\mathbf{k}}\right|  \sup_{t \in [0,T]} \left \{\left| \mathcal{A}_{\mathbf{j},\mathbf{k}}(t)\right| \right\} \label{def:lN4}
\end{align}
is bounded from above by $C  N^\frac{d}{2} 2^{-2^{Nb'} \left(\sum_{(\ell \, : \, v_\ell =1)} (1-h_\ell)\right)}$.
\end{Lemma}

\begin{proof}
Let us fix $N \in \N$ and $\mathbf{j} \in \gimel_{v,N}^4$. Using the definition \eqref{def:AJK}, the fast decay property \eqref{maj:psih}, the bound \eqref{eqn:bound:rv}, the inequality \eqref{log:ineg1} and the triangular inequality, we get
\begin{align}
& \sum_{\mathbf{k} \in \Z^d} \left|\varepsilon_{\mathbf{j},\mathbf{k}}\right|  \sup_{t \in [0,T]} \left \{\left| \mathcal{A}_{\mathbf{j},\mathbf{k}}(t)\right| \right\} 
 \leq C_0 \int_0^T \prod_{\ell=1}^d \left( \sum_{k_\ell \in \Z} \frac{\sqrt{\log(3+|j_\ell|+|k_\ell|)}}{(3+|2^{j_\ell}s-k_\ell|)^2} \right) \, ds \nonumber \\
& \leq C_0 \int_0^T \prod_{\ell=1}^d \left( \sum_{k_\ell \in \Z} \frac{\sqrt{\log(3+|j_\ell|)}\sqrt{\log(3+|k_\ell|)}}{(3+|k_\ell|)^2} \right) \, ds \nonumber\\
& \leq C_1  \prod_{\ell=1}^d \sqrt{\log(3+|j_\ell|)}\label{eqn:boundl4},
\end{align}
where $C_0$ and $C_1$ are positive almost surely finite random variables not depending on $\mathbf{j}$ and $N$. It follows from \eqref{eqn:boundl4} that
\begin{align*}
\mathcal{L}^4_{v,N} & \leq C_1 \prod_{\ell \, : \, v_\ell =1} \left( \sum_{j_\ell=-\infty}^{\lfloor -2^{N b'}\rfloor} 2^{j_\ell (1-h_\ell)}\sqrt{\log(3+|j_\ell|)}\right) \times \ldots \nonumber \\
& \qquad \ldots \times\prod_{\ell' \, : \, v_{\ell'} =0} \left( \sum_{j_{\ell'}=-\infty}^{-1} 2^{j_{\ell'} (1-h_{\ell'})}\sqrt{\log(3+|j_{\ell'}|)}\right) \\
& \leq C_2 N^\frac{d}{2} 2^{-2^{Nb'} \left(\sum_{(\ell \, : \, v_\ell =1)} (1-h_\ell)\right)},
\end{align*}
where $C_2$ is a positive almost surely finite random variable not depending on $N$.
\end{proof}

We are now in position to complete the proof of Theorem \ref{thm:main2}.

\begin{proof}[End of the Proof of Theorem \ref{thm:main2}]
The interval $[0,T]$ in statement of Theorem \ref{thm:main2} is denoted by $I$ in this proof. Let $\widetilde{\Omega}$ be the same event of probability $1$ as in the proof of Theorem \ref{thm:main1}; recall that it is defined as: $\widetilde{\Omega}= \Omega^* \cap \Omega^{**}$, where $\Omega^*$ and $\Omega^{**}$ are as in Lemmata \ref{lem:bound-gjk} and  \ref{lem:eventomega**}. Next, observe that for proving Theorem \ref{thm:main2} it is enough to show that there exists a positive almost surely finite random variable $C$ such that, on $\widetilde{\Omega}$, we have, for all $N,P \in \N$,  
\begin{equation}\label{eqn:main2step}
\big\| \widetilde{X}^{(d)}_{\mathbf{h},N+P}-\widetilde{X}^{(d)}_{\mathbf{h},N}\big\|_{I,\infty} \leq C N^{\frac{d}{2}} 2^{-N(  h_1+ \cdots + h_d -d + 1/2)},
\end{equation}
where, for all fixed $\omega\in\Omega$, the continuous function $\widetilde{X}^{(d)}_{\mathbf{h},N}(\cdot,\omega)$ is defined through \eqref{anteq1:thm:main2}.
Indeed, assuming that \eqref{eqn:main2step} is true, then it turns out that, for each fixed $\omega \in \widetilde{\Omega}$, the sequence of functions $\big(\widetilde{X}^{(d)}_{\mathbf{h},N}(\cdot,\omega)\big)_{N \in \N}$ is a Cauchy sequence in the Banach space of the real-valued continuous functions over $I$, equipped with its natural norm $\|\cdot \|_{I,\infty}$. Therefore, it converges, for this norm, to a continuous function over $I$ denoted by $\widetilde{X}^{(d)}_{\mathbf{h}}(\cdot,\omega)$. On the other hand, when $\omega\in\Omega\setminus\widetilde{\Omega}$ we set $\widetilde{X}^{(d)}_{\mathbf{h}}(t,\omega)=0$, for all $t\in I$. Next, observe that, in view of the previous definition of the stochastic process $\{ \widetilde{X}^{(d)}_{\mathbf{h}}(t)\}_{t \in I}$ we have, for all $t \in I$, almost surely,
\begin{equation}
\label{anteq2:thm:main2}
X^{(d)}_{\mathbf{h}}(t)=\widetilde{X}^{(d)}_{\mathbf{h}}(t),
\end{equation} 
since we know from \eqref{anteq1:thm:main2} and \eqref{eqn:fullserie}, that, for each fixed $t\in\R_+$ (and in particular for $t\in I$), the sequence of random variables $(\widetilde{X}^{(d)}_{\mathbf{h},N}(t,\omega))_{N \in \N}$ converges to $X^{(d)}_{\mathbf{h}}(t)$ in $L^2(\Omega)$. Next, using the fact that the two stochastic processes $\{X^{(d)}_{\mathbf{h}}(t) \}_{t \in I}$ and $ \{ \widetilde{X}^{(d)}_{\mathbf{h}}(t)\}_{t \in I}$ have continuous paths, we can derive from the almost sure equality \eqref{anteq2:thm:main2} that these two processes are indistinguishable. Thus, letting $P$ in \eqref{eqn:main2step} tends to $+ \infty$ , we obtain \eqref{eqn:toshowinthm2}.

It remains us to show that \eqref{eqn:main2step} holds. Using the inclusion
\[ D_{j}^1(t) \subseteq \{k \in \Z \, : \, |k| \leq 2^{N+P+1}T \},\quad\mbox{for all $t \in I:=[0,T]$ and $N \leq j \leq N+P$,}\]
and the triangle inequality, we get that
\begin{align*}
  & \|\widetilde{X}^{(d)}_{\mathbf{h},N+P}-\widetilde{X}^{(d)}_{\mathbf{h},N}\|_{I,\infty} \nonumber \\
  & \leq  \sum_{n=1}^{d} \left( \mathcal{M}_{n,N}+ \sum_{m=0}^3 \mathcal{H}_{n,N}^m + \sum_{m=1}^2 \mathcal{L}_{n,N}^m\right) + \sum_{v \in \mathcal{B}} \mathcal{L}_{v,N}^3+\sum_{v \in \{0,1\}^d} \mathcal{L}_{v,N}^4 . 
\end{align*}
Then the fact that \eqref{eqn:main2step} is satisfied follows from Lemmata \ref{lemmepour22:1}, \ref{lemmepour22:2}, \ref{lemmepour22:3}, \ref{lemmepour22:4},\ref{lemmepour22:5}, \ref{lemmepour24:1}, \ref{lemmepour24:2}, \ref{lemmepour24:3} and \ref{lemmepour24:4}.
\end{proof}

\bigskip

\noindent \textbf{Acknowledgement.} The three authors of this work are members of the GDR 3475 (Analyse Multifractale et Autosimilarit\'e) which partially supports them. Antoine Ayache is also partially supported by the Labex CEMPI (ANR-11-LABX-0007-01) and the Australian Research Council's Discovery Projects funding scheme (project number  DP220101680). Laurent Loosveldt is supported by the FNR OPEN grant APOGEe at University of Luxembourg. The first results of this paper were formulated during a visit of Antoine Ayache at the University of Luxembourg, also funded by the FNR OPEN grant APOGEe.

\appendix

\section{Some facts concerning multiple Wiener integrals} \label{sec:ape;profito}

In this section, mainly we give the proof of the crucial equality \eqref{eqn:itoprod}. This proof relies on some fundamental facts concerning multiple Wiener integrals. We refer to the two books \cite{MR2962301,MR2200233} for detailed presentations of such stochastic integrals and many other related topics (Wiener chaoses, Malliavin calculus, and so on). We recall that a function $f \in L^2(\R^n)$ is said to be symmetric if, for all $\sigma \in \mathfrak{S}_n$ (the set of permutations of $[\![1,n]\!]=\{1,\ldots,n\}$) and for Lebesgue almost every $(x_1,\ldots,x_n) \in \R^n$, one has  $f(x_{\sigma(1)},\ldots,x_{\sigma(n)})=f(x_1,\ldots,x_n)$. In other words, $f \in L^2(\R^n)$ is symmetric if and only if it is almost everywhere equal to its canonical symmetrization $\widetilde{f}$ defined, for all $(x_1,\ldots,x_n)\in\R^n$, as:
\begin{equation}
\label{eq:def-symmet}
\widetilde{f}(x_1,\ldots,x_n):=\frac{1}{n!} \sum_{\sigma \in \mathfrak{S}_n} f(x_{\sigma(1)},\ldots,x_{\sigma(n)}).
\end{equation}
We point out that a very fundamental property of multiple integrals is that
\begin{equation}
\label{eq:fftilde}
I_n(f) =I_n(\widetilde{f}), \quad\mbox{for all $f \in L^2(\R^n)$.} 
\end{equation}

For proving the equality \eqref{eqn:itoprod}, we will make use of the so-called product formula for multiple Wiener integrals \cite[Proposition 1.1.3]{MR2200233}. In order to give this important formula, first we need the following definition: let $m$ and $n$ be two arbitrary positive integers, if $f \in L^2(\R^m)$ and $g \in L^2(\R^n)$ are symmetric functions and $r\in [\![0,m \wedge n]\!]$, the contraction $f \otimes_r g$ is the $L^2(\R^{m+n-2r})$ function defined, for all $(x_1, \ldots,x_{m+n-2r}) \in \R^{m+n-2r}$, through the Lebesgue integral
\begin{align*}
& (f \otimes_r g)(x_1, \ldots,x_{m+n-2r}) \\
&:= \int_{\R^r} f(x_1,\ldots , x_{m-r},s_1,\ldots,s_r) g(x_{m-r+1},\ldots, x_{m+n-2r},s_1,\ldots,s_r) \, ds_1 \ldots ds_r\,,
\end{align*}
with the convention that $f \otimes_0 g:= f \otimes g$, which means that $f \otimes_0 g$ is the usual tensor product of $f$ and $g$; also notice that when $m=n$, then $f \otimes_n g$ is identified with the Lebesgue integral $\int_{\R^n} f g$. Using, the previous definition, one can write the product formula for multiple Wiener integrals in the following way: for each positive integers $m$ and $n$, and for every symmetric functions $f \in L^2(\R^m)$ and $g \in L^2(\R^n)$, one has 
\begin{equation}\label{eqn:productformula}
I_m(f)I_n(g) = \sum_{r=0}^{m \wedge n} r! \binom{m}{r} \binom{n}{r} I_{m+n-2r}(f \otimes_r g),
\end{equation}
where, for $p=m$ or $p=n$, the quantity $\binom{p}{r}$ is the usual binomial coefficient 
\[
\binom{p}{r}:=\frac{p!}{r! (p-r)!}\,.
\]

For proving the equality \eqref{eqn:itoprod}, we will also make use of the following fundamental result, which, for instance corresponds to \cite[Theorem 2.7.7]{MR2962301}.  

\begin{Thm} \label{thm:inthermite}
Let $f \in L^2(\R)$ be such that $\|f\|_{L^2(\R)}=1$. For all positive integer $n$, let $H_n$ the Hermite polynomial of degree $n$. Then, one has
\[ H_n \left( I_1(f) \right) = I_n(f^{\otimes_n}).\]
\end{Thm}

%\begin{Thm}
%If $\varphi_1,\ldots, \varphi_p$ are orthonormal elements in $L^2(\R)$ and $n_1, \ldots, n_p \in \N$, then
%\[
%I_{n_1 + \cdots +n_p} \left( \varphi_{1}^{\otimes_{n_1}} \otimes \cdots \otimes \varphi_{p}^{\otimes_{n_p}} \right) = \prod_{\ell=1}^{p} H_{n_{\ell}} \left( I_1 \left(\varphi_\ell\right) \right) 
%\]
%\end{Thm}

We are now in position to prove the equality \eqref{eqn:itoprod}

\begin{proof}[Proof of the equality \eqref{eqn:itoprod}]
It follows from Theorem \ref{thm:inthermite} that
\[ \prod_{\ell=1}^p H_{n_\ell} \big(I_1 ( \varphi_\ell)\big) = \prod_{\ell=1}^p I_{n_\ell} \left(\varphi_\ell^{\otimes_{n_\ell}} \right),  \]
and thus, it remains to show
\begin{equation} \label{eqn:pro:rv:prod}
\prod_{\ell=1}^p I_{n_\ell} \left( \varphi_\ell^{\otimes_{n_\ell}} \right) = I_{n_1+\cdots+n_p} \left( \bigotimes_{\ell=1}^p \varphi_\ell^{\otimes_{n_\ell}} \right).
\end{equation}
We proceed by induction on the positive integer $p$. It is clear that  \eqref{eqn:pro:rv:prod} is satisfied when $p=1$. So from now on, we assume that $p\ge 2$ and that 
\[ \prod_{\ell=1}^{p-1} I_{n_\ell} \left( \varphi_\ell^{\otimes_{n_\ell}} \right) = I_{n_1+ \cdots+n_{p-1}} \left( \bigotimes_{\ell=1}^{p-1} \varphi_\ell^{\otimes_{n_\ell}} \right).\]
Then, setting $n=n_1+\cdots+n_{p-1}$ and $d =n_1+ \cdots +n_p=n+n_p$, we can derive from the product formula \eqref{eqn:productformula} that
\begin{align}
\prod_{\ell=1}^p I_{n_\ell} \left( \varphi_\ell^{\otimes_{n_\ell}} \right) &= I_{n} \left( \bigotimes_{\ell=1}^{p-1} \varphi_\ell^{\otimes_{n_\ell}} \right) I_{n_p} \left( \varphi_p \right) \nonumber\\
&= \sum_{r=0}^{n \wedge n_p} r! \binom{n}{r} \binom{n_p}{r} I_{n+n_p-2r} \left(\widetilde{\bigotimes_{\ell=1}^{p-1} \varphi_\ell^{\otimes_{n_\ell}}} \otimes_r \varphi_p^{\otimes_{n_p}} \right)\nonumber\\
&=I_{d} \left(\widetilde{\bigotimes_{\ell=1}^{p-1} \varphi_\ell^{\otimes_{n_\ell}}} \otimes \varphi_p^{\otimes_{n_p}} \right)=I_{d} \left(\oversortoftilde{\widetilde{\bigotimes_{\ell=1}^{p-1}  \varphi_\ell^{\otimes_{n_\ell}}} \otimes \varphi_p^{\otimes_{n_p}}} \right). \label{prop:rv:prodher:eq1}
\end{align}
Notice that the third equality in \eqref{prop:rv:prodher:eq1} results from the equality 
\[ 
\widetilde{\bigotimes_{\ell=1}^{p-1} \varphi_\ell^{\otimes_{n_\ell}}} \otimes_r \varphi_p^{\otimes_{n_p}}=0,\quad\mbox{for all $r\in [\![1, n]\!]$,}
\]
which is a consequence of the orthonormality of the system $(\varphi_\ell)_{\ell=1}^p$. Also notice that the last equality in \eqref{prop:rv:prodher:eq1} results from
 \eqref{eq:fftilde}. Next observe that, in view of \eqref{prop:rv:prodher:eq1} and \eqref{eq:fftilde}, in order to show that \eqref{eqn:pro:rv:prod} holds, it remains us to prove that 
\begin{equation}
\label{prop:rv:prodher:eq2}
\oversortoftilde{\widetilde{\bigotimes_{\ell=1}^{p-1}  \varphi_\ell^{\otimes_{n_\ell}}} \otimes \varphi_p^{\otimes_{n_p}}}=\oversortoftilde{\bigotimes_{\ell=1}^{d} \varphi_\ell^{\otimes_{n_\ell}}}.
\end{equation}  
Notice that any arbitrary permutation $\sigma\in\mathfrak{S}_n$ can be extended in a natural way into a permutation $\check{\sigma}\in\mathfrak{S}_d$ defined, for all $i\in\{1,\ldots, n\}$, as $\check{\sigma}(i)=\sigma (i)$, and for, each $i\in\{n+1,\ldots , d\}$, as
$\check{\sigma}(i)=i$. Thus, using \eqref{eq:def-symmet}, the latter notation and the fact that the composition map $\nu\mapsto \nu o\,\check{\sigma}$ is bijection from $\mathfrak{S}_d$ to itself, one gets, for all $(x_1, \ldots,x_d)\in\R^d$, that
\begin{align*}
&\left(\oversortoftilde{\widetilde{\bigotimes_{\ell=1}^{p-1}  \varphi_\ell^{\otimes_{n_\ell}}} \otimes \varphi_p^{\otimes_{n_p}}} \,\right) (x_1, \ldots,x_d) \\  \quad &= \frac{1}{d!} \frac{1}{n!} \sum_{\sigma \in \mathfrak{S}_n} \sum_{\nu \in \mathfrak{S}_d} \bigotimes_{\ell=1}^{p-1}  \varphi_\ell^{\otimes_{n_\ell}} (x_{\nu(\sigma(1))}, \ldots,x_{\nu(\sigma(n))})  \otimes \varphi_p^{\otimes_{n_p}}(x_{\nu(n+1)},\ldots,x_{\nu(d)})\\
\quad &= \frac{1}{d!} \frac{1}{n!} \sum_{\sigma \in \mathfrak{S}_n} \sum_{\nu \in \mathfrak{S}_d} \bigotimes_{\ell=1}^{p} \varphi_\ell^{\otimes_{n_\ell}} (x_{\nu o\,\check{\sigma} (1)}, \ldots,x_{\nu o\, \check{\sigma}(d)})  \\
\quad &= \frac{1}{d!} \frac{1}{n!} \sum_{\sigma \in \mathfrak{S}_n} \sum_{\nu' \in \mathfrak{S}_d} \bigotimes_{\ell=1}^{p} \varphi_\ell^{\otimes_{n_\ell}} (x_{\nu'(1)}, \ldots,x_{\nu'(d)})  \\
\quad &= \frac{1}{d!} \sum_{\nu' \in \mathfrak{S}_d} \bigotimes_{\ell=1}^{p}  \varphi_\ell^{\otimes_{n_\ell}} (x_{\nu'(1)}, \ldots,x_{\nu'(d)})  \\
\quad &= \left(\oversortoftilde{\bigotimes_{\ell=1}^{p}  \varphi_\ell^{\otimes_{n_\ell}} } \, \right) (x_1, \ldots,x_d).
\end{align*}
\end{proof}

\section{Some useful lemmata}
\label{sec:use-lemmata}

The proofs of the following lemmata, which are extensively used in our articles, can be found in \cite{MR4110623}.

\begin{Lemma}\label{lem:log:ineg1}
For all $(x,y)\in \R_+^2$, we have 
\begin{equation} \label{log:ineg1}
\log(3+x+y) \le \log(3+x) \log(3+y).
\end{equation}
Moreover, for each fixed positive real number $T$, there exists a constant $c>0$ such that, for every $x\in\R_+$, we have
\begin{equation}\label{log:ineg2}
\log(3+x+2^xT) \le c(1+x).
\end{equation}
\end{Lemma} 

\begin{Lemma}\label{lem:maj:somme1}
For each fixed real number $L>1$, there exists a constant $c>0$ such that, for all $j\in\Z$ and for all $s\in\R$, we have
\begin{equation}
\sum_{k\in\Z} \frac{\sqrt{\log(3+|j|+|k|)}}{\big (3+|2^js-k|\big)^L} \le c\sqrt{\log\big (3+|j|+2^j|s|\big)}. \label{maj:somme1}
\end{equation}
\end{Lemma}

\begin{Lemma}\label{lem:maj:somme2}
For each fixed real number $L>1$, there exists a constant $c>0$ such that, for all $t\in\R_+$, for all $s\in [0,t]$ and for all $j\in\N$, we have 
\begin{equation}\label{maj:somme2}
\sum_{k\in D_j^3(t)} \frac{\sqrt{\log(3+|j|+|k|)}}{(3+|2^js-k|)^L} \le 
c (1+j) 2^{-j(L-1)(1-a)} \sqrt{\log(3+t)},
\end{equation}
where $D_j^3(t)$ is the infinite subset of $\Z$ defined through \eqref{def:dj3}.
\end{Lemma}

\bibliography{biblio}{}
\bibliographystyle{plain}

\end{document}